\documentclass[11pt]{amsart}
\usepackage[latin1]{inputenc} 
\usepackage{listings, tcolorbox}

\usepackage{amssymb, amsmath, amsthm}
\usepackage[colorlinks=true,linkcolor=blue,citecolor=red]{hyperref}
\usepackage{color}
\usepackage{graphics}
\usepackage{subfigure}
\usepackage{graphicx}  
\usepackage{float}  
\usepackage{epsfig}
\usepackage{amssymb}
\usepackage{epstopdf}
\usepackage{enumerate}
\usepackage{appendix}
\usepackage{amscd}
\usepackage{longtable,booktabs}
\usepackage{mathrsfs}
\usepackage{tikz}
\usepackage{bm}
\usepackage[T1]{fontenc}

\numberwithin{equation}{section}
\newtheorem{theorem}{Theorem}[section]

\newtheorem{lemma}[theorem]{Lemma}
\newtheorem{prop}[theorem]{Proposition}

\newtheorem{thm}[theorem]{Theorem}
\newtheorem{coro}[theorem]{Corollary}
\theoremstyle{definition}
\newtheorem{definition}[theorem]{Definition}
\newtheorem{re}[theorem]{Remark}
\newtheorem{problem}{Problem}
\def\spn{\mathop{\mathrm{span}}}

\newenvironment{Acknowledgment}
{\begin{trivlist}\item[]\textbf{Acknowledgments }}{\end{trivlist}}

\newcommand{\la}{\lambda}

\newcommand{\be}{\beta}

\newcommand{\ka}{\kappa}
\newcommand{\R}{\mathbb{R}}
\newcommand{\C}{\mathbb{C}}

\newcommand{\Z}{\mathbb{Z}}

\newcommand{\rmd}{\textrm{d}}
\newcommand{\qtq}[1]{\enspace\text{#1}\enspace}

\begin{document}
\title[Soliton resolution conjecture for the Benjamin-Ono equation]
{Soliton resolution conjecture for the Benjamin-Ono equation in the sharp well-posedness space}
\date{}
\author[Pan]{Hong-Yu Pan}
\address{\textnormal{Hong-Yu Pan} \newline \indent School of Mathematics, China University of Mining and Technology, Xuzhou 221116, China}
\email{hypan@cumt.edu.cn}
\author[Tian]{Shou-Fu Tian$^{*,\dag}$}
\address{\textnormal{Shou-Fu Tian} (Corresponding author) \newline \indent School of Mathematics, China University of Mining and Technology, Xuzhou 221116, China} 
\email{sftian@cumt.edu.cn}
\thanks{$^{*}$Corresponding authors(sftian@cumt.edu.cn, shoufu2006@126.com (S.F. Tian)).\\
	\hspace*{3ex}$^\dag$These authors contributed equally.}
	
\begin{abstract}
We prove the soliton resolution conjecture for the Benjamin-Ono equation in the sharp well-posedness space $H^s\left(\mathbb{R}\right)$ for $s>-\frac12$. This extends soliton resolution to the sharp regularity, whereas previous results required weighted $H^{1,1}\left(\mathbb{R}\right)$ initial data. We first prove the soliton resolution conjecture in the sharp well-posedness space for the integrable equation. In particular, our work solves three important open problems, including the problem of embedded singular continuous spectrum, the extension of G\'erard's explicit formula, and the inverse spectral problem for the Lax operator.

A central ingredient of our work is a new modular operator-theoretic framework combining the closed quadratic-form theory of rough Lax operators, weighted limiting absorption, and distorted Fourier analysis, which yields a unified treatment of distributional spectral theory, continuous-spectrum scattering, and radiative asymptotics at the sharp well-posedness threshold. Another important ingredient is a Fourier-side regularity and stability result for the eigenfunctions of the rough Lax operator, established through a zero-frequency identity, resolvent commutator estimates, and an isolating-contour argument. This yields well-defined soliton centers that are stable under rough approximation. Our framework provides a broadly applicable framework for proving soliton resolution conjecture for integrable dispersive equations in low-regularity spaces.
\end{abstract}
	\maketitle
	\tableofcontents
\section{Introduction}
\label{sec:introduction}
In this work, we consider the Cauchy problem for the Benjamin-Ono (BO) equation
\begin{equation}\label{eq:BO}
	\left\{
	\begin{array}{l}
		\partial_t u - \partial_x \lvert D_x \rvert u + \partial_x (u^2) = 0, \\
		u(0,x) = u_0(x),
	\end{array}
	\right.
	\enspace (t,x) \in \mathbb{R}^2,
\end{equation}
where the real-valued function $u(t,x)$ denotes the potential, $$\left| D_x \right|u\left( x \right) :=\frac{1}{\pi}\mathrm{P}.\mathrm{V}.\int_{\mathbb{R}}^{}{\frac{u^{\prime}\left( y \right)}{x-y}dy},$$and
\begin{equation*}
	\mathrm{P.V.} \int_{\mathbb{R}} \frac{f(y)}{x-y} dy := \lim_{\epsilon \to 0^+} \int_{|x-y| > \epsilon} \frac{f(y)}{x-y} dy.
\end{equation*} 
Benjamin \cite{Benjamin} (see also Davis and Acrivos \cite{DavisAcrivos}) derived the BO equation, which describes the weakly nonlinear propagation of internal gravity waves along the interface of a stratified two-layer fluid of great total depth. The potential function $u$ measures the interface's elevation relative to equilibrium. For a rigorous and comprehensive derivation of the BO equation, we refer readers to the monograph by Saut and Klein \cite{KS21}.

Investigating the functional spaces in which the BO equation admits a unique solution has been an important problem ever since its derivation. Saut was the first to prove the global well-posedness for the BO equation in the space $H^s(\mathbb{R})$ with $s \geqslant 2$ \cite{Saut}. Following this seminal work, numerous studies investigated the Cauchy problem for the BO equation in lower regularity spaces \cite{Ger26, IfrimTataru, IonescuKenig, KS21, Molinet, MP12, TaoBO}. Most notably, G\'erard, Kappeler, and Topalov \cite{GKT} achieved a milestone result for the BO equation on the torus $\mathbb{T}$. Via the construction of Birkhoff coordinates, they proved sharp global well-posedness for $s > -1/2$ and ill-posedness at $s = -1/2$, which additionally revealed remarkable qualitative dynamics. Based on the groundbreaking work \cite{GKT}, as well as the explicit formula for the BO equation derived by G\'erard \cite{Gerard}, Killip, Laurens, and Vi\c{s}an \cite{KLV} established global well-posedness on both the real line and the torus for initial data in $H^s$ with $ s > -1/2$.

Following the achievement of well-posedness results for the BO equation, the next natural and  important topic is to explore the long-time asymptotics of its solutions. Employing the celebrated explicit formula proposed by G\'erard \cite{Gerard}, G\'erard and his collaborators systematically analyzed the long-time dynamics of the BO equation. For example, Blackstone, Gassot, G\'erard, and Miller reduced the operator resolvent within the formula into finite-dimensional determinants and addressed the long-time dynamics for the BO equation with rational initial data in detail \cite{BGGM}. Moreover, for a class of initial data for which the number of solitons is infinite, Gassot and G\'erard investigated the long-time asymptotic behavior of the BO equation on the real line \cite{GG}. \textit{Notably, in the investigation of long-time dynamics, the soliton resolution conjecture is a central and fundamental open problem.} 

The investigation of the soliton resolution conjecture originated from the long-time asymptotic phenomenon of numerical solutions observed by Zabusky and Kruskal in the Korteweg-de Vries (KdV) equation \cite{ZabuskyKruskal}. This phenomenon profoundly reveals the dynamical behavior of nonlinear wave equations, and since then has continuously attracted widespread attention from numerous prominent scholars, including Tao \cite{T06} and Kenig \cite{DKM23}. In 2008, Tao \cite{TaoSolitons} pointed out the possibility of establishing the soliton resolution conjecture and explicitly stated:

\vspace{\baselineskip}
``\textit{More generally, if one starts with arbitrary (but smooth and decaying) initial data, what usually happens (numerically, at least) with evolutions of equations such as $u_t+u_{xxx}+6uu_x=0$ is that some non-linear (and chaotic-seeming) behaviour happens for a while, but eventually most of the solution radiates away to infinity and a finite number of solitons emerge, moving away from each other at different rates.}''
\vspace{\baselineskip}

\textit{It is widely believed that the \textbf{soliton resolution conjecture} asserts that for sufficiently large times, the solution of a nonlinear dispersive equation can be decomposed into a finite number of solitons and a radiation term.} For the vast majority of dispersive equations, this conjecture has long been considered an open problem. In 1993, Deift and Zhou investigated the asymptotic properties of solutions to the modified Korteweg-de Vries equation with decaying initial data \cite{DZ93}. Subsequently, based on the seminal work of Deift and Zhou, the soliton resolution conjecture was rigorously established for a broad class of classical integrable systems, including the celebrated KdV equation and nonlinear Schr\"{o}dinger equation \cite{ BJM18, CLW23, CL21, CJ16, DP11, DVZ94, LTY22a, LTY22b, TT26, YTL26}. The second author Tian and his collaborators have also established results concerning the soliton resolution conjecture and long-time asymptotics for the Wadati-Konno-Ichikawa equation \cite{LTY22a, LTY22b}, the spin-1 Gross-Pitaevskii equation \cite{TT26}, and the modified Camassa-Holm equation \cite{YTL26}. 

However, the soliton resolution conjecture for some integrable equations with a nonlocal Hilbert transform remains an open problem. Recently, Gassot, G\'erard, and Miller \cite{GGM} proved the following beautiful result:
\begin{prop}[Gassot, G\'erard, and Miller \cite{GGM}]\label{GGM26}
	Let \({u}_{0} \in  {H}^{1}\left( \mathbb{R}\right)\) be a real-valued function such that \(x{u}_{0} \in  {H}^{1}\left( \mathbb{R}\right)\) and
	\begin{equation}\label{tiaojian}
		x^2u_0\left( x \right) =c_0+v_0\left( x \right) ,\enspace c_0\in \mathbb{R} ,v_0\left(x \right) \in L^2(\mathbb{R} ).
	\end{equation}
	Then there exist a nonnegative integer \(N\) , elements \({p}_{1},\ldots ,{p}_{N}\) of \({\mathbb{C}}_{ + }\) , with \(\operatorname{Im}\left( {p}_{1}\right)  < \cdots  < \operatorname{Im}{p}_{N}\) , and real-valued functions \({u}_{\infty }^{ \pm  }\) in \({H}^{1}\left( \mathbb{R}\right)\) such that
	$$
	\left\| r^\pm\left( t,\cdot \right) \right\| _{H^1\left( \mathbb{R} \right)}:=\left\| u\left( t,\cdot \right) -\sum_{j=1}^N{R_{p_j}\left( \cdot -\frac{1}{\mathrm{Im}p_j}t \right)}-\mathrm{e}^{t\partial _x\left| D \right|}u_{\infty}^{\pm} \right\| _{H^1\left( \mathbb{R} \right)}\rightarrow 0,\enspace t\rightarrow \pm\infty,
	$$
	where 
	\begin{equation}\label{rp}
		R_p\left( x-\frac{1}{\mathrm{Im}p}t \right) =\frac{2\mathrm{Im}p}{\left| x-\frac{1}{\mathrm{Im}p}t+p \right|^2}.
	\end{equation}
\end{prop}

The above results of Gassot, G\'erard, and Miller demonstrate that, in the sense of the $H^1$-norm, the solution for the BO equation asymptotically decomposes into a sum of solitons for the form $R_p(x-\frac{1}{\mathrm{Im}p}t)$ and a dispersive term $\mathrm{e}^{t\partial_x|D|}u_{\infty}^{\pm}$. 

Proposition~\ref{GGM26} gives a global $H^1(\mathbb R)$
decomposition into finitely many solitons and a free radiation term
under the stated regularity and spatial-decay assumptions.
A complementary result of Badreddine, Killip, and Vi\c{s}an
\cite[Theorem~1.5]{BKVasymp} describes the long-time asymptotics
for arbitrary $L^2$ data in moving spatial windows, without
assuming any spatial weight.

To state this result, let $\Pi$ denote the orthogonal projection onto
the Hardy space
\[
L_+^2(\mathbb R)
:=
\left\{
f\in L^2(\mathbb R):
\operatorname{supp}\widehat f\subset[0,\infty)
\right\},
\]
equipped with the inner product
$\langle f,g\rangle=\int_{\mathbb R}\overline{f(x)}g(x)\,\mathrm{d}x$.
For real-valued $u_0\in L^2(\mathbb R)$, the BO Lax operator is
\[
L_{u_0}=D-T_{u_0},
\enspace
D=-i\partial_x,
\enspace
T_{u_0}f=\Pi(u_0f).
\]
It is self-adjoint on $L_+^2(\mathbb R)$ with domain
$H_+^1(\mathbb R):=H^1(\mathbb R)\cap L_+^2(\mathbb R)$;
see \cite[Theorem~3.3]{BKVasymp}.
Following \cite[Section~2.1]{GGM}, we denote by $X^*$ the adjoint
of multiplication by $x$ on the Hardy space, characterized by
\[
\operatorname{Dom}(X^*)
=
\left\{
f\in L_+^2(\mathbb R):
\widehat f|_{(0,\infty)}\in H^1(0,\infty)
\right\},
\enspace
\widehat{X^*f}(\xi)=i\partial_\xi\widehat f(\xi),
\quad \xi>0.
\]
This is the operator denoted by $X$ in \cite{BKVasymp}. For every Borel set $B\subset\mathbb R$, define
\begin{equation}\label{eq:spectral-measure-definition}
	E_{u_0}(B):=\mathbf1_B(L_{u_0}),
	\enspace
	\mu_f(B):=\langle f,E_{u_0}(B)f\rangle,
	\quad f\in L_+^2(\mathbb R),
\end{equation}
where $\mathbf1_B(L_{u_0})$ is defined by the Borel functional calculus.
Write the Lebesgue decomposition of $\mu_f$ as
\begin{equation*}
	\mathrm d\mu_f(\lambda)
	=
	\rho_f(\lambda)\,\mathrm d\lambda
	+
	\mathrm d\mu_{f,\mathrm{sing}}(\lambda),
	\enspace
	\rho_f\in L^1(\mathbb R),\quad \rho_f\ge0.
\end{equation*}
Here $\mu_{f,\mathrm{sing}}$ is a finite nonnegative Borel measure
concentrated on a Borel set $S_f\subset\mathbb R$ of Lebesgue
measure zero, that is,
\begin{equation*}
	|S_f|=0,
	\enspace
	\mu_{f,\mathrm{sing}}(\mathbb R\setminus S_f)=0.
\end{equation*}
The singular part $E_{u_0,\mathrm{sing}}$ is uniquely characterized by
\begin{equation}\label{eq:singular-spectral-measure1}
	\langle f,E_{u_0,\mathrm{sing}}(B)f\rangle
	=
	\mu_{f,\mathrm{sing}}(B),
	\enspace
	f\in L_+^2(\mathbb R),\quad B\subset\mathbb R\ \text{Borel}.
\end{equation}

\begin{prop}[Badreddine, Killip, and Vi\c{s}an \cite{BKVasymp}]\label{prop:BKV-windowed}
	Let $u(t)$ denote the solution to \eqref{eq:BO} with data $u(0)=u_0\in L^2(\R)$. Given $\la \in \R$, we have:\\
	{\upshape (i)} If $\lambda$ is not an eigenvalue of $L_{u_0}$, then for any $R>0$,
	\begin{equation*}
		\lim_{|t| \to \infty}\, \|  u(t, x - 2t \la) \|_{L^2([-R,R])} = 0.
	\end{equation*}
	{\upshape (ii)} If $\lambda$ is an eigenvalue of $L_{u_0}$ and $f$ is an $L^2$-normalized eigenfunction associated with $\la$, then for any $R>0$,
	\begin{equation*}
		\lim_{|t| \to \infty} \, \left\| u(t, x - 2t \la) - \tfrac{4|\la|}{4\lambda^2(x-c)^2 + 1 } \right\|_{L^2([-R,R])} = 0, \qtq{where} c: = \Re\langle f, X^*f \rangle,
	\end{equation*}
where $\widehat{ e^{-itX^*} h }(\xi) = \widehat{h}(\xi+t) \qtq{acting on} L^2_+(\R)$.
\end{prop}
Proposition~\ref{prop:BKV-windowed} applies to all real-valued
$L^2$ data, without a smallness assumption or a restriction to
finitely many eigenvalues. Its conclusion is local along each
fixed ray.
In Theorem~\ref{thm:main-resolution}, we instead obtain a global
$H^s(\mathbb R)$ soliton--radiation decomposition for real data
satisfying
$u_0,\langle x\rangle^\beta u_0\in H^s(\mathbb R)$,
$s>-\frac12$ and $\beta>\frac12$.
Compared with Proposition~\ref{GGM26}, this lowers the Sobolev
regularity, weakens the spatial-decay assumptions, and allows a
countably infinite soliton family through the ordered limit
\eqref{eq:resolution-infinite}.
Thus, Theorem~\ref{thm:main-resolution} and
Proposition~\ref{prop:BKV-windowed} are complementary.
For the local problem, Proposition~\ref{prop:fixed-ray} extends
the windowed conclusion below $L^2$, with convergence in
$H^s_{\mathrm{loc}}$ for unweighted $H^s$ data,
$-\frac12<s<0$.

Secondly, the explicit formula developed by G\'erard \cite{Gerard} has proved to be a powerful tool for analyzing the long-time dynamics of the BO equation. In particular, Gassot, G\'erard, and Miller \cite{GGM} combined this representation with a detailed analysis of the distorted Fourier transform associated with the Lax operator to prove the soliton resolution conjecture for sufficiently regular and decaying initial data \ref{GGM26}. This suggests that a natural route toward proving soliton resolution at the sharp Sobolev regularity is to extend G\'erard's explicit formula to the spaces
$
H^s(\mathbb{R}),\enspace -\frac12<s<0.
$
However, Chen \cite{Chen} extended G\'erard's explicit formula to arbitrary initial data in $L^2(\mathbb{R})$. In view of the sharp well-posedness theory in $H^s(\mathbb{R}),\enspace s>-\frac12$, this naturally raises the open problem of extending the explicit formula further to initial data
$
u_0\in H^s(\mathbb{R}),\enspace -\frac12<s<0.
$
More precisely, it remains unknown whether
$
\bigl(X^*-2tL_{u_0}-z\operatorname{Id}\bigr)^{-1}
$
exists on $H_+^s(\mathbb{R})$, and whether G\'erard's explicit formula remains valid in the sharp well-posedness space.

Thirdly, Gassot and G\'erard \cite{GG}, by analogy with the previous work \cite{BGGM} on proving the distorted Plancherel formula, proposed the following inverse spectral problem concerning the Lax operator $L_{u_0}$ for the infinite-order multisoliton form \begin{equation}\label{form1}
	u_0\left( x \right) :=\sum_{j=1}^{\infty}{\frac{2\mathrm{Im}\left( p_j \right)}{\left| x+p_j \right|^2}},\enspace x\in \mathbb{R}:
\end{equation}

\vspace{\baselineskip}
``\textit{In view of the distorted Plancherel formula [BGGM25, Theorem 26] proven for initial data ${u}_{0}$ with $x{u}_{0} \in  {L}^{2}\left( \mathbb{R}\right)$, it is natural to ask whether all initial data ${u}_{0} \in  {L}^{2}\left( \mathbb{R}\right)$ such that the eigenvalues of the Lax operator ${L}_{{u}_{0}}$ satisfy $2\pi \sum_{k=1}^{\infty}{\left| \lambda _k \right|}=\left\| \Pi u_0 \right\| _{L^2}^{2}$ are of the form \eqref{form1}. This seems to be an open problem}''.
\vspace{\baselineskip}

This open problem is equivalent to identifying the exact range of the mapping $ ( p_j )_{j \geqslant 1} \mapsto ( p_{j}^{\infty} )_{j \geqslant 1} $, which is of significant value in elucidating the analytical structure of the BO equation with initial data taking the infinite-order multisoliton form \eqref{form1}.

These naturally lead to the following open problems:
\begin{problem}
	\normalfont
	Chen \cite{Chen} extended G\'erard's explicit formula to initial data in $L^2(\mathbb{R})$. However, for the negative regularity space $H^s(\mathbb{R}),\enspace -\frac12<s<0$, it remains unknown whether
	$
	\bigl(X^*-2tL_{u_0}-z\operatorname{Id}\bigr)^{-1}
	$
	exists on $H_+^s(\mathbb{R})$. In order to study the long-time behaviour of the BO equation in the sharp well-posedness space, how can one establish G\'erard's explicit formula for the initial data in the sharp well-posedness space?
\end{problem}
\begin{problem}
	\normalfont
	In previous work, Badreddine, Killip, and Vi\c{s}an \cite{BKVasymp} emphasized a fundamental difficulty that arises once the stringent spatial-decay hypotheses required by existing inverse scattering approaches are removed. In the unweighted Sobolev spaces relevant to the sharp well-posedness theory, embedded singular continuous spectrum of the BO Lax operator can no longer be excluded a priori, while the classical soliton--radiation picture provides no candidate description of the long-time dynamics that such a spectral component may produce. Consequently, how can one analyze the effect of embedded singular continuous spectrum on the long-time behaviour of solutions to the BO equation?
\end{problem}
\begin{problem}
	\normalfont
	By means of a new gauge transformation, together with a modified Lax-pair representation of the  BO hierarchy, Killip, Laurens, and Vi\c{s}an \cite{KLV} established global well-posedness for the BO equation in $H^s(\mathbb{R}),\enspace s>-\frac12$. The breakthrough work of Gassot, G\'erard, and Miller \cite{GGM} proved the soliton resolution conjecture for sufficiently regular and decaying initial data, while the long-time behaviour of solutions arising from more general initial data in the sharp well-posedness space remains an open problem. Thus, what spatial decay is sufficient for global soliton resolution at the sharp Sobolev regularity, and is this decay optimal?
\end{problem}
\begin{problem}
	\normalfont
	An open problem about the inverse spectral problem for the Lax operators was formulated in \cite{GG}: Are all initial data ${u}_{0} \in {L}^{2}\left( \mathbb{R}\right)$ satisfying the spectral condition $2\pi \sum_{k=1}^{\infty}{\left| \lambda _k \right|}=\left\| \Pi u_0 \right\| _{L^2}^{2}$ for ${L}_{{u}_{0}}$ of the form \eqref{form1}?
\end{problem}

To solve the open problem of extending G\'erard's explicit formula to the sharp well-posedness space, we realize the rough Lax operator through its closed quadratic form. These constructions allow the  smooth expression formed from the resolvent and the trace to converge to a limit independent of the chosen smooth approximation. This yields an extension of G\'erard's explicit formula to initial data
$
u_0\in H^s(\mathbb{R}),\enspace -\frac12<s<0.
$
This provides a spectral framework for analyzing the contribution of the continuous spectrum to the long-time dynamics of BO solutions.

For every real $u_0\in H^s(\mathbb R)$, $s>-\frac12$, we prove that the positive spectrum of $L_{u_0}$ is purely absolutely continuous of multiplicity one. Full-line gauge conjugation yields a compact Fredholm equation. A zero-flux identity and dyadic tail estimates turn each boundary kernel vector into an $L^2$ eigenfunction, which is excluded by the Wu identities. Scalar resolvent bounds and a rank-one spectral density then give the asserted spectral type, resolving Problem~2.

In order to solve the important open problem of proving the soliton resolution conjecture in the sharp well-posedness space, we first construct the rough Lax dynamics for initial data in $H^s(\mathbb{R}),\enspace -\frac12<s<0$. Wu's discrete spectral identities, fixed-contour spectral limits, rank-one Schur extraction, and the velocity observable identify the soliton profiles generated by the negative spectrum and yield uniform local asymptotics in velocity cones. Under the condition
$
u_0,\ \langle x\rangle^\beta u_0\in H^s(\mathbb{R}),
\enspace \beta>\frac12,
$
we construct the radiation component and prove global strong soliton resolution. Theorem~\ref{thm:main-weight-sharpness} shows that the exponent $\frac12$ is optimal, even for smooth initial data with exactly one negative eigenvalue.

In the infinite-order multisoliton case, Gassot and G\'erard \cite{GG} established soliton resolution by analyzing the weak convergence of the resolvent operator in the Hardy space. However, the spectral trace condition determines only the imaginary parts of the poles $\operatorname{Im}(p_j^\infty)$ and provides no corresponding characterization of their real parts $\operatorname{Re}(p_j^\infty)$ or of the exact distribution of the discrete spectral data. Consequently, the precise range of the map
$
(p_j)_{j\ge1}\longmapsto(p_j^\infty)_{j\ge1}.
$
remains unknown, leaving the associated inverse spectral problem for the Lax operator open.

In summary, the fundamental difficulties and challenges can be explicitly formulated as follows:
\begin{enumerate}[(I)]
	\item In an effort to extend G\'erard's explicit formula to the sharp well-posedness space, one must overcome the loss of the operator framework available for $L^2$ initial data in $L^2$ space. In the negative regularity space, the potential is distributional, and the resolvent in G\'erard's formula need not have a direct meaning on $H_+^s(\mathbb{R})$. At the same time, the boundary trace entering the formula is not controlled by the ambient $H_+^s$ topology alone. It is therefore difficult to pass from smooth approximations to the rough solution while preserving both the resolvent expression and the boundary evaluation, and to show that the resulting value is independent of the chosen approximation.
	
	\item For the spectral analysis of unweighted distributional potentials, one must control positive-energy boundary resolvents and exclude their kernels. The required estimates must recover global square integrability from cancellation in a boundary solution and must identify the spectral density on the full Hardy space.

	\item To prove the soliton resolution conjecture in the sharp well-posedness space, one must control the spectral data simultaneously under rough approximation and in the large-time limit. Smooth approximations may perturb the negative eigenvalues of the Lax operator, and even a small eigenvalue error changes the velocity of the corresponding soliton. The radiation term in the desired asymptotic decomposition evolves under the free BO flow, so a precise dynamical correspondence between these two evolutions is required. One must account consistently for all discrete, threshold, and continuous spectral components and establish the corresponding asymptotic norm splitting.
	
	\item The condition $2\pi \sum_{k=1}^{\infty}{\left| \lambda _k \right|}=\left\| \Pi u_0 \right\| _{L^2}^{2}$ lacks control over $\mathrm{Re}( p_{j}^{\infty} )$, which further leads to the absence of the necessary spectral gap condition when analyzing the error in the $L^\infty$-norm between the solution with initial data \eqref{form1} and $u_{\mathrm{sol}}\left( t,\cdot \right) :=\sum_{j=1}^{\infty}{R_{p_{j}^{\infty}}\left( x-c_{p_{j}^{\infty}}t \right)},\enspace x\in \mathbb{R}$.
	
\end{enumerate}

To overcome the aforementioned difficulties, we develop a spectral and dynamical framework adapted to rough initial data. First, we prove that the positive spectrum of the BO Lax operator is purely absolutely continuous with multiplicity one for real $H^s$ data, $s>-\frac12$. Under the spatial-decay hypothesis of Theorem~\ref{thm:main-resolution}, we construct the complete distorted Fourier transforms. Second, we realize the rough Lax operator through its closed quadratic form. The smooth expression involving the resolvent and the boundary trace then converges to a limit independent of the chosen smooth approximation, which extends G\'erard's explicit formula to initial data in $H^s(\mathbb{R}), \enspace -\frac12<s<0$. Third, we use fixed-contour spectral stability and Wu's discrete spectral identities to identify the soliton profiles generated by the negative spectrum. We then combine the discrete, threshold, and continuous spectral components and, under the hypotheses of the main theorem, obtain ordered soliton resolution in both time directions together with the asymptotic splitting of the $H^s$-norm. Finally, for the inverse spectral problem, we derive
a local equality identity from saturation of the sharp
negative-spectrum trace inequality.
A shift--Volterra argument identifies an invariant Hardy
subspace, and model-space factorization then recovers
the initial pole representation.

Our method provides the estimates needed to extend the
explicit-formula approach to distributional potentials and
global radiative asymptotics, beyond the range covered by
the arguments in \cite{GGM,BKVasymp} as presently formulated.
In \cite[Lemma~4.3]{GGM}, the coordinate regularity of bound
states is obtained by differentiating $\widehat u_0$ and
using the spatial moment $xu_0\in L^2$.
The unweighted $L^2$ domain inclusion is already established
in \cite[Proposition~3.1]{GG}, while
\cite[Proposition~3.5 and Corollary~3.6]{BKVasymp}
provide coordinate-resolvent estimates and a contour argument
at the $L^2$ level.
However, these estimates depend on the $L^2$ norms and
$L^2$ differences of the potentials, which are not controlled
by convergence in $H^s$ when $s<0$.
Theorem~\ref{thm:main-eigenfunction-graph} supplies the
corresponding estimates in the unweighted $H^s$ topology
for $-\frac12<s<0$.
The essential step is to control the zero-frequency boundary
term by $H^s$--$H^{-s}$ duality, using
$H^{1+s}\hookrightarrow H^{-s}$, and combine this control
with cross-scale resolvent bounds and commutator identities.
Integration over a fixed isolating contour then yields
$f_j\in\operatorname{Dom}(X^*)$ and
$f_j,X^*f_j\in H_+^{1+s}$, together with graph convergence
for each isolated negative eigenvalue after a consistent
choice of eigenvector phases.
The boundary traces and soliton centers therefore pass
to the limit under rough approximation without requiring
uniform $L^2$ bounds or spatial-moment bounds on the
approximating potentials.
This provides the spectral stability that cannot be
deduced directly from the moment-dependent estimates of
\cite{GGM} or the $L^2$ estimates of \cite{BKVasymp}.
Together with the approximation-independent extension
of G\'erard's formula in
Theorem~\ref{thm:main-gerard}, it makes the explicit-formula
analysis available below $L^2$.
A separate difficulty is the identification of the radiation:
the windowed convergence of
\cite[Theorem~1.5]{BKVasymp} does not determine a free
radiation profile in the global norm.
We address this through the positive-energy resolvent
analysis of Theorem~\ref{thm:main-no-singular} and,
under
$u_0,\langle x\rangle^\beta u_0\in H^s(\mathbb R)$,
$s>-\frac12$ and $\beta>\frac12$,
through distorted Fourier analysis and estimates controlling
the threshold and spectral tails.
These ingredients yield the global $H^s$
soliton--radiation decomposition of
Theorem~\ref{thm:main-resolution} in both time directions.
When the negative spectrum is infinite, the conclusion
uses an ordered limit in which the large-time limit
precedes removal of the spectral cutoff.
This extends the finite-soliton setting of \cite{GGM}
and treats radiative data beyond the infinite-order
multisoliton class considered in \cite{GG}.
The global result retains a spatial weight and is therefore
complementary to the unweighted $L^2$ windowed theorem
of \cite{BKVasymp}. It is worth noting that there are some highlights in our work as
follows.
\begin{itemize}
	\item In Theorem~\ref{thm:main-gerard}, we extend G\'erard's explicit formula from real-valued initial data in $L^2(\mathbb{R})$, as established by Chen \cite{Chen}, to
	$
	u_0\in H^s(\mathbb{R}),\enspace -\frac12<s<0,
	$
	and prove that the formula is obtained as the limit of the corresponding smooth G\'erard formulas, independently of the chosen Schwartz approximation of $u_0$.
	
	\item In Theorem~\ref{thm:main-eigenfunction-graph}, we extend the unweighted $L^2$ coordinate regularity results of \cite{BKVasymp,GG} to potentials $u_0\in H^s(\mathbb R)$, $-\frac12<s<0$, and establish convergence of the bound-state eigenfunctions in the $X^*$ graph under rough approximation. This yields convergence of their boundary traces and soliton centers after a consistent choice of eigenvector phases.

    \item In Theorem~\ref{thm:main-no-singular}, we prove that the positive spectrum of $L_{u_0}$ is purely absolutely continuous with multiplicity one for every real $u_0\in H^s(\mathbb R)$, $s>-\frac12$.	
	
	\item Theorem~\ref{thm:main-resolution} gives global $H^s$ soliton resolution in both time directions under $\langle x\rangle^\beta u_0\in H^s$, $\beta>\frac12$. Theorem~\ref{thm:main-weight-sharpness} proves the sharpness of this weight for finite soliton resolution by constructing a smooth endpoint counterexample with exactly one negative eigenvalue.
	
	\item The equivalence between the Lax operator satisfying \eqref{laxtezheng} and the initial data satisfying \eqref{form1} is proved in Theorem \ref{thm3}, thereby solving the open problems regarding the inverse spectral problem for the Lax operator posed in \cite{GG}.
\end{itemize}

In this work, we prove the soliton resolution conjecture for the BO equation in the sharp well-posedness space, and solve the open problems for the Lax operator, singular continuous spectrum and G\'erard's explicit formula. Our main results read as follows. 
\subsection{Main results}
In an effort to study the long-time dynamics of the BO equation in the sharp well-posedness space and solve Problem~1, we develop a smooth approximation framework for G\'erard's explicit formula. Through this framework, we realize the rough Lax operator by its closed quadratic form. This extends G\'erard's explicit formula to initial data in $H^s(\mathbb R), \enspace -\frac12<s<0$, and provides the representation formula stated in Theorem~\ref{thm:main-gerard}.
\begin{theorem}[The extension of G\'erard's explicit formula]
	\label{thm:main-gerard}
	Let $u_{0,n}\in\mathcal S(\mathbb R;\mathbb R)$ converge to $u_0$ in
	$H^s(\mathbb R),\enspace s>-\frac12$, and let $u_n$ denote the corresponding
	smooth BO solutions. Then
	\begin{equation}
		u_+(t,z)
		=
		\lim_{n\to\infty}
		\frac{1}{2\pi i}I_+
		\left[
		(\overline{X^*-2(t-\tau)L_{u_n(\tau)}}-z)^{-1}(u_n(\tau))_+
		\right],\enspace t,\tau\in\mathbb R,\enspace z\in\mathbb C_+,
		\label{eq:rough-Gerard-main}
	\end{equation}
	where
	$$
	u_+(t,z):=\frac{1}{\sqrt{2\pi}}
	\int_0^\infty e^{iz\xi}\widehat{u(t)}(\xi)\,\rmd\xi,
	\enspace z\in\mathbb C_+,
	$$
	and 
	$$
	I_+(f):=\sqrt{2\pi}\lim_{\xi\to 0^+}\widehat f(\xi),
	\enspace f\in\operatorname{Dom}(X^*).
	$$
\end{theorem}
\begin{re}
	Comparing Theorem~\ref{thm:main-gerard} with the result of Chen \cite{Chen}, Chen extended G\'erard's explicit formula to initial data in $L^2(\mathbb R)$, while Theorem~\ref{thm:main-gerard} extends the formula further to initial data in the sharp well-posedness space, thereby answering Problem 1.
\end{re}

The discrete spectral component also requires control of the eigenfunctions
in the domain of the coordinate operator $X^*$.  This makes the soliton
centers well defined and allows them to pass to the limit under rough
approximation.  The following theorem establishes this regularity and
stability for potentials of negative Sobolev regularity, without a
spatial weight.

\begin{theorem}[Coordinate regularity and stability of bound states]
\label{thm:main-eigenfunction-graph}
Let $u_0\in H^s(\mathbb R;\mathbb R)$ with $-\frac12<s<0$.
For every negative eigenvalue $\lambda_j=-a_j<0$ of $L_{u_0}$ and
normalized eigenfunction $f_j$, one has
\begin{equation}
 f_j\in\operatorname{Dom}(X^*),\enspace
 f_j,\ X^*f_j\in H_+^{1+s}.
 \label{eq:main-eigenfunction-graph}
\end{equation}
The quantities $p_j=-\langle f_j,X^*f_j\rangle$ and
$c_j=\operatorname{Re}\langle f_j,X^*f_j\rangle$ satisfy
\begin{equation}
 |I_+(f_j)|^2=\frac{2\pi}{a_j},\enspace
 p_j=-c_j+\frac{i}{2a_j}\in\mathbb C_+.
 \label{eq:main-eigenfunction-center}
\end{equation}
If $u_{0,n}\in H^s(\mathbb R;\mathbb R)$ and $u_{0,n}\to u_0$ in
$H^s$, then, for each fixed $j$ and all sufficiently large $n$, there
are negative eigenvalues $\lambda_{j,n}\to\lambda_j$ of $L_{u_{0,n}}$
and corresponding normalized eigenfunctions $f_{j,n}$, with phases
chosen so that
\begin{equation}
 \|f_{j,n}-f_j\|_{H^{1+s}}
 +\|X^*f_{j,n}-X^*f_j\|_{H^{1+s}}\longrightarrow0.
 \label{eq:main-eigenfunction-stability}
\end{equation}
For this choice,
\begin{equation}
 I_+(f_{j,n})\longrightarrow I_+(f_j),\enspace
 p_{j,n}:=-\langle f_{j,n},X^*f_{j,n}\rangle\longrightarrow p_j.
 \label{eq:main-eigenfunction-data-stability}
\end{equation}
\end{theorem}

\begin{re}
The unweighted $L^2$ domain inclusion is known from
\cite{BKVasymp,GG}; the corresponding coordinate estimates are recalled
in Proposition~\ref{prop:prep-L2-graph}.
Theorem~\ref{thm:main-eigenfunction-graph} extends this conclusion to
$H^s(\mathbb R)$, $-\frac12<s<0$, and gives convergence in the
$X^*$ graph under $H^s$ approximation.  This also controls the boundary
traces and soliton centers used in Theorem~\ref{thm:main-resolution}.
The domain condition means $\widehat f_j\in H^1(0,\infty)$; it does
not assert $xf_j\in L^2(\mathbb R)$.
\end{re}

In order to prove the soliton resolution conjecture in the sharp well-posedness space and solve Problem~2, we develop a weighted limiting-absorption framework for the rough BO Lax operator. Through this framework, we prove that the positive spectrum is purely absolutely continuous with multiplicity one, thereby excluding embedded singular continuous spectrum for the class of initial data under consideration, which provides the essential framework for proving the following Theorem~\ref{thm:main-no-singular}.
\begin{theorem}[Purely absolutely continuous positive spectrum]
	\label{thm:main-no-singular}
	Let $u_0\in H^s(\mathbb R)$, $s>-\frac12$, be real-valued.
	Let $E_{u_0}(\cdot)$ be the projection-valued spectral measure
	of $L_{u_0}$ defined in \eqref{eq:spectral-measure-definition},
	and let $E_{u_0,\mathrm{sing}}$ denote its singular part
	defined in \eqref{eq:singular-spectral-measure1}.
	Then for every compact interval $J\subset(0,\infty)$, we have
	\begin{equation}\label{eq:no-singular-main}
		E_{u_0,\mathrm{sing}}(J)=0.
	\end{equation}
	Thus the spectrum of $L_{u_0}$ in $(0,\infty)$ is absolutely
	continuous of multiplicity one.
\end{theorem}
\begin{re}
	Theorem~\ref{thm:main-no-singular} addresses the open problem concerning the contribution of embedded singular continuous spectrum to the long-time dynamics of BO solutions. In particular, the singular part of the spectral measure assigns zero measure to every compact interval $J\subset(0,\infty)$. Hence, embedded singular continuous spectrum is absent and makes no contribution to the long-time dynamics for the class of initial data covered by Theorem~\ref{thm:main-no-singular}, thereby answering Problem 2.
\end{re}

In order to prove the soliton resolution conjecture in the sharp well-posedness space and solve Problem~3, we develop a spectral decomposition framework for the BO flow. Through this framework, we identify the soliton profiles associated with the negative spectrum, represent the positive continuous spectrum by the free radiative term, and control the threshold and spectral components,  as detailed in Theorem \ref{thm:main-resolution}.
\begin{theorem}[Soliton resolution conjecture in the sharp well-posedness space]\label{thm:main-resolution}
	Let the initial data $u_0\in H^s\left( \mathbb{R} \right), \enspace s>-\frac{1}{2}$ be a real-valued function such that 
	\begin{equation}\label{thm1tiaojian}
		\langle x\rangle ^{\beta}u_0\in H^s(\mathbb{R} ),\enspace \beta >\frac{1}{2},\enspace s>-\frac{1}{2}.
	\end{equation}
	Let $\left\{ \lambda _j \right\} _{j=1}^{N}\subset \mathbb{R}$ denote the negative eigenvalues of $L_{u_0}$ with $\lambda _m<\lambda _n<0,\enspace 1\leqslant m<n\leqslant N$. For each negative eigenvalue $\lambda_j<0$ of $L_{u_0}$, let $f_j$ be the
	corresponding normalized eigenfunction and set
	$
	p_j:=-\langle f_j,X^*f_j\rangle,
	\enspace
	R_{p_j}(x):=\frac{2\operatorname{Im}p_j}{|x+p_j|^2}.$ $$\mathrm{e}^{t\partial _x|D|}u_{\infty}^{\pm}:=\frac{1}{\sqrt{2\pi}}\int_{\mathbb{R}}{e^{ix\xi +it\xi |\xi |}}\widehat{u_{\infty}^{\pm}}(\xi )\,d\xi ,\enspace u_{\infty}^{\pm}\in H^s\left( \mathbb{R} \right) .
	$$Then there exist real-valued function
	$u_\infty^\pm\in H^s(\mathbb R)$ such that the following
	conclusions hold.
	
	If $N<\infty$, then the solution $u\left(t,x \right) $ can be decomposed as follows:
	\begin{equation}
		\underset{t\rightarrow \pm \infty}{\lim}\left\| u\left( t,x \right) -\sum_{j=1}^N{R_{p_j}\left( x-\frac{t}{\mathrm{Im}p_j} \right)}-e^{t\partial _x\left| D \right|}u_{\infty}^{\pm} \right\| _{H_{x}^{s}\left( \mathbb{R} \right)}=0.
		\label{eq:resolution-finite}
	\end{equation}
	If $N=\infty$, then $\sum_{j=1}^{\infty}{\left| \lambda _j \right|}<\infty$ and the solution $u\left(t,x \right) $ can be decomposed as follows:
	\begin{equation}
		\lim_{\delta \rightarrow 0^+} \underset{t\rightarrow \pm \infty}{\limsup}\left\| u\left( t,x \right) -\sum_{-\lambda _j\ge \delta}{R_{p_j}}\!\left( \,x-\frac{t}{\mathrm{Im}p_j} \right) -e^{t\partial _x|D|}u_{\infty}^{\pm} \right\| _{H_{x}^{s}\left( \mathbb{R} \right)}=0.
		\label{eq:resolution-infinite}
	\end{equation}
\end{theorem}

\begin{re}
	Comparing Theorem~\ref{thm:main-resolution} with Proposition~\ref{GGM26}, we see that Proposition~\ref{GGM26} assumes that $u_0\in H^1(\mathbb R)$, $xu_0\in H^1(\mathbb R)$, and $x^2u_0(x)=c_0+v_0(x)$ for some $c_0\in\mathbb R$ and $v_0\in L^2(\mathbb R)$. In Theorem~\ref{thm:main-resolution}, we extend the initial data regularity to $u_0\in H^s(\mathbb R)$, $s>-\frac12$, under the condition that $u_0,\langle x\rangle^\beta u_0\in H^s(\mathbb R)$, $\beta>\frac12$, thereby solving Problem 3.
\end{re}
We point out that the spatial weight threshold in
Theorem~\ref{thm:main-resolution} is sharp. Its soliton resolution
conclusion fails in general if the weight assumption is weakened to
$\langle x\rangle^\beta u_0\in H^s(\mathbb R)$ with
$\beta\leqslant\frac12$, as shown by the following theorem.
\begin{theorem}[Sharpness of the spatial weight]
	\label{thm:main-weight-sharpness}
	There exists a real-valued initial datum $u_0$ such that
	\begin{equation}
		\langle x\rangle^{1/2}u_0\in H^m(\mathbb R),
		\enspace m=0,1,2,\ldots,
		\label{EP-data}
	\end{equation}
	and $L_{u_0}$ has exactly one simple negative eigenvalue and purely
	absolutely continuous spectrum of multiplicity one on $(0,\infty)$.
	Let $u(t,x)$ be the corresponding BO solution and set
	$R_p(x):=2\operatorname{Im}p/|x+p|^2$ for $p\in\mathbb C_+$.
	Then, for every $s>-\frac12$, there are no integer $N\ge0$,
	parameters $p_1,\ldots,p_N\in\mathbb C_+$, and real-valued
	$u_\infty^+\in H^s(\mathbb R)$ such that
	\begin{equation}
		\lim_{t\to+\infty}
		\left\|
		u(t,x)
		-\sum_{j=1}^{N}R_{p_j}
		\left(x-\frac{t}{\operatorname{Im}p_j}\right)
		-e^{t\partial_x|D|}u_\infty^+
		\right\|_{H_x^s(\mathbb R)}
		=0.
		\label{EP-failure}
	\end{equation}
\end{theorem}

\begin{re}
	The Calogero--Moser derivative nonlinear Schr\"odinger equation (CM--DNLS) is closely related to the BO equation through their Lax structures \cite{FrankRead25,GL24,KLV25}. For the CM--DNLS equation, Kim and Kwon established the soliton resolution result for the focusing case in the $H^{1,1}(\mathbb R)$ space \cite{KK24}. Another equation closely related to the BO equation is the intermediate long-wave equation, while its soliton resolution remains an important open problem \cite{Saut2}. For further studies of integrable equations closely related to the BO equation, we refer to \cite{GL24,GL26,GP23,GP24}.
\end{re}
In contrast to the well-developed spectral analysis of the Lax operator in the finite-order case, the understanding of its discrete spectrum in the infinite-order case remains incomplete, leaving several open problems unsolved. An example is the inverse spectral problem formulated by Gassot and G\'erard in \cite{GG}. Thus, we construct a class of trace-class operators and reduce the inverse spectral problem to the analysis of these operators. Consequently, we establish that the characterization of the discrete spectrum of the Lax operator satisfying \eqref{laxtezheng} is equivalent to the representability of the initial data $u_0(x)$ in the form of \eqref{form1}, as detailed in Theorem \ref{thm3}.
\begin{theorem}[The inverse spectral problem for the infinite-order multisoliton]\label{thm3}
	Let $u_0\in L^2(\mathbb R;\mathbb R)$, and suppose that the complete negative spectrum of $L_{u_0}$ consists of infinitely many eigenvalues
	$\lambda _1<\lambda _2<\cdots<0$.  If
	\begin{equation}\label{laxtezheng}
		2\pi \sum_{k=1}^{\infty}{\left| \lambda _k \right|}=\left\| \Pi u_0 \right\| _{L^2}^{2}
	\end{equation}
	then there exists a sequence $(p_j^{\rm in})_{j\ge1}\subset\mathbb C_+$ such that
	\begin{equation}\label{eq:ISP-pole-representation}
		u_0(x)=\sum_{j=1}^{\infty}
		\frac{2\operatorname{Im}p_j^{\rm in}}{|x+p_j^{\rm in}|^2},\enspace x\in\mathbb R,
	\end{equation}
	where the series converges in $L^2(\mathbb R)$ and absolutely and uniformly on $\mathbb R$, and
	\begin{equation}\label{eq:ISP-pole-condition}
		\sum_{j,k\ge1}\frac{\operatorname{Im}p_j^{\rm in}}
		{|p_j^{\rm in}-\overline{p_k^{\rm in}}|^2}<\infty.
	\end{equation}
\end{theorem}
\begin{re}
	Theorem \ref{thm3} solves the open problem proposed by Gassot and G\'erard \cite{GG}, namely Problem 4, proving that if and only if the discrete spectrum of the Lax operator $L_{u_0}$ satisfies condition \eqref{laxtezheng}, the initial data can be expressed in the form \eqref{form1}. Combined with the work of \cite{GG}, this implies that \eqref{laxtezheng} is equivalent to the initial data taking the form \eqref{form1}, which also corresponds to the result of the distorted Plancherel formula in \cite{BGGM}.
\end{re}
\subsection{Plan of the proof}
We develop a spectral and dynamical framework adapted to the sharp well-posedness space. For Theorem~\ref{thm:main-gerard}, we realize the rough Lax operator through a closed quadratic form, prove the stability of its resolvents under smooth approximation, and pass the smooth resolvent and boundary-trace formula to the rough limit, thereby extending G\'erard's explicit formula to $H^s(\mathbb R)$, $-\frac12<s<0$. For Theorem~\ref{thm:main-eigenfunction-graph}, uniform coordinate estimates for the resolvent allow us to integrate the Riesz projections in the $X^*$ graph, giving eigenfunction regularity and stability of their boundary traces and centers. For Theorem~\ref{thm:main-no-singular}, full-line gauge conjugation gives a compact Fredholm equation. A zero-flux identity, dyadic tail estimates, and a contraction argument yield boundary invertibility; the resolvent jump then gives a rank-one positive spectral density.  For Theorem~\ref{thm:main-resolution}, we use the discrete spectral identities, the coordinate graph, and the velocity observable to identify the soliton parameters associated with the negative spectrum; profile rigidity then yields the local soliton limits, while weighted limiting absorption supplies the distorted Fourier representation of the radiation. The discrete, threshold, continuous, and high-energy components are subsequently combined through the four-block decomposition and the large-time germ construction, which give the ordered soliton--radiation decomposition and the asymptotic splitting of the $H^s$-norm. For Theorem~\ref{thm:main-weight-sharpness}, distant negative perturbations produce a fixed phase change in the classical radiation while tending to zero in the endpoint weighted topology. A Baire category argument then produces the counterexample. Finally, for Theorem~\ref{thm3}, equality in the sharp negative-spectrum trace produces a local trace identity; a shift--Volterra elimination and Hardy model-space factorization then recover the infinite Poisson pole data.

The structure of this work is as follows. Section~\ref{sec:rough-operator} establishes the product estimates required for the quadratic-form construction of the rough Lax operator, identifies its self-adjoint realization and operator domain, and proves cross-scale resolvent stability and spectral projection convergence. Section~\ref{sec:rough-Gerard} develops Hardy evaluation at negative regularity, constructs the conserved Stieltjes model and the cyclic transport, and proves Theorem~\ref{thm:main-gerard}. Section~\ref{sec:rigged-range} establishes the rigged-range identities, proves Theorem~\ref{thm:main-eigenfunction-graph}, removes the non-eigenvector directions from the moving resolvent, and derives the fixed-ray asymptotics. Section~\ref{sec:negative-velocity} analyzes the negative spectrum, constructs the graph centers, and proves the velocity law for the discrete spectral channels. Section~\ref{sec:profile-rigidity} proves profile rigidity and upgrades the Hardy limits to strong local $H^s$-convergence uniformly on velocity cones. Section~\ref{sec:no-singular} establishes weighted limiting absorption estimates. Section~\ref{sec:trace-transform} proves Theorem~\ref{thm:main-no-singular} and constructs the wave operators, distorted Fourier transforms, and $H^s$-Plancherel theory used in the radiation analysis. Section~\ref{sec:uniform-sewing} proves the uniform radiation limit on compact positive-energy intervals. Section~\ref{sec:four-block-resolution} combines the discrete, threshold, compact positive-energy, and high-energy blocks, and Section~\ref{sec:germ} constructs the large-time soliton germ and proves the ordered resolution and energy splitting in Theorem~\ref{thm:main-resolution}. Section~\ref{sec:weight-sharpness} proves Theorem~\ref{thm:main-weight-sharpness}. Section~\ref{sec:inverse-spectral-reconstruction} completes the trace-saturation and model-space reconstruction in Theorem~\ref{thm3}.
\begin{Acknowledgment}
	This work was supported by the National Natural Science Foundation of China under Grant No. 12371255, the Fundamental Research Funds for the Central Universities of CUMT under Grant No. 2024ZDPYJQ1003, and the Postgraduate Research \& Practice Program of Education \& Teaching Reform of CUMT under Grant No. 2025YJSJG031. 
\end{Acknowledgment}

\section{Notation and preliminaries}
\label{sec:preliminaries}
\label{sec:rough-operator}

\subsection{Notation and conventions}
\label{subsec:notation-guide}

Table~\ref{tab:principal-notation} records the notation used below.
The correspondence with the cited BO papers is explained after the table;
the indicated propositions and sections specify the full domains and norms.

\begingroup
\small
\renewcommand{\arraystretch}{1.15}
\begin{longtable}{@{}p{0.28\textwidth}p{0.68\textwidth}@{}}
\caption{Principal notation.}\label{tab:principal-notation}\\
\toprule
Symbol & Convention, meaning, and location \\
\midrule
\endfirsthead
\multicolumn{2}{@{}l}{\textit{Table~\thetable\ continued}}\\
\toprule
Symbol & Convention, meaning, and location \\
\midrule
\endhead
\bottomrule
\endfoot
$\mathcal F,\ \widehat f,\ c_{\rm F}$ &
Unitary Fourier transform:
$\widehat f(\xi)=(2\pi)^{-1/2}\int_{\mathbb R}e^{-ix\xi}f(x)\,\rmd x$;
$c_{\rm F}=(2\pi)^{-1/2}$.\\
$\langle f,g\rangle$ &
$\int_{\mathbb R}\overline f g$, linear in the second argument;
the Sobolev duality pairings use the same convention.\\
$D,\ \langle D_x\rangle$ &
$D=-i\partial_x$ and the Bessel multiplier
$\langle D_x\rangle=(1-\partial_x^2)^{1/2}$.\\
$\Pi,\ H_+^r,\ L_+^2,\ \mathcal F_0$ &
$\Pi=\mathbf1_{[0,\infty)}(D)$;
$H_+^r=\{f\in H^r:\operatorname{supp}\widehat f\subset[0,\infty)\}$;
$L_+^2=H_+^0$; $\mathcal F_0=\mathcal F|_{L_+^2}$ takes values in
$L^2(0,\infty)$.\\
$h_+,\ u_+(t,z),\ P_{>M}$ &
$h_+=\Pi h$, in particular $u_+(t)=\Pi u(t)$ and
$(u_0)_+=u_+(0)$; $u_+(t,z)=\mathcal E_z(u_+(t))$ is the analytic
evaluation in \eqref{eq:Hardy-evaluation};
$P_{>M}=\mathbf1_{\{|\xi|>M\}}(D)$.\\
$X^*,\ I_+$ &
$\widehat{X^*f}=i\partial_\xi\widehat f$ on
$\operatorname{Dom}(X^*)=\{f\in L_+^2:\widehat f\in H^1(0,\infty)\}$;
$I_+f=\sqrt{2\pi}\widehat f(0+)$.\\
$\mathcal D_+^\circ,\ \iota_+,\ M_x$ &
Vectors with Fourier transforms in $C_c^\infty(0,\infty)$; Hardy inclusion
$\iota_+:L_+^2\hookrightarrow L^2(\mathbb R)$; full-line multiplication
by $x$, respectively. Section~\ref{sec:rigged-range} and
Lemma~\ref{lem:Naimark}.\\
$L_v,\ T_v,\ E_v(B)$ &
For real $v$, the rough Lax operator $L_v=D-T_v$, where $T_vf=\Pi(vf)$ is interpreted
distributionally; $E_v(B)=\mathbf1_B(L_v)$ for Borel $B$.
Proposition~\ref{prop:closed-Lax}.\\
$\sigma(L),\ \operatorname{res}(L),\ \#$ &
Spectrum, resolvent set, and pushforward of a measure, respectively:
$T_\#\mu(B)=\mu(T^{-1}(B))$ for a Borel map $T$.\\
$\ka,\ A_t,\ H_{\kappa,+}^r$ &
Common positive shift and $A_t=L_{u(t)}+\ka$
(Proposition~\ref{prop:global-orbit-control}); $H_{\kappa,+}^r$ is the
Hardy restriction of the shifted Sobolev norm used in
Section~\ref{subsec:preliminaries}.\\
$w_t,\ m_t,\ m(\ka;v)$ &
$w_t=A_t^{-1/2}u_+(t)$,
$m_t=A_t^{-1}u_+(t)=m(\ka;u(t))$, and
$m(\ka;v)=(L_v+\ka)^{-1}v_+$.\\
$\mu_v,\ \mu_{u_0}^{\rm c}$ &
Stieltjes measure associated with $v_+$, with
$\mu_{u(t)}=\mu_{u_0}$ along the BO flow;
$\mu_{u_0}^{\rm c}=\mu_{u_0}|_{[0,\infty)}$.
\cite[Proposition~7.1]{BKV} and
Lemma~\ref{lem:Stieltjes-conservation}.\\
$\nu_t,\ \nu$ &
Finite spectral measure of $w_t$;
$\nu_t=\nu$ and $\rmd\nu=(E+\ka)^{-1}\rmd\mu_{u_0}$.
The spectral measure of $m_t$ is $(E+\ka)^{-1}\rmd\nu$.
Equations~\eqref{eq:rho-measure}--\eqref{eq:nu-measure}.\\
$\beta(\eta;v),\ \beta_n(a)$ &
Shifted-resolvent generating function
(\cite[(7.1)]{BKV}) and the smooth reflection coefficient
for $u_{0,n}$ (\eqref{eq:prep-reflection-coefficient}), respectively;
their arguments distinguish them.\\
$N,\ \mathcal J,\ \lambda_j=-a_j$ &
Number and index set of the negative eigenvalues, with
$N\in\{0,1,\ldots,\infty\}$, $a_j>0$, and $\sum_j a_j<\infty$.
Proposition~\ref{prop:negative-spectrum}.\\
$f_j,\ c_j,\ p_j$ &
Normalized eigenfunction, spatial center, and asymptotic pole:
$c_j=\operatorname{Re}\langle f_j,X^*f_j\rangle$ and
$p_j=-\langle f_j,X^*f_j\rangle=-c_j+i/(2a_j)$;
the velocity is $2a_j$. Proposition~\ref{prop:negative-spectrum}.\\
$p_m^{\rm in},\ a_{\min}$ &
Initial poles in Theorem~\ref{thm3}, which need not equal $p_j$;
$a_{\min}=\min_j a_j$ for a nonempty finite spectrum, and $a_{\min}=1$
for the empty spectrum
(Corollary~\ref{coro:CR-finite-spectrum-error-modulus}).\\
$U_n(t,\tau),\ U_{u_0}^{\mathrm{cyc}}(t,\tau)$ &
Auxiliary unitary propagator on a smooth orbit
(\cite[Proposition~3.9]{BKVasymp}) and transport on the rough
cyclic scale (Proposition~\ref{prop:cyclic-transport}).
$U_{u_0}^{\mathrm{cyc}}(t)=U_{u_0}^{\mathrm{cyc}}(t,0)$.\\
$m_\pm,\ \widetilde f^{\,\pm},\ \mathcal F_v^\pm$ &
Jost solutions normalized at $\pm\infty$, distorted coefficients, and
the transforms $\mathcal F_v^\pm f=c_{\rm F}\widetilde f^{\,\pm}$,
unitary on the absolutely continuous subspace.
\cite[(2.4)--(2.7)]{GGM} and
Section~\ref{sec:trace-transform}.\\
$\ell_v(\lambda),\ \Gamma_v(\lambda)$ &
Radiation multiplier and transmission coefficient:
$m_+=\ell_vm_-$,
$\widetilde f^{\,-}=\ell_v\widetilde f^{\,+}$, and
$\Gamma_v=\ell_v^{-1}$.
\eqref{eq:prep-Jost-phase}.
$\ell_v[g,f]$ denotes the quadratic form instead.\\
$J,\ K$ &
Compact energy interval $J\Subset(0,\infty)$ and compact set of
observation points $K\Subset\mathbb C_+$.\\
$X_\delta,\ \mathcal Y_\sigma$ &
Weighted spatial-block target and source spaces, respectively.
Section~\ref{sec:no-singular}.\\
$G(z),\ B(z),\ \mathcal M_{u_0}(z)$ &
Free resolvent with synthesis, compact potential factor, and Fredholm
inverse $\mathcal M_{u_0}(z)=(I-B(z))^{-1}$.
Sections~\ref{sec:no-singular} and~\ref{sec:trace-transform}.\\
$\mathcal X,\ \Sigma,\ \tau_\lambda$ &
$\mathcal X=\operatorname{Ran}W_R^*$ with its quotient Hilbert norm;
source synthesis map $\Sigma$; Fourier point trace
$\tau_\lambda f=\widehat f(\lambda)$, used as $\tau_\lambda\Sigma$
on synthesized sources. Section~\ref{sec:trace-transform}.\\
$A_R,\ C_R,\ F_R$ &
Weighted compositions $A_R=W_RG$, $C_R=\mathcal V R_0W_R^*$,
$F_R=W_RR_0W_R^*$; $R$ indexes the spatial weight.
Equation~\eqref{eq:NS-Woodbury}.\\
$\chi_{\rm E},\ \chi_{\rm th}$ &
Cutoffs near a positive-energy interval and near zero.
Lemma~\ref{lem:NS-radiation-tail}.\\
$b_+,\ b_{+,R},\ c_{\rm Fr},\ \mathfrak t^0$ &
Outgoing step, truncated primitive, Fresnel normalization, and threshold
scalar. Section~\ref{sec:uniform-sewing}.\\
$P_\delta^-,\ P_\delta^0,\ H_R$ &
Spectral projections onto $(-\infty,-\delta]$, $(-\delta,\delta)$,
and $(R,\infty)$, respectively.
Section~\ref{sec:four-block-resolution}.\\
$\mathcal G_\sigma^s$ &
Quotient of bounded $H^s$ paths by paths tending to zero as
$t\to\sigma\infty$. Section~\ref{sec:germ}.\\
\end{longtable}
\endgroup

We use the unitary Fourier transform and second-entry-linear pairing
of \cite{BKV,BKVasymp,KLV}.  In \cite{GG,GGM}, the Fourier transform is
$\sqrt{2\pi}\,\widehat f$ and the pairing is
$\int f\overline g=\langle g,f\rangle$.
Thus $I_+$ and $X^*$ retain their normalization in \cite{GG,GGM}, and
$-\langle f_j,X^*f_j\rangle$ here is
$-\langle X^*f_j,f_j\rangle$ in their convention.
The operator denoted here by $X^*$ is denoted by $X$ in
\cite{BKV,BKVasymp,KLV}: in each case it is the maximal Fourier
derivative with domain $H^1(0,\infty)$.
Our $\Pi$ and $L_v$ are $C_+$ and $\mathcal L$ in \cite{KLV};
$m(\ka;v)$ is their resolvent vector, and $\mu_v$ is the Stieltjes
measure of \cite[Proposition~7.1]{BKV}. Let $P_{\mathrm{ac}}(L_v)$ denote the orthogonal projection onto
the absolutely continuous spectral subspace of $L_v$.
For every Borel set $B\subset\mathbb R$, define
\begin{equation}\label{eq:singular-spectral-measure}
	E_{v,\mathrm{sing}}(B)
	:=
	E_v(B)\bigl(I-P_{\mathrm{ac}}(L_v)\bigr),
\end{equation}
where $I$ is the identity on $L_+^2(\mathbb R)$.

For the discrete data, $p_m^{\rm in}$, $p_j$, and $f_j$ here correspond
to $p_m$, $p_j^\infty$, and $\varphi_j$ in \cite{GG}, respectively.
We retain $p_j$ for asymptotic poles as in \cite{GGM}, and $f_j$ for
eigenfunctions as in \cite{BKV}.  The velocity $2a_j$ is
$c_{p_j}$ in \cite{GG,GGM}; $c_j$ here is the spatial center.
For the scattering data, we retain $m_\pm$ and
$\widetilde f^{\,\pm}$ from \cite{BGGM,GGM}, with the factor
$c_{\rm F}$ in the unitary transforms.
Resolvent columns at $\lambda\pm i0$ correspond to $c_{\rm F}m_\mp$,
whereas the dual spectral rows are represented by $c_{\rm F}m_\pm$;
see \eqref{eq:SC-reciprocity}.
The convention for $\ell_v$ is fixed by
\eqref{eq:prep-Jost-phase}, so the transmission limit at $+\infty$
of $m_-$ is $\Gamma_v=\ell_v^{-1}$.

Unless stated otherwise, the rough-potential estimates below use
$-\frac12<s<0$.  Nonnegative regularity is treated separately in the
proofs of the main theorems.

\subsection{Preliminaries}
\label{subsec:preliminaries}
\label{subsec:rough-Lax-preliminaries}
\label{subsec:imported-rough-results}
\label{subsec:prep-smooth-flow}
\label{subsec:prep-classical-scattering}
\label{subsec:prep-GGM-radiation}
\label{subsec:prep-general-tools}

We use the rough Lax theory and BO flow of Killip, Laurens, and
Vi\c{s}an \cite{KLV}, and the spectral and variational identities of
Badreddine, Killip, and Vi\c{s}an \cite{BKV}.

For $\kappa\ge1$ and $r\in\mathbb R$, put
$\|f\|_{H_\kappa^r}^2
=\int_{\mathbb R}(|\xi|+\kappa)^{2r}|\widehat f(\xi)|^2\,\rmd\xi$.
At a fixed $\kappa$ this is an equivalent norm on $H^r$;
$H_{\kappa,+}^r$ denotes its Hardy restriction.

For real $u\in H^s(\mathbb R)$, $-\frac12<s<0$, we use the
self-adjoint Lax realization of \cite[Proposition~3.2]{KLV}.
An admissible shift means
$\kappa\ge\kappa_0(u):=C_s(1+\|u\|_{H^s})^{2/(1+2s)}$,
with $C_s\ge1$ chosen as in \cite[(2.15)]{BKV}.
The shifted resolvent estimates used below are
\cite[(2.18)--(2.19)]{BKV}.

\begin{prop}[Product and infinitesimal form bounds]
	\label{lem:product-form}
	Let $-\frac12<s<0$.  There exists $C_s<\infty$ such that for every
	$u_0\in H^{s}(\R)$ and $f,g\in H^{1/2}(\R)$,
	\begin{equation}
		|\langle u_0,\overline f g\rangle|
		\le C_s\|u_0\|_{H^{s}}
		\|f\|_{H^{1/2}}\|g\|_{H^{1/2}}.
		\label{eq:product-form-bound}
	\end{equation}
	Moreover, for every $\eta>0$ there exist $C_{\eta,s}$ and an integer
	$m=m(s)$ such that
	\begin{equation}
		|\langle u_0,|f|^2\rangle|
		\le \eta\||D|^{1/2}f\|_2^2
		+C_{\eta,s}(1+\|u_0\|_{H^{s}})^m\|f\|_2^2.
		\label{eq:infinitesimal-form-bound}
	\end{equation}
\end{prop}

\begin{proof}[Proof outline]
Choose $\frac14-\frac{s}{2}<\theta<\frac12$.  The multiplication
estimate \cite[Theorem~5.1]{BHnew} gives
$\|\overline f g\|_{H^{-s}}\lesssim_s
\|f\|_{H^\theta}\|g\|_{H^\theta}$, so duality proves
\eqref{eq:product-form-bound}.  For $g=f$, interpolate between $L^2$
and $H^{1/2}$ and apply Young's inequality to obtain
\eqref{eq:infinitesimal-form-bound}.
\end{proof}

\begin{prop}[Multiplier bounds for rough potentials]
	\label{lem:rough-multiplier}
	Let $-\frac12<s<0$ and $u_0\in H^{s}(\R)$.  If ${r_{\rm m}}>\frac12$ and
	$f\in H^{r_{\rm m}}(\R)$, then
	\begin{equation}
		\|u_0 f\|_{H^{s}}
		\le C_{s,{r_{\rm m}}}\|u_0\|_{H^{s}}\|f\|_{H^{r_{\rm m}}}.
		\label{eq:rough-multiplier-above-half}
	\end{equation}
	If $0<\delta<\frac12+s$ and $f\in H^{1/2}(\R)$, then
	\begin{equation}
		\|u_0 f\|_{H^{s-\delta}}
		\le C_{s,\delta}\|u_0\|_{H^{s}}\|f\|_{H^{1/2}}.
		\label{eq:rough-multiplier-bootstrap}
	\end{equation}
	Both products are defined distributionally.
\end{prop}

\begin{proof}[Proof outline]
Apply \cite[Theorem~5.1]{BHnew} to the dual products
\[
H^{{r_{\rm m}}}H^{-s}\to H^{-s},\qquad
H^{1/2}H^{-s+\delta}\to H^{-s}.
\]
The stated restrictions make both
multiplications continuous.  Sobolev duality then defines $u_0f$
and gives the two asserted bounds.
\end{proof}

Write $T_{u_0}f:=\Pi(u_0 f)$, using the distributional product above.
For real $u_0\in H^{s}$ define on $H_+^{1/2}$ the form
$\ell_{u_0}[g,f]:=\langle D^{1/2}g,D^{1/2}f\rangle
-\langle u_0,\overline g f\rangle$.

We identify the operator domain of the realization in \cite[Proposition~3.2]{KLV}.
\begin{prop}[Self-adjoint realization of the rough Lax operator,
{based on \cite[Theorem~2.3]{BKV} and \cite[Proposition~3.2]{KLV}}]
	\label{prop:closed-Lax}
	Let $u_0\in H^{s}(\R;\R)$ with $-\frac12<s<0$.  The form is closed and bounded from below.  Its self-adjoint
	operator $L_{u_0}$ has form domain
	$\operatorname{Dom}((L_{u_0}+\ka)^{1/2})=H_+^{1/2}$ and operator domain
	\begin{equation}
		\operatorname{Dom}(L_{u_0})
		=\{f\in H_+^{1+s}:Df-\Pi(u_0 f)\in L_+^2\},
		\label{eq:operator-domain}
	\end{equation}
	with $L_{u_0}f=Df-\Pi(u_0 f)$.
	On every ball $\|u_0\|_{H^{s}}\le R$, the shift $\ka=\ka(R)$ can be
	chosen so that
	\begin{equation}
		c_R\|f\|_{H^{1/2}}^2
		\le \ell_{u_0}[f,f]+\ka\|f\|_2^2
		\le C_R\|f\|_{H^{1/2}}^2,
		\enspace f\in H_+^{1/2}.
		\label{eq:uniform-coercivity}
	\end{equation}
\end{prop}

\begin{proof}[Proof outline]
The infinitesimal form bound and the KLMN theorem give a closed,
semibounded form, with \eqref{eq:uniform-coercivity} uniform on each
$H^s$ ball.  The representation theorem identifies its operator by
$h=Df-\Pi(u_0f)$ in distributions.  Indeed, testing away from zero
frequency leaves only a distribution supported at zero; its
$H^{-1/2}$ regularity forces it to vanish.

To identify the domain, fix $0<\delta<\frac12+s$ and apply
Proposition~\ref{lem:rough-multiplier} first at $H^{1/2}$ and then
above $H^{1/2}$.  The equation for $h$ successively gives
$f\in H^{1+s-\delta}$ and $f\in H^{1+s}$.  Conversely, an
$f$ in the set \eqref{eq:operator-domain} satisfies the form
representation identity, completing the domain characterization.
\end{proof}

For a positive self-adjoint operator $A\ge1$, write
$\mathcal H_A^r=\operatorname{Dom}(A^r)$ with norm $\|A^r\cdot\|_2$ when
$r\ge0$, and use the corresponding pivot-space completion when $r<0$.
For real $p$ in a fixed $H^s$ ball, let
$\widetilde A_p:H_+^{1/2}\to H_+^{-1/2}$ be the form operator associated with
$\ell_p+\ka\langle\cdot,\cdot\rangle$, using a common coercive shift on that ball.
Uniform coercivity, the form resolvent identity, and
Proposition~\ref{lem:product-form} give
$\|\widetilde A_p^{-1}-\widetilde A_{u_0}^{-1}\|_{H_+^{-1/2}\to H_+^{1/2}}
\lesssim_R\|p-u_0\|_{H^s}$.

\begin{prop}[Uniform rough elliptic estimate]
	\label{lem:rough-elliptic}
	Given $R,B>0$, there exists $C=C(s,R,B)$ such that whenever
	$\|u_0\|_{H^{s}}\le R$, $|z|\le B$, and $f\in\operatorname{Dom}(L_{u_0})$,
	\begin{equation}
		\|f\|_{H^{1+s}}
		\le C\bigl(\|(L_{u_0}-z)f\|_{H^{s}}+\|f\|_2\bigr).
		\label{eq:rough-elliptic}
	\end{equation}
\end{prop}

\begin{proof}[Proof outline]
Choose $\frac12<{r_{\rm m}}<1+s$.  The rough multiplier bound and
interpolation bound $\|\Pi(u_0f)\|_{H^s}$ by an arbitrarily small
multiple of $\|f\|_{H^{1+s}}$, up to a controlled $L^2$ term.  Insert this into
the free elliptic estimate for $D-z$ and absorb the small term.
The constants are uniform for the prescribed bounds on $u_0$ and $z$.
\end{proof}

We extend the shifted estimates \cite[(2.18)--(2.19)]{BKV} to
compact subsets of the resolvent set.
\begin{prop}[Sobolev-scale resolvent stability]
	\label{prop:cross-scale-resolvent}
	Suppose $u_{0,n}\to u_0$ in $H^{s}$ and $K\Subset\operatorname{res}(L_{u_0})$ is compact.  Then
	$K\subset\operatorname{res}(L_{u_{0,n}})$ for all sufficiently large $n$, and for every
	$-\frac12\le r\le s$,
	\begin{equation}
		\sup_{z\in K}
		\|(L_{u_{0,n}}-z)^{-1}-(L_{u_0}-z)^{-1}\|_{H_+^r\to H_+^{r+1}}
		\le C_K\|u_{0,n}-u_0\|_{H^{s}}.
		\label{eq:cross-scale-resolvent}
	\end{equation}
	In particular $L_{u_{0,n}}\to L_{u_0}$ in norm-resolvent sense.
\end{prop}

\begin{proof}[Proof outline]
Use a common positive shift.  The shifted estimates
\cite[(2.18)--(2.19)]{BKV} give norm-resolvent convergence;
compactness of $K$ gives a uniform $L^2$ resolvent bound.
The rough elliptic estimate and duality upgrade this bound to
$H^s\to H^{1+s}$.  The identity
\begin{equation}
 (L_{u_{0,n}}-z)^{-1}-(L_{u_0}-z)^{-1}
 =(L_{u_{0,n}}-z)^{-1}T_{u_{0,n}-u_0}(L_{u_0}-z)^{-1}
 \label{eq:rigged-resolvent-identity}
\end{equation}
and the multiplier estimate prove the Lipschitz bound at $r=s$.
At $r=-\frac12$, use spectral calculus on the adapted form scales
and the form bound for $T_{u_{0,n}-u_0}$ in the same identity.
Complex interpolation supplies the intervening exponents.
\end{proof}

\begin{prop}[Near-real resolvent bounds]
	\label{lem:near-real-resolvent}
	For every $R,B>0$ there exists $C=C(R,B,s)$ such that if
	$\|p\|_{H^{s}}+\|u_0\|_{H^{s}}\le R$, $|\Re z|\le B$, and
	$0<|\Im z|\le1$, then
	\begin{align}
		\|(L_{u_0}-z)^{-1}\|_{H_+^{s}\to H_+^{1+s}}
		&\le C|\Im z|^{-1},                                    \label{eq:near-real-one}\\
		\|(L_p-z)^{-1}-(L_{u_0}-z)^{-1}\|_{H_+^{s}\to H_+^{1+s}}
		&\le C\|p-u_0\|_{H^{s}}|\Im z|^{-2}.                 \label{eq:near-real-two}
	\end{align}
\end{prop}

\begin{proof}[Proof outline]
Start with the spectral $L^2$ bound $|\Im z|^{-1}$.
The rough elliptic estimate, duality, and a second elliptic estimate
give \eqref{eq:near-real-one}.  Applying that bound to the two
resolvents in \eqref{eq:rigged-resolvent-identity}, together with
the rough multiplier estimate, gives \eqref{eq:near-real-two}.
\end{proof}

As a consequence, if $F\in C_c^\infty(\R)$ and
$\|p\|_{H^s}+\|q\|_{H^s}\le R$, then
\begin{equation}
	\|F(L_p)-F(L_q)\|_{H_+^s\to H_+^{1+s}}
	\le C_{F,R}\|p-q\|_{H^s}.
	\label{eq:smooth-functional-calculus-Lipschitz}
\end{equation}
This follows from the Helffer--Sj\"ostrand formula with an almost-analytic
extension of order $N\ge3$ and \eqref{eq:near-real-two}.

\begin{prop}[Stability of spectral projections and covariance]
	\label{prop:stable-spectral-data}
	For every real $u_0\in H^{s}$,
	\begin{equation}
		\sigma_{\mathrm{ess}}(L_{u_0})=[0,\infty).
		\label{eq:essential-spectrum}
	\end{equation}
	If a fixed positively oriented rectifiable Jordan curve $\Gamma$, or a
	finite positively oriented cycle of such curves, lies in $\operatorname{res}(L_{u_0})$, then
	for all nearby $p$ it lies in $\operatorname{res}(L_p)$, the Riesz projections
	$P_\Gamma(p)=(2\pi i)^{-1}\int_\Gamma(z-L_p)^{-1}\,\rmd z$
	are well defined, and the map $p\mapsto P_\Gamma(p)$ is continuous in every
	cross-scale operator norm appearing in
	Proposition~\ref{prop:cross-scale-resolvent}.  Moreover, if
	$(\tau_yf)(x)=f(x-y)$, then
	\begin{equation}
		L_{\tau_y u_0}=\tau_yL_{u_0}\tau_y^*,
		\enspace E_{\tau_y u_0}(B)=\tau_yE_{u_0}(B)\tau_y^*.
		\label{eq:translation-covariance}
	\end{equation}
\end{prop}

\begin{proof}[Proof outline]
The essential-spectrum assertion is \cite[Proposition~3.2]{KLV}.
Integrate Proposition~\ref{prop:cross-scale-resolvent} around $\Gamma$
for projection continuity.  Translation covariance holds for the
quadratic forms by change of variables, hence for their represented
operators and spectral projections.
\end{proof}

We use the global BO flow of \cite[Theorem~1.1]{KLV} for
$s>-\frac12$.  For $-\frac12<s<0$, smooth approximation is provided
by \cite[Theorem~5.5]{KLV}, and orbit bounds and preservation of
$H^s$-equicontinuity by \cite[Lemma~4.4 and Corollary~5.3]{KLV},
as recalled in \cite[Proposition~2.4]{BKV}.
Here a bounded family $\mathcal U\subset H^s$ is
$H^s$-equicontinuous if
$\lim_{M\to\infty}\sup_{v\in\mathcal U}\|P_{>M}v\|_{H^s}=0$.

For real $u\in H^s(\mathbb R)$, $s>-\frac12$, we use the Wu relations of
Badreddine, Killip, and Vi\c{s}an \cite[Theorem~1.4]{BKV}, including
simplicity, exclusion of nonnegative eigenvalues, and Fourier boundary
continuity.

For real $u\in H^s(\mathbb R)$, $-\frac12<s<0$, and an admissible
$\kappa$, let $\mu_u$ be the Stieltjes measure of Badreddine, Killip,
and Vi\c{s}an \cite[Proposition~7.1]{BKV}.  With
$w=(L_u+\kappa)^{-1/2}u_+$ and $\nu$ its finite spectral measure,
our normalization is $\rmd\mu_u(E)=(E+\kappa)\,\rmd\nu(E)$ and
\[
 \beta(z;u)=\langle u_+,(L_u+z)^{-1}u_+\rangle
 =\int_{\mathbb R}\frac{\rmd\mu_u(E)}{E+z},
 \enspace z\notin\sigma(-L_u).
\]
Here the pairing is Sobolev duality.  For $u\in L^2$, $\mu_u$ is
the spectral measure of $u_+$.  Conservation and local Lipschitz
continuity of $\beta(\kappa;u)$ are established in
\cite[discussion following (2.20)]{BKV}.

For $-\frac12<s<0$, the finite-multisoliton variational characterization
of Badreddine, Killip, and Vi\c{s}an \cite[Theorem~1.5]{BKV}
applies to real $H^s$ potentials with prescribed distinct negative
eigenvalues and an admissible shift.  We use its one-soliton equality
case below.

\begin{prop}[Global orbit control,
{based on \cite[Theorem~2.3 and Proposition~2.4]{BKV}
and \cite[Theorems~1.1 and~5.5]{KLV}}]
	\label{prop:global-orbit-control}
	For the solution $u$ of \eqref{eq:BO}, there exist real Schwartz data
	$u_{0,n}\to u_0$ and corresponding smooth solutions $u_n$ such that
	\begin{align}
		\sup_{t\in\R}\|u(t)\|_{H^{s}}
		+\sup_n\sup_{t\in\R}\|u_n(t)\|_{H^{s}}&<\infty,
		\label{eq:uniform-orbit-bound}\\
		\lim_{M\to\infty}\sup_{t\in\R}
		\|P_{>M}u(t)\|_{H^{s}}&=0,                            \label{eq:uniform-orbit-tail}\\
		\sup_{|t|\le T}\|u_n(t)-u(t)\|_{H^{s}}&\longrightarrow0
		\enspace(T<\infty).                                        \label{eq:smooth-flow-approximation}
	\end{align}
	If $\langle x\rangle^\be u_0\in H^{s}$ with $\be>\frac12$, the same
	sequence may be chosen so that
	\begin{equation}
		\|u_{0,n}-u_0\|_{H^{s}}
		+\|\langle x\rangle^\be(u_{0,n}-u_0)\|_{H^{s}}\longrightarrow0.
		\label{eq:weighted-smooth-approximation}
	\end{equation}
	Choose one common $\ka$ sufficiently large, once and for all, so that the
	operators $A_t=L_{u(t)}+\ka$ and $A_{n,t}=L_{u_n(t)}+\ka$ are at least
	one and satisfy, uniformly in $n,t$, for every $s\le r\le1+s$ and
	$f\in H_+^r$,
	\begin{equation}
		C^{-1}\|f\|_{H^r}\le\|A_t^rf\|_2\le C\|f\|_{H^r},\enspace
		C^{-1}\|f\|_{H^r}\le\|A_{n,t}^rf\|_2\le C\|f\|_{H^r}.
		\label{eq:uniform-fractional-scales}
	\end{equation}
	For $r<0$, the powers in \eqref{eq:uniform-fractional-scales} denote their
	canonical extensions from the completed physical $H_+^r$ scale to $L_+^2$.
	The choice is also made large enough to dominate the uniform positivity
	threshold for every potential in the orbit ball, to satisfy the
	admissibility threshold in \cite[Theorem~1.5]{BKV} on that ball, and to
	ensure $\|(D+\ka)^{-1/2}T_{u_0}(D+\ka)^{-1/2}\|<1$.
	All $\ka$-dependent objects below use this single value.
\end{prop}

\begin{proof}[Proof outline]
Approximate $\langle x\rangle^\be u_0$ by real smooth functions
with compact support.  Multiplication by $\langle x\rangle^{-\be}$
gives a family bounded and equicontinuous in $H^s$.  The cited BO flow results give the uniform
orbit bounds, tails, and convergence on compact time intervals.

Choose a common coercive shift on this orbit ball.  Form comparison,
duality, and interpolation identify the scales for
$-\frac12\le r\le\frac12$.  At the remaining endpoint, the product
estimate and \cite[(2.18)]{BKV} identify
$A_p:H_+^{1+s}\to H_+^s$ as an isomorphism, uniformly for
$p=u(t),u_n(t)$.  Its inverse agrees with the inverse on the abstract
scale, giving $\mathcal H_{A_p}^{1+s}=H_+^{1+s}$; interpolation
proves \eqref{eq:uniform-fractional-scales}.
Finally, enlarge the shift to meet the common admissibility
thresholds and use the infinitesimal form bound to make the
conjugated perturbation smaller than one.
\end{proof}

\begin{prop}[Badreddine, Killip, and Vi\c{s}an
{\cite[Proposition~3.5 and Corollary~3.6]{BKVasymp}};
Gassot and G\'erard {\cite[Proposition~3.1]{GG}}]
\label{prop:prep-L2-graph}
Let $A>0$, $v,w\in L^2(\R;\R)$, and $z\in\C$ satisfy
$\|v\|_2+\|w\|_2+|z|<A$ and
$\min\{\operatorname{dist}(z,\sigma(L_v)),
\operatorname{dist}(z,\sigma(L_w))\}>A^{-1}$.
For every $h\in\operatorname{Dom}(X^*)$,
\begin{align*}
 \|X^*(L_v-z)^{-1}h\|_{H^1}
 &\le C_A(\|h\|_2+\|X^*h\|_2),\\
 \|X^*[(L_v-z)^{-1}-(L_w-z)^{-1}]h\|_{H^1}
 &\le C_A\|v-w\|_2(\|h\|_2+\|X^*h\|_2).
\end{align*}
Every eigenfunction $f$ of $L_v$ belongs to
$\operatorname{Dom}(X^*)$, with $X^*f\in H_+^1$.
\end{prop}

Lemmas~\ref{lem:X-graph-resolvent} and
\ref{lem:isolated-graph-projection} extend these graph estimates and the
stability of isolated spectral projections to negative regularity.

For a real $L^2$ BO orbit $v(t)$, let
$U_v(t,\tau)=U_v(t,0)U_v(\tau,0)^*$ denote the unitary flow of
Badreddine, Killip, and Vi\c{s}an \cite[Proposition~3.9]{BKVasymp},
with Lax operator
$P_vh=-i\partial_x^2h-2\partial_x\Pi(vh)+2v_+'h$.
For $a=t-\tau$, their Lemma~4.1 identifies the closed operator
$C_{v,t,\tau}:=U_v(t,\tau)^*X^*U_v(t,\tau)
=\overline{X^*-2aL_{v(\tau)}}$ with domain
$\operatorname{Dom}(C_{v,t,\tau})
=\{h\in L^2_+:e^{ia\xi^2}\widehat h\in H^1(0,\infty)\}$.
In our notation, \cite[Lemma~4.2 and Proposition~4.3]{BKVasymp} give
\begin{equation}
 [U_v(t,\tau)h](z)
 =\frac1{2\pi i}I_+[(C_{v,t,\tau}-z)^{-1}h],\enspace
 |[U_v(t,\tau)h](z)|\le\frac{\|h\|_2}{\sqrt{4\pi\Im z}},
 \enspace h\in L^2_+,\ z\in\C_+.
 \label{eq:prep-generalized-explicit}
\end{equation}
Here the moving-resolvent trace means
$I_+q=\lim_{y\to\infty}\langle\chi_y,q\rangle$, with
$\chi_y(x)=iy/(x+iy)$.  Its existence follows from the Abel-trace
estimate in \cite[Lemma~4.2]{BKVasymp}.
Taking $h=v_+(\tau)$ recovers G\'erard's explicit formula.

For rough data below $L^2$, Section~\ref{sec:rough-Gerard} constructs
the cyclic transport $U_{u_0}^{\mathrm{cyc}}(t,\tau)$ on the cyclic
spectral subspace.

For real $v$ with $v,xv\in L^2(\R)$ and $\lambda>0$, we use the
classical Jost normalizations and distorted coefficients of \cite[(2.4)--(2.5)]{GGM}:
$(L_v-\lambda)m_\pm=0$,
$\lim_{x\to\pm\infty}e^{-i\lambda x}m_\pm(x,\lambda)=1$, and
$\widetilde f^{\,\pm}(\lambda)
=\int_\R f(x)\overline{m_\pm(x,\lambda)}\,\rmd x$.
The integral is initially defined for $f\in L_+^1\cap L_+^2$;
its $L_+^2$ extension and spectral normalization are due to
Blackstone, Gassot, G\'erard, and Miller \cite{BGGM}; see also \cite[(2.7)]{GGM}.

With our Fourier normalization, the transform unitary on the
absolutely continuous subspace is
$\mathcal F_v^\pm f=c_{\rm F}\widetilde f^{\,\pm}$,
where $c_{\rm F}=(2\pi)^{-1/2}$.  The generalized columns for the
resolvent boundary values $\lambda\pm i0$ are $c_{\rm F}m_\mp$.
Their construction for distributional potentials is in
Section~\ref{sec:trace-transform}.

In our unitary Fourier convention, the classical real radiation data
of \cite[(2.6)]{GGM} are specified by
\begin{equation}
 \widehat u_\infty^\pm(\lambda)
 =c_{\rm F}\widetilde{v_+}^{\,\mp}(\lambda),\enspace
 \widehat u_\infty^\pm(-\lambda)
 =\overline{\widehat u_\infty^\pm(\lambda)},\enspace\lambda>0.
 \label{eq:prep-classical-radiation-data}
\end{equation}
We fix the unimodular radiation multiplier by
\begin{equation}
 m_+=\ell_vm_-,\enspace
 \widetilde f^{\,-}=\ell_v\widetilde f^{\,+},\enspace
 \widehat u_\infty^+=\ell_v\widehat u_\infty^-.
 \label{eq:prep-Jost-phase}
\end{equation}
Thus $\Gamma_v:=\ell_v^{-1}
=\lim_{x\to+\infty}e^{-i\lambda x}m_-(x,\lambda)$ is the
transmission coefficient; the last identity in
\eqref{eq:prep-Jost-phase} uses the radiation convention of
\cite[(2.9)]{GGM}.
For $v\in\mathcal S(\R;\R)$, set
$m_1=1+(D-\lambda-i0)^{-1}T_vm_1$ and
\begin{equation}
 \beta_v(\lambda):=i\int_\R v(x)m_1(x,\lambda)e^{-i\lambda x}\,\rmd x,
 \label{eq:prep-reflection-coefficient}
\end{equation}
as in \cite[Lemma~4.2, (4.26)]{WuDirect}.
For $v=u_{0,n}$ we write $\beta_n$.

For the smooth approximants used in Section~\ref{sec:uniform-sewing},
we also set
\[
 \varkappa_v(a):=-\frac1{2\pi a}\int_\R v(x)m_-(x,a)\,\rmd x,
 \enspace a>0.
\]
In Wu's notation, our $m_-$ and $m_1$ are respectively
$m_e(\cdot,a+0i)$ and $m_1(\cdot,a+0i)$.
The coefficient and differential identities are taken directly from
\cite[Lemmas~4.1, 4.2 and~4.4]{WuDirect}; the proof of the estimates
uniform under rough approximation is outlined in
Section~\ref{sec:uniform-sewing}.

\section{A conserved spectral measure and evolution on cyclic subspaces}
\label{sec:rough-Gerard}

\subsection{Hardy evaluation and cyclic spectral evolution}

\begin{lemma}[Hardy evaluation at negative regularity]
	\label{lem:Hardy-evaluation}
	For $z\in\C_+$ and $h\in H_+^{s}$ define
	\begin{equation}
		\mathcal E_zh=(2\pi)^{-1/2}
		\int_0^\infty e^{iz\xi}\widehat h(\xi)\,\rmd\xi.
		\label{eq:Hardy-evaluation}
	\end{equation}
	Then $z\mapsto\mathcal E_zh$ is holomorphic.  If $K\Subset\C_+$ is compact and
	$\ell\in\{0,1,2,\ldots\}$, then
	\begin{equation}
		\sup_{z\in K}|\partial_z^\ell\mathcal E_zh|
		\le C_{K,s,\ell}\|h\|_{H^{s}}.
		\label{eq:Hardy-evaluation-bound}
	\end{equation}
	For $h\in L_+^2$, this definition agrees with the usual analytic Hardy
	extension.
\end{lemma}

\begin{proof}[Proof outline]
Cauchy--Schwarz in the Fourier integral reduces the derivative bound
to integrability of
$\xi^{2\ell}\langle\xi\rangle^{-2s}e^{-2\Im z\,\xi}$.
This is uniform on $K$ and also justifies differentiation under the
integral.  At $L_+^2$ regularity the integral is the usual Hardy extension.
\end{proof}

The functional $I_+$ is continuous on $\operatorname{Dom}(X^*)$
equipped with the graph norm.
Integration by parts gives, for $g\in\operatorname{Dom}(X^*)$,
\begin{equation}
	|I_+g|^2=-4\pi\Im\langle g,X^*g\rangle
	\le2\pi\bigl(\|g\|_2^2+\|X^*g\|_2^2\bigr).
	\label{eq:graph-trace}
\end{equation}
For sources in $H_+^s$, we take limits of smooth resolvent traces,
identified by $\mathcal E_z$; on $\operatorname{Dom}(X^*)$, we use the graph trace.

Use the shift in Proposition~\ref{prop:global-orbit-control}, put
$a(E)=E+\ka$, and set
$w_t=A_t^{-1/2}u_+(t)\in L_+^2$ and
$m_t=A_t^{-1}u_+(t)\in H_+^{1+s}$.
The first expression uses the form-scale extension
$A_t^{-1/2}:H_+^{-1/2}\to L_+^2$.  Let $\nu_t$ be the finite spectral measure of
$w_t$ for $L_{u(t)}$, and define
\begin{equation}
	\rmd\mu_{u(t)}(E)=(E+\ka)\,\rmd\nu_t(E).
	\label{eq:rho-measure}
\end{equation}
This is the positive locally finite Stieltjes measure of
\cite[Proposition~7.1]{BKV}; its continuous restriction is
$\mu_{u_0}^{\rm c}=\mu_{u_0}|_{[0,\infty)}$.
\begin{lemma}[Conservation and uniqueness of the Stieltjes measure]
	\label{lem:Stieltjes-conservation}
	The measure $\mu_{u(t)}$ is independent of $t$, and so is
	$\nu_t$; write $\nu=\nu_t$.  The finite spectral measure of $m_t$ is
	\begin{equation}
		(E+\ka)^{-1}\,\rmd\nu(E)
		=(E+\ka)^{-2}\,\rmd\mu_{u(t)}(E),
		\label{eq:nu-measure}
	\end{equation}
	and is therefore independent of $t$ as well.
\end{lemma}

\begin{proof}[Proof outline]
On the common orbit ball, the conservation law in
\cite[Proposition~7.1 and the discussion following (2.20)]{BKV} gives,
for every sufficiently large $\eta$,
\begin{equation}
 \beta(\eta;u(t))
 =\int_\R\frac{\rmd\mu_{u(t)}(E)}{E+\eta}
 =\beta(\eta;u_0).
 \label{eq:beta-Stieltjes}
\end{equation}
Choose a fixed $\eta_0$ above the common spectral threshold.
The finite measures $(E+\eta_0)^{-1}\rmd\mu_{u(t)}$ have identical
Stieltjes transforms by this identity.  Analytic continuation and
Stieltjes inversion therefore give conservation of $\mu_{u(t)}$.
The relations defining $\nu_t$ and $m_t=A_t^{-1/2}w_t$ then give
conservation of $\nu_t$ and \eqref{eq:nu-measure}.
\end{proof}

For $r\in\R$, set
$L_r^2:=L^2\!\left(\R,a(E)^{2r-1}\,\rmd\nu(E)\right)$.
Initially, for bounded compactly supported $F$, set
$\mathscr F_tF=F(L_{u(t)})m_t$.
The spectral theorem extends $\mathscr F_t$ to an onto isometry from
$L_0^2$ onto the cyclic space generated by $m_t$.  Define $\mathcal K_{\ka,t}^r:=\mathscr F_tL_r^2$ with norm
$\|\mathscr F_tF\|_{\mathcal K_{\ka,t}^r}=\|F\|_{L_r^2}$.
For $r<0$, $\mathscr F_t$ is extended by completion.  This is the
closed reducing subscale generated by $m_t$ inside the scale of $A_t$.

\begin{prop}[Evolution on the cyclic spectral subspace]
	\label{prop:cyclic-transport}
	The formula
	\begin{equation}
		 U_{u_0}^{\mathrm{cyc}}(t,\tau)=\mathscr F_t\mathscr F_\tau^{-1}
		\label{eq:cyclic-transport}
	\end{equation}
	defines an onto isometry
	$\mathcal K_{\ka,\tau}^r\to\mathcal K_{\ka,t}^r$ for every $r\in\R$.
	It satisfies the cocycle law and
	\begin{equation}
		 U_{u_0}^{\mathrm{cyc}}(t,\tau)G(L_{u(\tau)})
		=G(L_{u(t)}) U_{u_0}^{\mathrm{cyc}}(t,\tau)
		\label{eq:all-multiplier-intertwining}
	\end{equation}
	for every Borel multiplier $G$ on the maximal domain in the spectral model,
	$\mathcal D_r(G)=\{F\in L_r^2:GF\in L_r^2\}$.
	In particular, both sides below map
	$\mathcal K_{\ka,\tau}^{1+s}$ into
	$\mathcal K_{\ka,t}^{s}$, and
	\begin{equation}
		A_t U_{u_0}^{\mathrm{cyc}}(t,\tau)= U_{u_0}^{\mathrm{cyc}}(t,\tau)A_\tau.
		\label{eq:strong-graph-intertwining}
	\end{equation}
	For $s\le r\le1+s$ and fixed $H\in L_r^2$, the map
	$t\mapsto\mathscr F_tH$ is strongly continuous in the physical $H_+^r$
	topology.
\end{prop}

\begin{proof}[Proof outline]
In the common spectral model the transport is the identity and
$G(L_{u(t)})$ is multiplication by $G(E)$.  This gives the isometry,
cocycle, and both intertwining identities on the stated domains.

For continuity, resolvent stability and flow continuity imply
$m_t\to m_{t_0}$ in $H^{1+s}$.  The functional-calculus bound
\eqref{eq:smooth-functional-calculus-Lipschitz} gives the same
continuity for $\mathscr F_tF$ with $F\in C_c^\infty$.
Approximate a fixed $H\in L_r^2$ by such $F$ and use the uniform
scale equivalence to pass to the limit in $H^r$.
\end{proof}

\begin{prop}[Spectral moment of the potential]
	\label{prop:potential-spectral-moment}
	For every $t$,
	\begin{equation}
		u_+(t)=A_tm_t\in\mathcal K_{\ka,t}^{s},
		\enspace
		 U_{u_0}^{\mathrm{cyc}}(t,\tau)u_+(\tau)=u_+(t).
		\label{eq:potential-in-cyclic-scale}
	\end{equation}
	More precisely,
	$\|u_+(t)\|_{\mathcal K_{\ka,t}^{s}}^2
	=\int_\R a(E)^{2s}\,\rmd\mu_{u_0}(E)<\infty$.
\end{prop}

\begin{proof}[Proof outline]
The Mellin representation of $x^{2s}$ and Tonelli's theorem express
the spectral moment through
$\int_0^\infty r^{2s}\beta(\ka+r;u(t))\,\rmd r$.
Form comparison with $D+\ka+r$ and order reversal of positive
inverses bound this by $C\|u(t)\|_{H_\kappa^s}^2$.
For rough sources, the comparison is first applied to Fourier
truncations and then passed to the limit in form duality.

Thus $a(E)$ belongs to $L_s^2$, with squared norm
$\int a(E)^{2s}\,\rmd\mu_{u_0}(E)$.
Apply bounded spectral cutoffs to $m_t$, pass to the limit in
$\mathcal K_{\ka,t}^{1+s}$, and then apply
$A_t:H_+^{1+s}\to H_+^s$.  This identifies the model representative
$a(E)$ with $A_tm_t=u_+(t)$.  Its independence of $t$ proves
\eqref{eq:potential-in-cyclic-scale}.
\end{proof}

For the smooth orbit $u_n$, set
$m_{n,t}=(L_{u_n(t)}+\ka)^{-1}(u_n(t))_+$.
Let $U_n(t,\tau)$ be the auxiliary unitary propagator of $u_n$
in \cite[Proposition~3.9]{BKVasymp}.

\begin{lemma}[Strong convergence of compact spectral kernels]
	\label{lem:compact-spectral-kernel-convergence}
	If $F\in C_c^\infty(\R)$ and $I\Subset\R$ is compact, then
	\begin{equation}
		\sup_{t\in I}
		\|F(L_{u_n(t)})m_{n,t}-F(L_{u(t)})m_t\|_{H^{1+s}}
		\longrightarrow0.
		\label{eq:compact-spectral-kernel-convergence}
	\end{equation}
	Moreover,
	\begin{equation}
		U_n(t,\tau)F(L_{u_n(\tau)})m_{n,\tau}
		=F(L_{u_n(t)})m_{n,t}.
		\label{eq:smooth-kernel-transport}
	\end{equation}
\end{lemma}

\begin{proof}[Proof outline]
Combine the functional-calculus Lipschitz bound with smooth flow
convergence and shifted resolvent stability.  Both factors in
$F(L_{u_n(t)})m_{n,t}$ then converge uniformly on $I$, giving
\eqref{eq:compact-spectral-kernel-convergence}.
The smooth Lax transport identity for $m_{n,t}$ and its spectral
multipliers, \cite[Proposition~3.9(ii)--(iii)]{BKVasymp}, gives
\eqref{eq:smooth-kernel-transport}.
\end{proof}

Applying \cite[Lemma~4.1]{BKVasymp} to the smooth orbit
$r\mapsto u_n(\tau+r)$ gives
\[
C_{n,t,\tau}:=U_n(t,\tau)^*X^*U_n(t,\tau)
=\overline{X^*-2(t-\tau)L_{u_n(\tau)}},
\]
with
$\operatorname{Dom}(C_{n,t,\tau})=U_n(t,\tau)^*\operatorname{Dom}(X^*)$.
The Abel trace and analytic evaluation on
$(C_{n,t,\tau}-z)^{-1}L_+^2$ are those of
\eqref{eq:prep-generalized-explicit}.

\subsection{Proof of Theorem~\ref{thm:main-gerard}: the extension of G\'erard's explicit formula}

\begin{proof}[Outline of the proof of Theorem~\ref{thm:main-gerard}]
Reduce any initial regularity $s_*>-\frac12$ to
$-\frac12<s<\min\{s_*,0\}$ using uniqueness of the BO flow.
For each Schwartz approximant, the closed smooth formula
\eqref{eq:prep-generalized-explicit} identifies the resolvent trace
in \eqref{eq:rough-Gerard-main} with
$\mathcal E_z((u_n(t))_+)$.

Flow-map continuity gives $u_n\to u$ in $C([-T,T];H^s)$.
The Hardy evaluation bound then passes this identity to the limit,
uniformly for $z$ in compact subsets of $\C_+$ and for every
fixed derivative in $z$.  The smooth right side is independent of
$\tau$, so the convergence is uniform in $\tau$ as well.
The limit is $u_+(t,z)$ and is independent of the approximating sequence.
\end{proof}

\section{Sobolev range identities and asymptotics along fixed rays}
\label{sec:rigged-range}

\subsection{Range identities and decay along fixed rays}

Use $\mathcal D_+^\circ=\mathcal F^{-1}C_c^\infty((0,\infty))$
from Section~\ref{subsec:notation-guide}.
The $H_+^{-s}$--$H_+^s$ pairing extends the $L^2$ inner product.
By \eqref{eq:rough-multiplier-above-half}, throughout this section
$L_{u_0}:H_+^{1+s}\to H_+^s$ denotes the bounded distributional
realization $L_{u_0}f=Df-\Pi(u_0f)$, agreeing with the self-adjoint
operator on $\operatorname{Dom}(L_{u_0})$.  Since
$\operatorname{Dom}(L_{u_0})\subset H_+^{1+s}\hookrightarrow H_+^{-s}$ for
$s>-\frac12$, every $L^2$-eigenvector defines an element of the dual Sobolev space.

\begin{thm}[Range identity in the Sobolev space]
	\label{thm:rigged-range}
	For every $\lambda\in\R$, let
	$\mathscr E_\lambda=\operatorname{Ker}_{L_+^2}(L_{u_0}-\lambda)$.  Then
	\begin{equation}
		\overline{(L_{u_0}-\lambda)\mathcal D_+^\circ}^{\,H_+^s}
		=\{F\in H_+^s:\langle e,F\rangle_{H_+^{-s},H_+^s}=0
		\text{ for every }e\in\mathscr E_\lambda\}.
		\label{eq:rigged-range-identity}
	\end{equation}
\end{thm}

\begin{proof}[Proof outline]
Density of $\mathcal D_+^\circ$ in $H_+^{1+s}$ reduces the assertion to
identifying the annihilator of $(L_{u_0}-\lambda)H_+^{1+s}$ in $H_+^{-s}$.
Choose a real $\zeta_0$ below the spectrum.  Testing an annihilating vector
$\psi$ against $(L_{u_0}-\zeta_0)^{-1}a$, and using
Proposition~\ref{prop:cross-scale-resolvent} and resolvent symmetry in
Sobolev duality, gives
\[
\psi=(\lambda-\zeta_0)(L_{u_0}-\zeta_0)^{-1}\psi.
\]
Thus $\psi$ belongs to $\operatorname{Dom}(L_{u_0})$ and is a
$\lambda$-eigenvector.  Form symmetry proves the converse inclusion.
The Hahn--Banach double-annihilator theorem yields
\eqref{eq:rigged-range-identity}.
\end{proof}

For a smooth approximating potential and a fixed real $\lambda$, use
\[
C_{n,t,0}+2t\lambda=\overline{X^*-2t(L_{u_{0,n}}-\lambda)},
\]
where
$C_{n,t,0}$ is the operator from Section~\ref{sec:rough-Gerard} and
$\operatorname{Dom}(C_{n,t,0})
=\{h\in L_+^2:U_n(t,0)h\in\operatorname{Dom}(X^*)\}$.
This is the maximal dissipative realization given by the closed-operator
conjugation identity for the smooth flow.  The real scalar shift $2t\lambda$ preserves its domain.  For
both signs of $t$ its resolvent is taken at $z\in\C_+$ and is
bounded by $(\Im z)^{-1}$.  Define
$\mathcal O_{t,\lambda,z}^{(n)}(F)
=(2\pi i)^{-1}I_+(C_{n,t,0}+2t\lambda-z)^{-1}F$ for $z\in\C_+$.
For $h\in\mathcal D_+^\circ$, use
\[
(C_{n,t,0}+2t\lambda-z)^{-1}(L_{u_{0,n}}-\lambda)h
	=\frac1{2t}(C_{n,t,0}+2t\lambda-z)^{-1}(X^*-z)h
	-\frac1{2t}h
\]
and $I_+h=0$, since $\widehat h$ vanishes near zero.  The smooth G\'erard
identity and the $L^2$ norm of the Hardy reproducing kernel give
\begin{equation}
	|\mathcal O_{t,\lambda,z}^{(n)}((L_{u_{0,n}}-\lambda)h)|
	\le \frac{\|(X^*-z)h\|_2}
	{2|t|\sqrt{4\pi\Im z}}.
	\label{eq:t-inverse-estimate}
\end{equation}

This extends the range argument of Badreddine, Killip, and Vi\c{s}an
\cite[proof of Proposition~5.2, (5.19) and (5.24)--(5.28)]{BKVasymp}
to negative regularity.
\begin{lemma}[Decay of the continuous range component]
	\label{lem:removal-complement}
	Let $F\in H_+^s$ satisfy
	$\langle e,F\rangle=0$ for every $e\in\mathscr E_\lambda$, and let
	$F_n\in L_+^2$ converge to $F$ in $H_+^s$.  Then, for every
		compact $K\Subset\C_+$,
	\begin{equation}
		\lim_{|t|\to\infty}\limsup_{n\to\infty}
		\sup_{z\in K}|\mathcal O_{t,\lambda,z}^{(n)}(F_n)|=0.
		\label{eq:complement-double-limit}
	\end{equation}
	Thus the observable tends to zero along $x=-2t\lambda$ in both time
	directions, with $z\in\C_+$.
\end{lemma}

\begin{proof}[Proof outline]
By Theorem~\ref{thm:rigged-range}, approximate $F$ in $H_+^s$ within
$\varepsilon$ by $(L_{u_0}-\lambda)h_\varepsilon$, with
$h_\varepsilon\in\mathcal D_+^\circ$.  Its smooth representatives
$G_{n,\varepsilon}=(L_{u_{0,n}}-\lambda)h_\varepsilon$ converge in $H_+^s$.
Hardy evaluation, the smooth intertwining identity, and the common
fractional scales give the uniform bound
\begin{equation}
 |\mathcal O_{t,\lambda,z}^{(n)}(H)|
 \le C_K'\|H\|_{H_+^s},\qquad z\in K,
 \label{eq:observable-Y-bound}
\end{equation}
since real translations do not change the evaluation norm.
After $n\to\infty$ at fixed $t$, the approximation error is at most
$C_K'\varepsilon$, whereas \eqref{eq:t-inverse-estimate} makes the range
term $O_\varepsilon(|t|^{-1})$.  Let $|t|\to\infty$, followed by
$\varepsilon\downarrow0$, to obtain \eqref{eq:complement-double-limit}.
\end{proof}

\subsection{The coordinate operator on an isolated eigenspace}

We extend the eigenfunction domain inclusion of Gassot and G\'erard
\cite[Proposition~3.1]{GG} and the resolvent graph estimates of
Badreddine, Killip, and Vi\c{s}an \cite[Proposition~3.5]{BKVasymp}
to negative regularity, then pass to isolated spectral projections
in the graph of $X^*$.

\begin{lemma}[$X^*$-graph resolvent continuity]
	\label{lem:X-graph-resolvent}
	Let $-\frac12<s<0$, let $u_{0,n},u_0$ be real-valued, and assume
	$u_{0,n}\to u_0$ in $H^s(\mathbb R)$.
	Let $K\Subset\operatorname{res}(L_{u_0})$ be compact, and let
	$h\in\mathcal D_+^\circ$.  Then
	\begin{align*}
		\sup_{z\in K}\Bigl\{
\|[(L_{u_{0,n}}-z)^{-1}-(L_{u_0}-z)^{-1}]h\|_{H^{1+s}}
&\\
{}+\|X^*[(L_{u_{0,n}}-z)^{-1}-(L_{u_0}-z)^{-1}]h\|_2
\Bigr\}&\longrightarrow0.
\end{align*}
Moreover, $(L_p-z)^{-1}h\in\operatorname{Dom}(X^*)$ for every real
$p\in H^s$ and $z\in\operatorname{res}(L_p)$, and its $X^*$ image lies in $H_+^{1+s}$.
Locally uniformly in $p$ and uniformly on $K$,
\begin{align}
\| (L_p-z)^{-1}h\|_{H^{1+s}}
+\|X^*(L_p-z)^{-1}h\|_{H^{1+s}}
&\le C_K(\|h\|_2+\|X^*h\|_2),\label{eq:X-graph-uniform-bound}\\
\sup_{z\in K}\|X^*[(L_p-z)^{-1}-(L_q-z)^{-1}]h\|_{H^{1+s}}
&\le C_K\|p-q\|_{H^s}(\|h\|_2+\|X^*h\|_2).
\label{eq:X-graph-Lipschitz}
\end{align}
Here $p,q$ range over a sufficiently small $H^s$ neighborhood of $u_0$
on which $K$ remains in the resolvent set.
\end{lemma}

\begin{proof}[Proof outline]
The $H^{1+s}$ convergence follows from
Proposition~\ref{prop:cross-scale-resolvent}.  For the coordinate estimate,
start with a smooth potential $p$ and $f_p=(L_p-z)^{-1}h$.
Differentiating the Fourier resolvent equation first proves
$f_p\in\operatorname{Dom}(X^*)$.  Integration by parts in its Toeplitz
term then gives
\[
 X^*f_p=(L_p-z)^{-1}
 \left[X^*h-if_p+\frac{i}{2\pi}I_+(f_p)p_+\right].
\]
The boundary term is controlled by the zero-frequency equation:
\[
 I_+(f_p)=-z^{-1}\int_0^\infty
 \widehat p(-\eta)\widehat f_p(\eta)\,\rmd\eta,
 \qquad |I_+(f_p)|\le C_K\|p\|_{H^s}\|f_p\|_{H^{-s}}.
\]
Here $K$ is separated from $[0,\infty)$, and $1+s>-s$ allows the
cross-scale resolvent estimate to control the pairing.
Comparing this formula for two smooth potentials controls their trace
difference by $C_K\|p-q\|_{H^s}\|h\|_2$.  Substitution into the coordinate
formula and the resolvent identity yield
\eqref{eq:X-graph-uniform-bound}--\eqref{eq:X-graph-Lipschitz}.

For rough $p$, take smooth $p_m\to p$ in $H^s$.  The preceding bounds make
$f_{p_m}$ and $X^*f_{p_m}$ converge in $H^{1+s}$, uniformly on $K$.
Closedness of $X^*$ identifies the limiting coordinate vector, and the
$H^1(0,\infty)$ trace theorem identifies the limiting boundary value.
Passing the uniform estimates through two such approximations proves the
assertions for rough potentials.
\end{proof}

\begin{lemma}[Isolated spectral projections in the coordinate graph]
\label{lem:isolated-graph-projection}
Let $u_0\in H^s(\mathbb R;\mathbb R)$, $-\frac12<s<0$, and let
$\lambda<0$ be an isolated simple eigenvalue of $L_{u_0}$.
For every real approximation $u_{0,n}\to u_0$ in $H^s$, its isolated
spectral projections admit normalized generators $f_n,f$ such that
$L_{u_{0,n}}f_n=\lambda_nf_n$, $L_{u_0}f=\lambda f$,
$\lambda_n\to\lambda$, and
\begin{equation}
\|f_n-f\|_{H^{1+s}}+\|X^*f_n-X^*f\|_{H^{1+s}}\to0.
\label{eq:isolated-projection-graph-limit}
\end{equation}
In particular, every vector in $\ker(L_{u_0}-\lambda)$ belongs to
$\operatorname{Dom}(X^*)$, and
\begin{equation}
I_+(f_n)\to I_+(f),\enspace
\langle f_n,X^*f_n\rangle\to\langle f,X^*f\rangle.
\label{eq:isolated-trace-center-limit}
\end{equation}
The approximation is in the unweighted $H^s$ topology.
\end{lemma}

\begin{proof}[Proof outline]
Choose a contour $\Gamma$ isolating $\lambda$ in the left half-plane
and form the Riesz projections
\[
 P=-\frac1{2\pi i}\int_\Gamma(L_{u_0}-z)^{-1}\,\rmd z,
 \qquad
 P_n=-\frac1{2\pi i}\int_\Gamma(L_{u_{0,n}}-z)^{-1}\,\rmd z.
\]
Norm-resolvent convergence preserves their rank one and gives
$\lambda_n\to\lambda$.  Fix $h\in\mathcal D_+^\circ$ with $Ph\ne0$.
Lemma~\ref{lem:X-graph-resolvent} permits these contour integrals in the
complete space with norm $\|g\|_{H^{1+s}}+\|X^*g\|_{H^{1+s}}$; hence
$P_nh\to Ph$ in that graph norm.  Normalizing these vectors fixes compatible
phases and gives \eqref{eq:isolated-projection-graph-limit}.
The trace estimate \eqref{eq:graph-trace} and continuity of the diagonal
matrix element give \eqref{eq:isolated-trace-center-limit}.  Since every
vector in the eigenline is a scalar multiple of its generator, the domain
inclusion follows; the center is invariant under a change of phase.
\end{proof}

\begin{lemma}[Rank-one resolvent reduction at a simple eigenvalue]
	\label{lem:ordered-Schur-extraction}
	Let $\lambda<0$ be a simple eigenvalue of $L_{u_0}$.  Suppose
	$u_{0,n}\to u_0$ in $H^{s}$ is the fixed smooth approximation used above and
	$L_{u_{0,n}}f_n=\lambda_nf_n$, $\|f_n\|_2=1$, $\lambda_n\to\lambda$, and
	$\|f_n-f\|_{H^{1+s}}+\|X^*(f_n-f)\|_2\to0$,
	where $L_{u_0}f=\lambda f$ and $\|f\|_2=1$.  Let $g_n\in L_+^2$ and
	$g\in H_+^s$ satisfy $g_n\to g$ in $H_+^s$ and
	$d_n:=\langle f_n,g_n\rangle\to
	d:=\langle f,g\rangle_{H_+^{-s},H_+^s}$.
	Assume that, for each fixed $t$,
	\begin{equation}
		\Gamma_t(z)=\lim_{n\to\infty}
		\mathcal O_{t,\lambda,z}^{(n)}(g_n)
		\label{eq:ordered-Schur-fixed-time-bridge}
	\end{equation}
	exists locally uniformly in $z\in\C_+$.  If
	$b=\langle f,X^*f\rangle$, then
	\begin{equation}
		\lim_{|t|\to\infty}\sup_{z\in K}
		\left|\Gamma_t(z)-\frac{dI_+(f)}{2\pi i(b-z)}\right|=0
		\enspace(K\Subset\C_+\text{ compact}).
		\label{eq:ordered-Schur-limit}
	\end{equation}
	The order is $n\to\infty$ at fixed $t$, followed by $|t|\to\infty$.
\end{lemma}

\begin{proof}[Proof outline]
Set $b_n=\langle f_n,X^*f_n\rangle$, $w_n=X^*f_n-b_nf_n$,
$\zeta_{n,t}(z)=b_n-2t(\lambda_n-\lambda)-z$, and
$\mathfrak A_{n,t}(z)=C_{n,t,0}+2t\lambda-z$.  The eigenvalue equation
separates the resolvent into
\[
 \mathfrak A_{n,t}(z)^{-1}g_n
 =\frac{d_n}{\zeta_{n,t}(z)}f_n+
 \mathfrak A_{n,t}(z)^{-1}
 \left[g_n-d_nf_n-\frac{d_n}{\zeta_{n,t}(z)}w_n\right].
\]
Dissipativity bounds $|\zeta_{n,t}(z)|$ below by $\Im z$.
At fixed $t$, graph convergence and the trace estimate identify the first
term's observable limit as $dI_+(f)/(2\pi i(b-z))$.
The remaining sources converge uniformly on $K$ in $H_+^s$ to
\[
 v_z=g-df-\frac{d}{b-z}(X^*f-bf),\qquad \langle f,v_z\rangle=0.
\]
Theorem~\ref{thm:rigged-range} puts each $v_z$ in the closed range.
Approximate the compact family $\{v_z:z\in K\}$ by bounded Fourier
cutoffs with values in $L_+^2$, then by a finite net.
Lemma~\ref{lem:removal-complement} removes each net element, and
\eqref{eq:observable-Y-bound} controls the approximation errors uniformly.
Taking $n\to\infty$ first, then $|t|\to\infty$, and finally shrinking the
net error proves the required locally uniform limit.
\end{proof}

\subsection{Proof of Theorem~\ref{thm:main-eigenfunction-graph}: coordinate regularity and stability of bound states}
\label{subsec:proof-eigenfunction-graph}

\begin{proof}[Proof outline]
Simplicity and isolation follow from \cite[Theorem~1.4]{BKV} and
\cite[Proposition~3.2]{KLV}.

Lemma~\ref{lem:isolated-graph-projection},
with phases chosen to match $f_j$, gives
\eqref{eq:main-eigenfunction-graph}, \eqref{eq:main-eigenfunction-stability},
and the trace and center limits in
\eqref{eq:main-eigenfunction-data-stability}.
The boundary identities \cite[(1.15)--(1.16)]{BKV} and
\eqref{eq:graph-trace} give
\[
 |I_+(f_j)|^2=\frac{2\pi}{a_j},\qquad
 \Im\langle f_j,X^*f_j\rangle=-\frac1{2a_j}.
\]
The definitions of $c_j,p_j$ yield \eqref{eq:main-eigenfunction-center}.
The center is unchanged by a unimodular factor, and the contour argument
applies to each fixed isolated eigenvalue.
\end{proof}

\cite[Theorem~1.4]{BKV} gives the rough Wu relations,
simplicity, the normalized boundary residue, and
$\sigma_{\mathrm p}(L_{u_0})\subset(-\infty,0)$.  If $L_{u_0}f=-af$ and $\|f\|_2=1$, then
Lemma~\ref{lem:isolated-graph-projection} gives
$f\in\operatorname{Dom}(X^*)$.
For the smooth normalized eigenfunctions in that lemma, the Fourier
eigenvalue equation at zero gives
$a_n I_+(f_n)=\langle (u_{0,n})_+,f_n\rangle$.
Passing to the limit in the trace and the $H^s$--$H^{-s}$ pairing gives
$aI_+(f)=\langle(u_0)_+,f\rangle$.
As in Theorem~\ref{thm:main-eigenfunction-graph}, the Wu residue and
\eqref{eq:graph-trace} yield $|I_+(f)|^2=2\pi/a$ and
\begin{equation}
	\Im\langle f,X^*f\rangle=-\frac1{2a}.
	\label{eq:Wu-boundary-center}
\end{equation}

The following extends the windowed resolution of Badreddine, Killip,
and Vi\c{s}an \cite[Theorem~1.5 and Proposition~5.2]{BKVasymp}
below $L^2$.
\begin{prop}[Asymptotics on a fixed ray]
	\label{prop:fixed-ray}
	Let $u_0\in H^s(\mathbb R;\mathbb R)$, $-\frac12<s<0$.
For every $\lambda\in\R$ and
$\chi\in C_c^\infty(\R)$,
	\[
\begin{cases}
			\displaystyle
			\|\chi u(t,\cdot-2t\lambda)\|_{H^s}\longrightarrow0,
			&\lambda\notin\sigma_{\mathrm p}(L_{u_0}),\\[2mm]
			\displaystyle
			\|\chi[\,u(t,\cdot+2at)-R_{-c_\lambda+i/(2a)}\,]\|_{H^s}
			\longrightarrow0,
			&\lambda=-a\in\sigma_{\mathrm p}(L_{u_0}),\enspace a>0,
		\end{cases}
		\enspace |t|\to\infty ,
\]
		where $R_{-c+i/(2a)}(x)=4a/[1+4a^2(x-c)^2]$ and,
		in the eigenvalue case,
		$c_\lambda=\operatorname{Re}\langle f,X^*f\rangle$,
		with $f$ an $L^2$-normalized eigenfunction satisfying
		$L_{u_0}f=\lambda f$.
\end{prop}

\begin{proof}[Proof outline]
The smooth conjugation identity and flow approximation identify, at
fixed $t$, the locally uniform limit
$\mathcal O_{t,\lambda,z}^{(n)}((u_{0,n})_+)
\to u_+(t,z-2t\lambda)$.
For a non-eigenvalue, Theorem~\ref{thm:rigged-range} and
Lemma~\ref{lem:removal-complement} make this Hardy component vanish as
$|t|\to\infty$.

For $\lambda=-a$, Lemma~\ref{lem:isolated-graph-projection} gives
\begin{equation}
 L_{u_{0,n}}f_n=\lambda_nf_n,\quad \lambda_n\to\lambda,\quad
 \|f_n-f\|_{H^{1+s}}+\|X^*(f_n-f)\|_2\to0.
 \label{eq:graph-convergent-eigenline}
\end{equation}
The source pairings converge by $1+s>-s$.
Lemma~\ref{lem:ordered-Schur-extraction} therefore gives
\[
u_+(t,z+2at)\to dI_+(f)/(2\pi i(b-z)),
\]
where
$d=\langle f,(u_0)_+\rangle$ and $b=\langle f,X^*f\rangle$.
The zero-frequency equation and Wu boundary identity yield
\begin{equation}
 dI_+(f)=2\pi a|\widehat f(0+)|^2=2\pi.
 \label{eq:Wu-phase-product}
\end{equation}
Together with \eqref{eq:Wu-boundary-center}, this identifies the limit as
$2ia/[2a(z-c_\lambda)+i]$, the Hardy extension of the stated soliton.
In both cases, reality recovers the other frequency component, while
Lemma~\ref{lem:local-Hs-compactness} and \eqref{eq:uniform-orbit-tail}
upgrade the Hardy limit to local $H^s$ convergence.  All limits use
$n\to\infty$ at fixed $t$ before $|t|\to\infty$.
\end{proof}

\section{Negative spectrum, spatial centers, and asymptotic velocities}
\label{sec:negative-velocity}

\begin{prop}[Negative spectral data]
	\label{prop:negative-spectrum}
	Let $u_0\in H^s(\mathbb R;\mathbb R)$ with $-\frac12<s<0$.  The point spectrum of $L_{u_0}$ is an at
most countable family of simple
	eigenvalues indexed by a set $\mathcal J$,
	\begin{equation}
		\lambda_j=-a_j<0,
		\enspace \sum_j a_j<\infty,
		\label{eq:negative-spectrum-sum}
	\end{equation}
	which is locally finite in $(-\infty,0)$, and zero is its unique possible
	accumulation point.  Moreover,
	$\sigma_{\mathrm p}(L_{u_0})\subset(-\infty,0)$.  If $f_j$ is normalized,
	then $f_j\in\operatorname{Dom}(X^*)$ and
	\begin{equation}
		p_j=-\langle f_j,X^*f_j\rangle
		=-c_j+\frac{i}{2a_j}.
		\label{eq:pj-graph-data}
	\end{equation}
	Here $c_j=-\operatorname{Re}p_j$ is the spatial center.  The soliton
	speed $c_{p_j}$ recalled in Section~\ref{subsec:notation-guide} is
	$c_{p_j}=(\operatorname{Im}p_j)^{-1}=2a_j=-2\lambda_j$.
	Along the BO flow the eigenvalue is fixed and
	\begin{equation}
		c_j(u(t))=c_j(u_0)+2a_jt.
		\label{eq:center-motion}
	\end{equation}
	Finally, the conserved measure in \eqref{eq:rho-measure} has atoms
	\begin{equation}
		\mu_{u_0}(\{-a_j\})=2\pi a_j,
		\enspace
		\nu(\{-a_j\})=\frac{2\pi a_j}{\ka-a_j}.
		\label{eq:Wu-atom-weight}
	\end{equation}
\end{prop}

\begin{proof}[Proof outline]
The spectral conclusions follow from \cite[Proposition~3.2]{KLV} and
\cite[Theorem~1.4]{BKV}; \cite[(7.5)]{BKV} gives
\eqref{eq:Wu-atom-weight}.  Since $E+\ka$ is bounded on the negative
spectrum and $\nu$ is finite, these weights imply $\sum_j a_j<\infty$.
Conservation of $\mu_{u(t)}$ preserves all its nonzero atoms and hence the
complete negative spectrum.

The graph convergence of Theorem~\ref{thm:main-eigenfunction-graph} and the
Wu boundary identity give \eqref{eq:pj-graph-data} and $p_{j,n}\to p_j$.
For smooth solutions, apply the closed conjugation identity
\begin{equation}
 U_n(t,0)^*X^*U_n(t,0)=\overline{X^*-2tL_{u_{0,n}}}
 \label{eq:closed-Peter-center}
\end{equation}
on the transported eigenline.  Its diagonal coordinate element increases
by $2a_{j,n}t$.
To pass to rough data, fix $t$ and apply
Lemma~\ref{lem:isolated-graph-projection} at times $0$ and $t$,
using sharp flow continuity.  The diagonal elements are phase independent,
so both graph limits pass through the smooth identity.  Their real parts
give \eqref{eq:center-motion}.
\end{proof}

For the velocity analysis, let $\iota_+:L_+^2\hookrightarrow L^2(\R)$
be the Hardy inclusion and $M_x$ multiplication by $x$ on the full line.

\begin{lemma}[Naimark resolvent identity]
	\label{lem:Naimark}
	The operator $X^*$ is maximal dissipative.
	For $z\in\C_+$,
	\begin{equation}
		(X^*-z)^{-1}=\iota_+^*(M_x-z)^{-1}\iota_+
		\label{eq:Naimark-resolvent}
	\end{equation}
	and
	\begin{equation}
		[(X^*-z)^{-1}h]^{\wedge}(\xi)
		=ie^{-iz\xi}\int_\xi^\infty e^{iz\eta}\widehat h(\eta)\,\rmd\eta.
		\label{eq:Naimark-Volterra}
	\end{equation}
	If $t\neq0$ and $t\Im z>0$, then
	\begin{equation}
		(X^*/t-z)^{-1}=\iota_+^*(M_x/t-z)^{-1}\iota_+,
		\enspace \|(X^*/t-z)^{-1}\|\le|\Im z|^{-1}.
		\label{eq:scaled-Naimark}
	\end{equation}
\end{lemma}

\begin{proof}[Proof outline]
Dissipativity follows from \eqref{eq:graph-trace}.  Direct differentiation
	verifies that the right-hand side $F$ of \eqref{eq:Naimark-Volterra}
	solves $(i\partial_\xi-z)F=\widehat h$.  The
	Volterra--Schur bound gives $\|F\|_2\le(\Im z)^{-1}\|h\|_2$, proving
	surjectivity and maximality.  Compressing the full-line multiplication
	resolvent solves the same equation and hence proves
	\eqref{eq:Naimark-resolvent}.  For $t\Im z>0$, scaling by
	$(X^*/t-z)^{-1}=t(X^*-tz)^{-1}$ gives \eqref{eq:scaled-Naimark}.
\end{proof}

Using the pushforward notation of Section~\ref{subsec:notation-guide},
for $h\in\mathcal K_{\ka,0}^0$ define the finite positive measures
\begin{align*}
	\mu_{t,h}&=(x/t)_\#\bigl(| U_{u_0}^{\mathrm{cyc}}(t,0)h(x)|^2\,\rmd x\bigr),\\
	\nu_h&=(E\mapsto-2E)_\#\rmd\langle h,E_{u_0}((-\infty,E])h\rangle.
\end{align*}

\begin{thm}[Asymptotic velocity measure]
	\label{thm:velocity-observable}
	For every $h\in\mathcal K_{\ka,0}^0$ and $\phi\in C_0(\R)$,
	\begin{equation}
		\lim_{t\to\pm\infty}\int_\R\phi(x/t)
		| U_{u_0}^{\mathrm{cyc}}(t,0)h(x)|^2\,\rmd x
		=\langle h,\phi(-2L_{u_0})h\rangle.
		\label{eq:velocity-observable}
	\end{equation}
	Equivalently, $\mu_{t,h}\Rightarrow\nu_h$ narrowly.
\end{thm}

\begin{proof}[Proof outline]
Use the dense graph core
$\mathcal C_{u_0}=(L_{u_0}-i)^{-1}\mathcal D_+^\circ$ and write
$B=-2L_{u_0}$, $B_n=-2L_{u_{0,n}}$, and
$B_{n,t}=t^{-1}C_{n,t,0}$.

Fix $h=F(L_{u_0})m_0$ and its compatible smooth representatives
$h_n=F(L_{u_{0,n}})m_{n,0}$, with $F\in C_c^\infty(\R)$.
The conjugation identity and Lemma~\ref{lem:Naimark} identify
$\langle h_n,(B_{n,t}-z)^{-1}h_n\rangle$ with the Cauchy transform of
$(x/t)_\#(|U_n(t,0)h_n|^2\rmd x)$ whenever $t\Im z>0$.

Approximate $h$ within $\varepsilon$ by $(B-z)f$ with
$f\in\mathcal C_{u_0}$.  Lemma~\ref{lem:X-graph-resolvent} supplies smooth
$f_n$ with $M_\varepsilon:=\sup_n\|X^*f_n\|_2<\infty$.  The relation
$B_{n,t}f_n=B_nf_n+t^{-1}X^*f_n$ and the uniform resolvent bounds give
\[
 \|[(B_{n,t}-z)^{-1}-(B_n-z)^{-1}]h_n\|_2
 \le \frac{4\varepsilon}{|\Im z|}
      +\frac{M_\varepsilon}{|t|\,|\Im z|}.
\]
First take $n\to\infty$ using compact spectral kernel convergence and
norm-resolvent convergence, then $|t|\to\infty$ and
$\varepsilon\downarrow0$.  This proves convergence of the Cauchy
transforms in the appropriate half-plane.  Isometry extends it to every
$h\in\mathcal K_{\ka,0}^0$.
Uniqueness of the Cauchy transform identifies each vague cluster point
with $\nu_h$; equality of total masses gives tightness and therefore
narrow convergence.
\end{proof}

Applying Theorem~\ref{thm:velocity-observable} to $m_0$ and using
\eqref{eq:Wu-atom-weight} gives the velocity measure
$\nu_{m_0}$:
\begin{equation}
	(x/t)_\#(|m_t(x)|^2\,\rmd x)
	\Rightarrow
	\sum_j\frac{2\pi a_j}{(\ka-a_j)^2}\delta_{2a_j}
	+\nu_{m_0}^{\mathrm c},
	\enspace \operatorname{supp}\nu_{m_0}^{\mathrm c}\subset(-\infty,0],
	\label{eq:velocity-measure-decomposition}
\end{equation}
where $\nu_{m_0}^{\mathrm c}=\nu_{m_0}|_{(-\infty,0]}$ is nonatomic.

\section{Classification of asymptotic profiles and local \texorpdfstring{$H^s$}{H\textasciicircum s} compactness}
\label{sec:profile-rigidity}

\begin{lemma}[Compact multiplication]
	\label{lem:compact-multiplication}
	If $-\frac12<s<0$ and $\psi\in\mathcal S(\R)$, then
	\begin{equation}
		w\longmapsto w\overline\psi,
		\enspace H^{1+s}(\R)\longrightarrow H^{-s}(\R),
		\label{eq:compact-multiplication}
	\end{equation}
	is compact.  Multiplication by a compactly supported smooth function is also
	compact from $H^{1+s}$ to $L^2$.
\end{lemma}

\begin{proof}[Proof outline]
Insert a spatial cutoff $\chi_R$.  On a fixed compact set, Rellich's
	theorem and $1+s>-s$ make $w\mapsto\chi_Rw\overline\psi$ compact into
	$H^{-s}$.  The multiplier norm of $(1-\chi_R)\psi$ tends to zero as
	$R\to\infty$, so \eqref{eq:compact-multiplication} is a norm limit of
	compact operators.  The second statement follows identically, with a larger
	regularity gap.
\end{proof}

For any real $v\in H^{s}(\R)$ and any coercive shift $\ka$, write
$m(\ka;v):=(L_v+\ka)^{-1}v_+\in H_+^{1+s}$.

\begin{prop}[Localized resolvent convergence under weak convergence]
	\label{prop:weak-local-inverse}
	Suppose $v_n\rightharpoonup V$ in $H^{s}(\R;\R)$ and
	$g_n\rightharpoonup g$ in $H_+^{s}$, with both sequences bounded.
	Choose one common coercive shift and set
	$R_n=(L_{v_n}+\ka)^{-1}$ and $R=(L_V+\ka)^{-1}$.
	Here $R_n$ and $R$ denote the canonical bounded extensions
	$H_+^{s}\to H_+^{1+s}$ from
	Propositions~\ref{prop:cross-scale-resolvent} and
	\ref{prop:global-orbit-control}.
	Then, for every $\chi\in C_c^\infty(\R)$,
	\begin{equation}
		\chi(R_ng_n-Rg)\longrightarrow0
		\enspace\hbox{in }L^2.
		\label{eq:weak-local-inverse}
	\end{equation}
	The conclusion applies in particular to the gauge vectors and, for every
	$F\in C_c(\R)$,
	\begin{equation}
		\chi\bigl[F(L_{v_n})m(\ka;v_n)-F(L_V)m(\ka;V)\bigr]
		\longrightarrow0
		\enspace\hbox{in }L^2.
		\label{eq:spectral-filter-locality}
	\end{equation}
\end{prop}

\begin{proof}[Proof outline]
The common inverse estimates bound $w_n=R_ng_n$ in $H^{1+s}$.
For a weak limit $w$, test the resolvent equation with
$\psi\in\mathcal D_+^\circ$.  Compact multiplication gives
$w_n\overline\psi\to w\overline\psi$ in $H^{-s}$, so the potential pairing
passes to the limit despite the weak convergence of $v_n$.
The limiting equation and coercivity identify $w=Rg$.
Compact localization then proves \eqref{eq:weak-local-inverse}.

Taking $g_n=(v_n)_+$ gives the gauge limit.  Repeated application gives
local convergence for every power of $R_n$ acting on its gauge vector.
Write $F(L_{v_n})=G(R_n)$ with $G(r)=F(r^{-1}-\ka)$ and $G(0)=0$.
Uniform polynomial approximation of $G$ on the common compact spectral
interval, together with the uniform gauge norm, extends the power limits
to \eqref{eq:spectral-filter-locality}.
\end{proof}

Fix $\sigma\in\{+,-\}$ and assume $t_n\to\sigma\infty$ and
\begin{equation}
	u(t_n,\,\cdot+x_n)\rightharpoonup U\neq0
	\enspace\hbox{in }H^s.
	\label{eq:profile-assumption}
\end{equation}

\begin{lemma}[Spectral support of a limit profile at one energy]
	\label{lem:profile-spectral-compression}
	The sequence $x_n/t_n$ is bounded.  If, along a subsequence,
	$x_n/t_n\to v$ and $\lambda=-v/2$, then
	\begin{equation}
		E_U(\R\setminus\{\lambda\})m(\ka;U)=0,
		\enspace (L_U-\lambda)m(\ka;U)=0,
		\enspace U_+=(\ka+\lambda)m(\ka;U).
		\label{eq:profile-single-spectral-point}
	\end{equation}
\end{lemma}

\begin{proof}[Proof outline]
Translation covariance and Proposition~\ref{prop:weak-local-inverse}
give
\begin{equation}
 m(\ka;u(t_n))(\cdot+x_n)\longrightarrow m(\ka;U)
 \quad\hbox{in }L^2_{\mathrm{loc}}.
 \label{eq:shifted-gauge-local-limit}
\end{equation}
The limit is nonzero because $(L_U+\ka)m(\ka;U)=U_+$ and $U\ne0$ is real.
Thus a fixed moving window carries positive gauge mass.  Tightness of the
velocity measures forces $x_n/t_n$ to remain bounded.

If $x_n/t_n\to v$, test with a compact spectral filter $F$ supported away
from $\lambda=-v/2$.  Choose a velocity cutoff equal to one near $v$ and
disjoint from the velocities of $\operatorname{supp}F$.
Theorem~\ref{thm:velocity-observable} and spectral intertwining make the
filtered mass in the moving window tend to zero.
Proposition~\ref{prop:weak-local-inverse} identifies its local limit as
$F(L_U)m(\ka;U)$, which therefore vanishes.
Increasing such filters to $\mathbf1_{\R\setminus\{\lambda\}}$ proves
single-point spectral support.  The spectral theorem and the gauge
equation then give all identities in
\eqref{eq:profile-single-spectral-point}.
\end{proof}

\begin{lemma}[A single spectral atom determines a one-soliton profile]
	\label{lem:one-atom-one-soliton}
	In Lemma~\ref{lem:profile-spectral-compression}, one has $\lambda=-a<0$ and
	\begin{equation}
		\|m(\ka;U)\|_2^2=\frac{2\pi a}{(\ka-a)^2},
		\enspace
		\beta(\ka;U)=\frac{2\pi a}{\ka-a}.
		\label{eq:profile-gauge-mass}
	\end{equation}
	Moreover, $U=R_{-b+i/(2a)}$ for a unique center $b\in\R$, and $-a$ is one of the
	original eigenvalues of $L_{u_0}$.
\end{lemma}

\begin{proof}[Proof outline]
Since the gauge vector is nonzero, \cite[Theorem~1.4]{BKV} implies
$\lambda=-a<0$.  Apply the Wu relation to its normalized eigenvector:
\[
 (\ka-a)^2\|m(\ka;U)\|_2^2=2\pi a.
\]
Pairing the gauge equation gives the other identity in
\eqref{eq:profile-gauge-mass}.  Weak lower semicontinuity gives
$\|U\|_{H^s}\le\sup_t\|u(t)\|_{H^s}$, so the common shift from
Proposition~\ref{prop:global-orbit-control} remains admissible for $U$.
These identities saturate \cite[(1.17)]{BKV} for $N=1$, and
\cite[Theorem~1.5]{BKV} identifies $U$
as the stated soliton with a unique center.

Local gauge convergence places its mass in windows whose velocities
approach $2a$.  Closed-set Portmanteau, first on a fixed spatial window
and velocity neighborhood, then with the window enlarged and the
neighborhood shrunk, gives
$\|m(\ka;U)\|_2^2\le\nu_{m_0}(\{2a\})$.
Thus $-a$ is an original eigenvalue; the atom formula shows that the
profile exhausts its mass.
\end{proof}

\begin{lemma}[Uniqueness of the profile associated with a velocity atom]
	\label{lem:one-profile-per-atom}
	Suppose two moving windows $y_n^{(1)}$, $y_n^{(2)}$ satisfy
	$y_n^{(k)}/t_n\to2a$ and $|y_n^{(1)}-y_n^{(2)}|\to\infty$,
	and, along this common index sequence,
	$u(t_n,\,\cdot+y_n^{(k)})\rightharpoonup U^{(k)}$ in $H^s$, $k=1,2$.
	Then at least one of $U^{(1)}$ and $U^{(2)}$ equals zero.
\end{lemma}

\begin{proof}[Proof outline]
If both limits were nonzero, the preceding two lemmas would make each
profile a soliton with gauge mass $M_a=2\pi a/(\ka-a)^2$.
Local gauge convergence puts at least $M_a-\varepsilon$ in each of two
eventually disjoint fixed windows.  Both windows lie over every fixed
velocity neighborhood of $2a$ for large $n$.
Portmanteau therefore bounds that neighborhood's limiting mass below by
$2M_a-2\varepsilon$.  Shrinking the neighborhood and then sending
$\varepsilon\downarrow0$ contradicts
$\nu_{m_0}(\{2a\})=M_a$.
\end{proof}

\begin{thm}[Rigidity of nonzero asymptotic profiles]
	\label{thm:profile-rigidity}
	Under \eqref{eq:profile-assumption}, there exist a unique index $j$ and a
	unique $d\in\R$ such that $x_n/t_n\longrightarrow2a_j$,
	$x_n-2a_jt_n\longrightarrow d$, and $U=R_{p_j+d}$.
\end{thm}

\begin{proof}[Proof outline]
The preceding classification identifies every subsequential velocity
limit with the unique negative eigenvalue of the fixed profile $U$.
Thus $x_n/t_n\to2a_j$ for a unique $j$.
Set $d_n=x_n-2a_jt_n$.  Proposition~\ref{prop:fixed-ray} gives the
canonical local limit $u(t_n,\cdot+2a_jt_n)\to R_{p_j}$.
If $|d_n|\to\infty$ along a subsequence, weak compactness and this local
limit produce two nonzero weak profiles, $U$ and $R_{p_j}$, at the same
velocity in separated windows.  Lemma~\ref{lem:one-profile-per-atom}
excludes this, so $(d_n)$ is bounded.
Every subsequential limit $d$ now gives $U=R_{p_j+d}$ by translation
continuity.  Uniqueness of the soliton center makes all such limits equal,
which proves convergence of the full sequence.
\end{proof}

\begin{lemma}[Localized compactness in $H^s$]
	\label{lem:local-Hs-compactness}
	Let $s\in(-\frac12,0)$ and let $F_n$ be a real-valued bounded sequence in
	$H^s(\R)$ satisfying the common condition on the high-frequency tails
	\begin{equation}
		\lim_{M\to\infty}\sup_n\|P_{>M}F_n\|_{H^s}=0.
		\label{eq:common-high-frequency-tail}
	\end{equation}
	Then $\{\chi F_n\}$ is precompact in $H^s$ for every
	$\chi\in C_c^\infty(\R)$.  If, in addition, there is a real
	$F\in H^s(\R)$ such that
	\begin{equation}
		\mathcal E_z(\Pi F_n)\longrightarrow\mathcal E_z(\Pi F)
		\enspace(z\in\C_+),                                      \label{eq:local-compactness-Hardy-target}
	\end{equation}
	then
	\begin{equation}
		\|\chi(F_n-F)\|_{H^s}\longrightarrow0
		\enspace\text{for every }\chi\in C_c^\infty(\R).         \label{eq:local-compactness-strong-conclusion}
	\end{equation}
\end{lemma}

\begin{proof}[Proof outline]
For fixed $M$, the localized low-frequency operator
$\chi\chi_{\rm lo}(D/M)$ is compact on $H^s$: its conjugated Fourier
kernel is Hilbert--Schmidt.  The remaining high-frequency term is
uniformly small by \eqref{eq:common-high-frequency-tail}.
This proves local precompactness.

Any subsequence consequently admits a weak $H^s$ limit $G$ and strong
localized limits.  Continuity of Hardy evaluation and
\eqref{eq:local-compactness-Hardy-target} imply
$\mathcal E_z(\Pi(G-F))=0$ throughout $\C_+$.
Laplace-transform uniqueness, reality, and $H^s$ membership give $G=F$.
Every subsequential local limit is therefore $F$, proving
\eqref{eq:local-compactness-strong-conclusion}.
\end{proof}

For the BO orbit, \eqref{eq:uniform-orbit-tail} is invariant under
translations, so Lemma~\ref{lem:local-Hs-compactness} applies to every
moving sequence.

\begin{thm}[Uniform asymptotics on compact intervals of positive velocities]
	\label{thm:velocity-cones}
	Let $K\Subset(0,\infty)$ be a compact interval.
	\begin{enumerate}[(i)]
		\item If $K\cap\{2a_j:j\in\mathcal J\}=\varnothing$, then for every
		$\chi\in C_c^\infty$,
		$\sup_{v\in K}\|\chi u(t,\,\cdot+vt)\|_{H^s}\longrightarrow0$
		as $t\to\pm\infty$.
		\item If $K$ contains finitely many velocities $2a_j$, then
		\[
\sup_{v\in K}\left\|\chi\left[u(t,\,\cdot+vt)
			-\sum_{2a_j\in K}R_{p_j}
			(\,\cdot+(v-2a_j)t\,)
			\right]\right\|_{H^s}\longrightarrow0.
\]
	\end{enumerate}
\end{thm}

\begin{proof}[Proof outline]
A failing sequence $t_n\to\sigma\infty$, $v_n\in K$ admits, by the
common frequency tail and local compactness, a subsequence with
$v_n\to v$ and a strong local $H^s$ profile $U$.
In (i), failure makes $U\ne0$, while profile rigidity forces $v=2a_j$,
contrary to the choice of $K$.

For (ii), if $v$ is not a soliton velocity, every subtracted soliton
escapes locally and the same contradiction applies.
If $v=2a_j$, consider $d_n=(v_n-2a_j)t_n$.
A bounded subsequence is handled by the fixed-ray asymptotics and
separation of the other centers.  On a subsequence with $|d_n|\to\infty$,
all subtracted profiles escape; a nonzero residual limit would contradict
the bounded-offset conclusion of profile rigidity.  These alternatives
prove the uniform convergence.
\end{proof}

\section{Weighted limiting absorption estimates}
\label{sec:no-singular}

If the hypothesis is initially given with an exponent
$\be_*>\frac12$, fix
$\be\in(\frac12,\min\{\be_*,1\})$.  Multiplication by
$\langle x\rangle^{\be-\be_*}$ is bounded on $H^{s}$, so the hypothesis
holds with $\be$.  Throughout this section
\begin{equation}
	-\frac12<s<0,
	\enspace \frac12<\be<1,
	\enspace u_0,\ \langle x\rangle^\be u_0\in H^{s}(\R;\R),
	\label{eq:NS-assumptions}
\end{equation}
and $J\subset(0,\infty)$ is a fixed compact interval.  Constants are uniform
in the imaginary part of the spectral parameter and may depend on a slightly
larger positive-energy interval containing $J$.

\subsection{Spatial localization and Fredholm factorization}

Choose real functions $\theta,\chi\in C_c^\infty(\R)$ such that the
translates $\theta_k(x)=\theta(x-k)$ form a uniformly locally finite
partition of unity and $\chi_k(x)=\chi(x-k)$ is one on a fixed neighborhood
of $K_k=\operatorname{supp}\theta_k$.  Let
$\widetilde\chi_k$ be translates of a fixed cutoff which equals one on a
neighborhood of $\operatorname{supp}\chi_k$.  Put
$u_{0,k}=\theta_k u_0$, $c_k=\|u_{0,k}\|_{H^{s}}$, and
$\mathcal Q_{s,\be}(u_0)^2
=\sum_{k\in\Z}\langle k\rangle^{2\be}c_k^2$.

For $r,\delta\in\R$, write
$\|F\|_{H^{r,\delta}_{\rm bl}}^2
=\sum_{j\in\Z}\langle j\rangle^{2\delta}\|\chi_jF\|_{H^r}^2$.

\begin{lemma}[Weighted block calculus]
	\label{lem:NS-weighted-block-calculus}
	For every $r\in[-1,1]$,
	\begin{align}
		\sum_k\|\theta_kF\|_{H^r}^2
		&\le C_r\|F\|_{H^r}^2,
		\label{eq:NS-analysis-map}\\
		\left\|\sum_kF_k\right\|_{H^r}^2
		&\le C_r\sum_k\|F_k\|_{H^r}^2,
		\enspace \operatorname{supp}F_k\subset K_k.
		\label{eq:NS-synthesis-map}
	\end{align}
	If $b$ ranges over a bounded subset of $S^{-m}(\R)$, $m\ge0$, then, for
	every fixed $\delta\ge0$,
	\begin{align}
		\|\chi_jb(D)\widetilde\chi_kF\|_{H^{r+m}}
		&\le C_N\langle j-k\rangle^{-N}
		\|\widetilde\chi_kF\|_{H^r},
		\label{eq:NS-smooth-symbol-block}\\
		\|b(D)F\|_{H^{r+m,\delta}_{\rm bl}}
		&\le C_\delta\|F\|_{H^{r,\delta}_{\rm bl}}.
		\label{eq:NS-smooth-symbol-weighted}
	\end{align}
	The first estimate uses separated cutoffs; overlapping cutoffs use the
	ordinary Sobolev multiplier bound.  If $-1/2\le r\le1/2$ and
	$|\delta|<1/2$, then
	\begin{equation}
		\|\Pi F\|_{H^{r,\delta}_{\rm bl}}
		\le C_{r,\delta}\|F\|_{H^{r,\delta}_{\rm bl}}.
		\label{eq:NS-Pplus-weighted-Sobolev}
	\end{equation}
\end{lemma}

\begin{proof}[Proof outline]
Almost-diagonal bounds for the localized Sobolev multipliers and block
Schur estimates give \eqref{eq:NS-analysis-map}--\eqref{eq:NS-smooth-symbol-weighted};
duality handles negative orders, and the weighted Schur test uses
$N>\delta+1$.  For $|\delta|<1/2$, localization yields
\begin{equation}
 \|F\|_{H^{r,\delta}_{\rm bl}}
 \asymp\|\langle x\rangle^\delta\langle D_x\rangle^rF\|_2.
 \label{eq:NS-local-global-weight-equivalence}
\end{equation}
For $|\delta|<1/2$, the weight $\langle x\rangle^{2\delta}$ belongs to $A_2$.
The weighted Hilbert-transform bound, together with commutation of $\Pi$
and $\langle D_x\rangle^r$, proves \eqref{eq:NS-Pplus-weighted-Sobolev}.
\end{proof}

The fixed partition gives
$\mathcal Q_{s,\be}(u_0)\lesssim\|\langle x\rangle^\be u_0\|_{H^s}
\lesssim\mathcal Q_{s,\be}(u_0)$:
on $K_k$, $\langle x\rangle^\be/\langle k\rangle^\be$ and its inverse,
localized by enlarged cutoffs, are uniformly bounded $H^s$ multipliers.
Apply \eqref{eq:NS-analysis-map} to $\langle x\rangle^\be u_0$ for one
inequality and \eqref{eq:NS-synthesis-map} to the weighted localized
pieces for the other.

For $\sigma,\delta\ge0$, define the source and field spaces
\begin{align}
	\mathcal Y_\sigma
	&=\left\{\mathbf s=(\mathbf s_k)_{k\in\Z}:\ \operatorname{supp}\mathbf s_k\subset K_k,
	\ \sum_k\langle k\rangle^{2\sigma}
	\|\mathbf s_k\|_{H^{-1/2}}^2<\infty\right\},
	\label{eq:NS-source-space}\\
	X_\delta
	&=\left\{f\in H^{1/2}_{\mathrm{loc}}(\R):
	\|f\|_{X_\delta}:=\sup_j\langle j\rangle^\delta
	\|\chi_jf\|_{H^{1/2}}<\infty\right\}.
	\label{eq:NS-field-space}
\end{align}
Let $\Sigma\mathbf s=\sum_k\mathbf s_k$ in $H^{-1/2}$, and define
\begin{equation}
	R_0(z)=(D-z)^{-1},\enspace
	G(z)\mathbf s=R_0(z)\Pi\Sigma\mathbf s,\enspace
	\mathcal Vf=(u_{0,k} f)_k,\enspace B(z)=\mathcal VG(z).
	\label{eq:NS-factorization-maps}
\end{equation}
Choose $\eta\in C_c^\infty(\R)$ equal to one on a neighborhood of $K_0$, and
set $\eta_k(x)=\eta(x-k)$.  For $\varphi\in C_c^\infty(\R)$, define the
product $u_{0,k} f$ in \eqref{eq:NS-factorization-maps} by
$\langle u_{0,k} f,\varphi\rangle
=\langle u_{0,k},\overline{\eta_kf}\,\varphi\rangle$.
The support condition makes this independent of the enlarged cutoff;
\eqref{eq:NS-local-product} gives its unique extension to $H^{-1/2}(\R)$.

\begin{lemma}[Weighted gain and compactness of the potential factor]
	\label{lem:NS-moment-gain}
	Under \eqref{eq:NS-assumptions}, for every $\delta\ge0$,
	\begin{equation}
		\mathcal V:X_\delta\longrightarrow\mathcal Y_{\be+\delta},
		\enspace
		\|\mathcal Vf\|_{\mathcal Y_{\be+\delta}}
		\le C\mathcal Q_{s,\be}(u_0)\|f\|_{X_\delta}.
		\label{eq:NS-moment-gain}
	\end{equation}
	Moreover, $\mathcal V:X_0\to\mathcal Y_\be$ is compact.
\end{lemma}

\begin{proof}[Proof outline]
The localized product estimate is
\begin{equation}
 \|u_{0,k}f\|_{H^{-1/2}}
 \le Cc_k\|\chi_kf\|_{H^{1/2}}.
 \label{eq:NS-local-product}
\end{equation}
Weighted square summation gives \eqref{eq:NS-moment-gain}.
For each fixed block, the Sobolev-conjugated multiplication operator is
Hilbert--Schmidt: its squared Hilbert--Schmidt norm is bounded by
$C\int_\R|\widehat u_{0,k}(r)|^2\log(2+|r|)\langle r\rangle^{-1}\,\rmd r
\le C_sc_k^2$.
Thus $P_N\mathcal V$ is compact, where $P_N$ retains $|k|\le N$, and
\begin{equation}
 \|(I-P_N)\mathcal V\|_{X_0\to\mathcal Y_\be}^2
 \le C\sum_{|k|>N}\langle k\rangle^{2\be}c_k^2\longrightarrow0.
 \label{eq:NS-compact-tail}
\end{equation}
The norm limit is compact.
\end{proof}

\subsection{Boundary resolvents and Fredholm inversion}

\begin{lemma}[Boundary resolvent estimates with cancellation at the spectral parameter]
	\label{lem:NS-radiation-tail}
	Fix $J\Subset(0,\infty)$ compact.
	\begin{enumerate}[(i)]
		\item If $\sigma>\frac12$, then
		$G(z):\mathcal Y_\sigma\to X_0$ is uniformly bounded for
		$\Re z\in J$, $0<|\Im z|\le1$.  As
		$z\to\lambda\pm i0$, it converges strongly on each fixed finite family of
		spatial blocks, uniformly for $\lambda\in J$; denote the limit by
		$G^\pm(\lambda)$.
		\item Let $\frac12<\sigma<1$, $\mathbf s\in\mathcal Y_\sigma$, and
		$\gamma=\widehat{\Sigma\mathbf s}$.  If $\gamma(\lambda)=0$, then
		$\|G^\pm(\lambda)\mathbf s\|_{X_{\sigma-1/2}}
		\le C_{J,\sigma}\|\mathbf s\|_{\mathcal Y_\sigma}$.
		\item If $\sigma\ge1$ and $\gamma(\lambda)=0$, then
		$G^\pm(\lambda)\mathbf s\in H_+^{1/2}$ and
		$\|G^\pm(\lambda)\mathbf s\|_{H^{1/2}}
		\le C_{J,\sigma}\|\mathbf s\|_{\mathcal Y_\sigma}$.
	\end{enumerate}
\end{lemma}

\begin{proof}[Proof outline]
Choose $J\Subset I\Subset(0,\infty)$ and
$\chi_{\rm E}\in C_c^\infty(I)$ equal to one near $J$.
For $S=\Sigma\mathbf s$ and $u=\chi_{\rm E}(D)S$, weighted block calculus gives
\begin{equation}
 \sum_j\langle j\rangle^{2\sigma}\|\chi_ju\|_{H^m}^2
 \le C_{m,\chi_{\rm E},\sigma}\|\mathbf s\|_{\mathcal Y_\sigma}^2,
 \qquad m\in\mathbb N_0.
 \label{eq:NS-bandlimited-weight}
\end{equation}
In particular, for $\sigma>1/2$,
\begin{equation}
 \|u\|_1\le C_\sigma\|\mathbf s\|_{\mathcal Y_\sigma}.
 \label{eq:NS-u-L1}
\end{equation}
Set $T_z=(D-z)^{-1}$ on the full line and
$K_z=\Pi(1-\chi_{\rm E}(D))(D-z)^{-1}$.
Then $G(z)\mathbf s=T_zu+K_zS$, and the pole term has the Volterra formulas
\begin{align}
 T_{\lambda+i\epsilon}u(x)
 &=i\int_{-\infty}^x e^{i(\lambda+i\epsilon)(x-y)}u(y)\,\rmd y,
 \label{eq:NS-Volterra-plus}\\
 T_{\lambda-i\epsilon}u(x)
 &=-i\int_x^\infty e^{i(\lambda-i\epsilon)(x-y)}u(y)\,\rmd y.
 \label{eq:NS-Volterra-minus}
\end{align}
Their bounded kernels and $(D-z)T_zu=u$ give uniform local $H^1$ bounds;
dominated convergence gives the blockwise boundary limits.
Separate the threshold contribution in $K_z$ by a cutoff near zero.
The remaining symbol has order $-1$, so the multiplier and weighted
$\Pi$ bounds give
\begin{equation}
 \|K_{\lambda\pm i0}S\|_{X_\delta}
 \le C\|\mathbf s\|_{\mathcal Y_\sigma},
 \qquad 0\le\delta<\min\{\sigma,1/2\}.
 \label{eq:NS-K-weighted}
\end{equation}
The same bounds and local convergence hold in the nonreal strip, proving (i).

By \eqref{eq:NS-bandlimited-weight}, $\chi_{\rm E}\gamma\in H^\sigma$;
thus the cancellation condition is well defined and reads
\begin{equation}
 \int_\R e^{-i\lambda y}u(y)\,\rmd y=0.
 \label{eq:NS-zero-amplitude}
\end{equation}
It converts the Volterra formulas into outward tails on both half-lines.
Weighted Cauchy--Schwarz then yields
\begin{equation}
 \|\chi_kT_{\lambda\pm i0}u\|_{H^1}
 \le C\langle k\rangle^{1/2-\sigma}
 \|\mathbf s\|_{\mathcal Y_\sigma}.
 \label{eq:NS-outward-tail-sum}
\end{equation}
For $1/2<\sigma<1$, combine this with
\eqref{eq:NS-K-weighted} at $\delta=\sigma-1/2$ to obtain (ii).
For $\sigma\ge1$, the weighted bound gives $xu\in L^2$; Hardy's inequality
on the outward primitives gives $T_{\lambda\pm i0}u\in L^2$, and its
equation gives $H^1$ regularity.  The off-pole term lies in $H^{1/2}$,
which proves (iii).
\end{proof}

For $\lambda\in J$, define
$B^\pm(\lambda)\mathbf s
=\lim_{\epsilon\downarrow0}B(\lambda\pm i\epsilon)\mathbf s$
in $\mathcal Y_\be$.
For $\lambda_n\to\lambda$, $\epsilon_n\downarrow0$, and fixed
$\mathbf s\in\mathcal Y_\be$, the Volterra formulas give convergence on
finitely many output blocks, while
$\|(I-P_N)B(z)\mathbf s\|_{\mathcal Y_\be}
\le\|(I-P_N)\mathcal V\|_{X_0\to\mathcal Y_\be}
\|G(z)\mathbf s\|_{X_0}\to0$
uniformly in the boundary strip by \eqref{eq:NS-compact-tail} and
Lemma~\ref{lem:NS-radiation-tail}(i).  Thus
\begin{equation}
 B(\lambda_n\pm i\epsilon_n)\mathbf s\longrightarrow B^\pm(\lambda)\mathbf s
 \enspace\hbox{in }\mathcal Y_\be.
 \label{eq:NS-joint-fixed-vector-boundary}
\end{equation}

\begin{lemma}[Collective compactness and boundary injectivity]
	\label{lem:NS-Fredholm-kernel}
	There exists $\epsilon_0>0$ such that the family
	\[
\left\{B(\lambda+i\epsilon):\lambda\in J,
		0<|\epsilon|\le\epsilon_0\right\}
		\cup\left\{B^\pm(\lambda):\lambda\in J\right\}
\]
	is collectively compact on $\mathcal Y_\be$.  Every operator
	$I-B^\pm(\lambda)$ is invertible, and
	\begin{equation}
		\sup_{\substack{\lambda\in J\\0<|\epsilon|\le\epsilon_0}}
		\|(I-B(\lambda+i\epsilon))^{-1}\|_{
			\mathcal Y_\be\to\mathcal Y_\be}<\infty.
		\label{eq:NS-uniform-Fredholm-inverse}
	\end{equation}
\end{lemma}

\begin{proof}[Proof outline]
The uniform map $G(z):\mathcal Y_\be\to X_0$ followed by compact
$\mathcal V:X_0\to\mathcal Y_\be$ proves collective compactness.
For nonreal $z$, a kernel vector gives $f=G(z)\mathbf s\in H_+^{1/2}$ and
\begin{equation}
 (D-z)f=\Pi\Sigma\mathbf s=\Pi(u_0f).
 \label{eq:NS-nonreal-kernel-equation}
\end{equation}
Self-adjointness and the Fredholm alternative give invertibility.

At the boundary, let $\mathbf s=B^\pm(\lambda)\mathbf s$,
$f=G^\pm(\lambda)\mathbf s$, and $\gamma=\widehat{\Sigma\mathbf s}$.
The estimate
\begin{equation}
 \sum_k\|\mathbf s_k\|_{H^{-1/2}}
 \le\left(\sum_k\langle k\rangle^{-2\be}\right)^{1/2}
 \|\mathbf s\|_{\mathcal Y_\be}<\infty
 \label{eq:NS-absolute-source-tail}
\end{equation}
justifies passage from blockwise convergence of
$f_\epsilon=G(\lambda\pm i\epsilon)\mathbf s$ to the source pairing.
Its limit is real because $\mathbf s_k=u_{0,k}f$ and $u_{0,k}$ is real.
The Poisson kernel identity for its imaginary part therefore gives
$0=\mp\pi|\gamma(\lambda)|^2$.

Starting with $\sigma_0=\be$, the cancellation estimate and the potential
moment gain iterate, while $\sigma_n<1$, as
\[
 \mathbf s\in\mathcal Y_{\sigma_n}
 \Longrightarrow f\in X_{\sigma_n-1/2}
 \Longrightarrow\mathbf s\in\mathcal Y_{\sigma_{n+1}},
 \qquad \sigma_{n+1}=\sigma_n+\be-\tfrac12.
\]
After finitely many steps $\sigma_N\ge1$, and
Lemma~\ref{lem:NS-radiation-tail}(iii) gives $f\in H_+^{1/2}$.
The boundary equation and form representation then make $f$ an
$L_{u_0}$ eigenvector at $\lambda>0$.
Proposition~\ref{prop:negative-spectrum} excludes this, proving boundary
injectivity and hence invertibility.

If the inverse bounds failed, compact nonreal substrips and normalization
would give $z_n\to\lambda\pm i0$ and unit vectors $\mathbf s_n$ with
$(I-B(z_n))\mathbf s_n\to0$.
Collective compactness gives $\mathbf s_n\to\mathbf s$ along a subsequence;
\eqref{eq:NS-joint-fixed-vector-boundary} then gives
$\mathbf s=B^\pm(\lambda)\mathbf s$ with $\|\mathbf s\|=1$,
a contradiction.
\end{proof}

For $R>1$, let $m_R(x)=R(R^2+x^2)^{-1/2}$,
$W_R=M_{m_R}:L_+^2\to L^2(\R)$, and $W_R^*=\Pi M_{m_R}$.
Then $m_R\in L^2$, $m_R>0$ almost everywhere, and
$m_R(x)\lesssim_R\langle x\rangle^{-1}$.
For $0\le\sigma<1$ define $(J_Rg)_k=\theta_km_Rg$.
Uniform local finiteness gives
\begin{equation}
	\|J_Rg\|_{\mathcal Y_\sigma}^2
	\le\sum_k\langle k\rangle^{2\sigma}\|\theta_km_Rg\|_2^2
	\le C_{R,\sigma}\sum_k\|\theta_kg\|_2^2
	\le C_{R,\sigma}\|g\|_2^2,
	\enspace \Sigma J_Rg=m_Rg.
	\label{eq:NS-source-injection-bound}
\end{equation}
Here the middle inequality uses
$\sup_{x\in K_k}\langle k\rangle^\sigma m_R(x)\le C_{R,\sigma}$ for
$\sigma<1$.

\begin{lemma}[Uniform free two-weight bounds]
	\label{lem:NS-free-two-weight}
	Uniformly for $\Re z\in J$ and $0<|\Im z|\le1$,
	\begin{align}
		W_R&:X_0\longrightarrow L^2,
		\label{eq:NS-W-X}\\
		R_0(z)W_R^*&:L^2\longrightarrow X_0,
		\label{eq:NS-R0W-X}\\
		W_RR_0(z)W_R^*&:L^2\longrightarrow L^2
		\label{eq:NS-free-sandwich}
	\end{align}
	are bounded, with bounds $C_{J,R}$.
\end{lemma}

\begin{proof}[Proof outline]
Square summation over unit blocks, using
$m_R(x)^2\lesssim_R\langle x\rangle^{-2}$, proves
\eqref{eq:NS-W-X}.
For $1/2<\sigma_0<1$, use $R_0(z)W_R^*=G(z)J_R$.
The source estimate \eqref{eq:NS-source-injection-bound} and
Lemma~\ref{lem:NS-radiation-tail}(i) give \eqref{eq:NS-R0W-X}.
Composition gives
\eqref{eq:NS-free-sandwich}.
\end{proof}

Define four bounded operators
\begin{align*}
	F_R(z)&=W_RR_0(z)W_R^*,
	& A_R(z)&=W_RG(z):\mathcal Y_\be\to L^2,\\
	C_R(z)&=\mathcal VR_0(z)W_R^*:L^2\to\mathcal Y_\be,
	& S_{u_0,R}(z)&=W_R(L_{u_0}-z)^{-1}W_R^*.
\end{align*}
The preceding lemmas bound $F_R,A_R,C_R$ uniformly in the same boundary strip.

\begin{prop}[Sandwiched Woodbury resolvent formula]
	\label{prop:NS-Woodbury}
	For $\Im z\neq0$,
	\begin{equation}
		S_{u_0,R}(z)=F_R(z)+
		A_R(z)(I-B(z))^{-1}C_R(z).
		\label{eq:NS-Woodbury}
	\end{equation}
	Consequently,
	\begin{equation}
		\sup_{\substack{\Re z\in J\\0<|\Im z|\le1}}
		\|W_R(L_{u_0}-z)^{-1}W_R^*\|_{L^2\to L^2}<\infty.
		\label{eq:NS-sandwiched-LAP}
	\end{equation}
\end{prop}

\begin{proof}[Proof outline]
For $g\in L^2$ and $f=(L_{u_0}-z)^{-1}W_R^*g$, Proposition~\ref{prop:closed-Lax}
gives $f\in H_+^{1+s}\hookrightarrow X_0$.
With $\mathbf s=\mathcal Vf$, the form equation becomes
\[
 f=R_0(z)W_R^*g+G(z)\mathbf s,
 \qquad (I-B(z))\mathbf s=C_R(z)g.
\]
Fredholm inversion and multiplication by $W_R$ give
\eqref{eq:NS-Woodbury}.
The uniform bounds for its factors prove the LAP near the real axis;
away from it use the self-adjoint resolvent bound
$\|(L_{u_0}-z)^{-1}\|\le |\Im z|^{-1}$.
\end{proof}

\begin{coro}[Weighted spectral densities]\label{coro:NS-weighted-densities}
Under \eqref{eq:NS-assumptions}, the spectral measure of $W_R^*g$,
$g\in L^2(\mathbb R)$, has a bounded density on every compact positive-energy interval.
Moreover, $E_{u_0,\mathrm{sing}}(J)=0$ for every compact $J\subset(0,\infty)$.
\end{coro}
\begin{proof}[Proof outline]
Put $h=W_R^*g$ and $\mu_h(B)=\langle h,E_{u_0}(B)h\rangle$.
The spectral theorem and \eqref{eq:NS-sandwiched-LAP} bound the Poisson
integral of $\mu_h$ by $C_{J,R}\|g\|_2^2$.
Integrate against nonnegative test functions supported in
$\operatorname{int}J$, then let $\epsilon\downarrow0$.
For $J_0\Subset\operatorname{int}J$ this gives
\begin{equation}
 \mu_h(A)\le\pi^{-1}C_{J,R}\|g\|_2^2|A|,
 \qquad A\subset J_0\text{ Borel}.
 \label{eq:NS-density-bound}
\end{equation}
Thus every $W_R^*g$ has bounded spectral density there.
Since $m_R>0$ almost everywhere, $\operatorname{Ran}W_R^*$ is dense.
The singular projection vanishes on this range, hence everywhere;
exhausting $(0,\infty)$ proves the conclusion.
\end{proof}

\section{Trace-class comparison and the positive absolutely continuous spectrum}
\label{sec:trace-transform}

The trace-class comparison below uses \eqref{eq:NS-assumptions};
Theorem~\ref{thm:main-no-singular} applies to all real $H^s$ potentials,
$s>-\frac12$.

\subsection{Trace-class resolvents and wave operators}

Write $\langle D_x\rangle=(1-\partial_x^2)^{1/2}$ and
$w(x)=\langle x\rangle^{-\be}$.

\begin{lemma}[Bessel--L\'evy kernel and increment estimates]
	\label{lem:SC-Bessel-Levy-common}
	Let $-\frac12<s<0$ and define
	\[
b_s(t)=\int_0^\infty e^{-v}v^{-3/2+s/2}
		e^{-t^2/(4v)}\,\rmd v.
\]
	For Schwartz $g$,
	\begin{align}
		\langle D_x\rangle^{-s} g(x)
		&=g(x)+c_s\int_\R
		\bigl(g(x)-g(x-t)\bigr)b_s(t)\,\rmd t,
		\label{eq:SC-Bernstein-formula}\\
		\langle r\rangle^{-s}
		&=1+c_s\int_\R
		\bigl(1-\cos(tr)\bigr)b_s(t)\,\rmd t.
		\label{eq:SC-Levy-representation}
	\end{align}
	Moreover,
	\begin{equation}
		b_s(t)\lesssim_s |t|^{-1+s}\lesssim_s b_s(t)\enspace(0<|t|\le1),
		\enspace b_s(t)\le C_s e^{-c|t|}\enspace(|t|\ge1).
		\label{eq:SC-Bessel-kernel-bounds}
	\end{equation}
	If $w\in H^1(\R)$, then
	\begin{equation}
		\|w(\,\cdot\,)-w(\,\cdot-t\,)\|_2
		\le\min\{|t|\|w'\|_2,\,2\|w\|_2\}.
		\label{eq:SC-weight-translation}
	\end{equation}
	Finally, if $u(\xi)=(\xi+\ka)^{-1}$ on $(0,\infty)$ and
	$u_t(\xi)=e^{it\xi}u(\xi)$, then
	\begin{align}
		\||a\rangle\langle a|-|b\rangle\langle b|\|_{\mathfrak S_1}
		&\le(\|a\|_2+\|b\|_2)\|a-b\|_2,
		\label{eq:SC-rank-one-difference}\\
		\|u_t-u\|_2^2&\le C_\ka\min\{|t|,1\}.
		\label{eq:SC-modulation-half}
	\end{align}
\end{lemma}

\begin{proof}[Proof outline]
Subordination and the Gaussian Fourier formula give
\eqref{eq:SC-Bernstein-formula}--\eqref{eq:SC-Levy-representation}.
Rescaling the subordination variable proves the power bound near zero
and exponential decay at infinity.
The fundamental theorem of calculus gives the translation bound, while
splitting the rank-one difference gives
\eqref{eq:SC-rank-one-difference}.
For \eqref{eq:SC-modulation-half}, split the frequency integral at
$|t|^{-1}$ when $|t|\le1$; for larger $t$, use the uniform $L^2$ bound.
\end{proof}

\begin{lemma}[Endpoint weighted multiplier estimate]
	\label{lem:SC-endpoint-multiplier}
	If $-\frac12<s<0$ and $\be>1/2$, then
	$M_w\langle D_x\rangle^{s}:L^\infty(\R)\longrightarrow H^{-s}(\R)$ is bounded.
\end{lemma}

\begin{proof}[Proof outline]
The negative Bessel kernel maps $L^\infty$ to itself.
Apply \eqref{eq:SC-Bernstein-formula} to the commutator with $w$:
Minkowski's inequality bounds it by
\[
 C_s\|h\|_\infty\int_\R
 \min\{|t|\|w'\|_2,2\|w\|_2\}b_s(t)\,\rmd t.
\]
This is finite since $w,w'\in L^2$ and $s>-1$.
Together with the term $wh$, this bounds
$\langle D_x\rangle^{-s}w\langle D_x\rangle^sh$ in $L^2$.
Weak-star approximation and uniqueness of the distributional limit
extend the estimate to all $h\in L^\infty$.
\end{proof}

\begin{lemma}[$L^1$ Bessel primitive]
	\label{lem:SC-L1-primitive}
	Under \eqref{eq:NS-assumptions},
	\begin{equation}
		p:=\langle D_x\rangle^{s}u_0\in L^1(\R),
		\enspace
		\|p\|_1\le C_{s,\be}
		\|\langle x\rangle^\be u_0\|_{H^{s}}.
		\label{eq:SC-L1-primitive}
	\end{equation}
\end{lemma}

\begin{proof}[Proof outline]
Set $v=\langle x\rangle^\be u_0$, so $u_0=wv$.
Duality and Lemma~\ref{lem:SC-endpoint-multiplier} give
\begin{equation}
 |\langle p,h\rangle|
 =|\langle v,w\langle D_x\rangle^sh\rangle|
 \le C_{s,\be}\|v\|_{H^s}\|h\|_\infty.
 \label{eq:SC-L1-duality}
\end{equation}
Since $p\in L^2$, bounded smooth approximations to its complex sign
on $[-R,R]$ show a uniform bound for $\int_{-R}^R|p|$.
Letting $R\to\infty$ proves the claim.
\end{proof}

\begin{lemma}[Trace-class half-line fractional kernel]
	\label{lem:SC-fractional-kernel}
	Let $\ka>0$ and $-1/2<s<0$.  The integral operator on $L^2(0,\infty)$
	with kernel
	\begin{equation}
		K_{s,\ka}(\xi,\eta)
		=\frac{\langle\xi-\eta\rangle^{-s}}
		{(\xi+\ka)(\eta+\ka)}
		\label{eq:SC-fractional-kernel}
	\end{equation}
	belongs to the trace ideal $\mathfrak S_1$.
\end{lemma}

\begin{proof}[Proof outline]
Set $P_t=|u_t\rangle\langle u_t|$ and
$Q_t=P_0-\frac12(P_t+P_{-t})$.
The increment estimates give
$\|Q_t\|_{\mathfrak S_1}\lesssim_\ka\min\{|t|^{1/2},1\}$.
Hence the L\'evy representation yields the trace-norm convergent integral
\[
 K_{s,\ka}=P_0+c_s\int_\R Q_tb_s(t)\,\rmd t,
\]
since its integral near zero is bounded by
$\int_0^1t^{-1/2+s}\,\rmd t<\infty$.
Truncation identifies its kernel with \eqref{eq:SC-fractional-kernel}.
\end{proof}

\begin{prop}[The first Born term is trace class]
	\label{prop:SC-first-Born}
	For every $\ka>0$,
	$(D+\ka)^{-1}T_{u_0}(D+\ka)^{-1}\in\mathfrak S_1.$
\end{prop}

\begin{proof}[Proof outline]
Let $c_{\mathrm F}=(2\pi)^{-1/2}$,
$p=\langle D_x\rangle^su_0\in L^1$, and
$(U_xf)(\xi)=e^{-ix\xi}f(\xi)$.
Fourier inversion gives
\begin{equation}
 (D+\ka)^{-1}T_{u_0}(D+\ka)^{-1}
 =c_{\mathrm F}^2\int_\R p(x)U_xK_{s,\ka}U_x^*\,\rmd x.
 \label{eq:SC-Born-Bochner}
\end{equation}
Unitary invariance of trace norm gives
\begin{equation}
 \int_\R|p(x)|\|U_xK_{s,\ka}U_x^*\|_{\mathfrak S_1}\,\rmd x
 =\|p\|_1\|K_{s,\ka}\|_{\mathfrak S_1}<\infty.
 \label{eq:SC-Born-integrability}
\end{equation}
Approximate $p$ in $L^1\cap L^2$: the right side converges in trace norm,
and the associated potentials converge in $H^s$.
The product-form bound therefore passes the smooth identity to $u_0$.
\end{proof}

\begin{prop}[Trace-class resolvent difference]
	\label{prop:SC-resolvent-trace-class}
	For all sufficiently large $\ka$,
	$(L_{u_0}+\ka)^{-1}-(D+\ka)^{-1}\in\mathfrak S_1.$
\end{prop}

\begin{proof}[Proof outline]
Set $B_\ka=(D+\ka)^{-1/2}T_{u_0}(D+\ka)^{-1/2}$.
Its Fourier kernel gives
\begin{equation}
 \|B_\ka\|_{\mathfrak S_2}^2
 =c_{\mathrm F}^2\int_\R|\widehat u_0(r)|^2I_\ka(r)\,\rmd r,
 \qquad I_\ka(r)=\frac{\log(1+|r|/\ka)}{|r|}.
 \label{eq:SC-B-HS}
\end{equation}
Here $I_\ka(0)=\ka^{-1}$ and
$I_\ka(r)\lesssim_s\langle r\rangle^{2s}$ for $\ka\ge1$.
Approximation and dominated convergence yield
\begin{equation}
 \|B_\ka\|_{\mathfrak S_2}\longrightarrow0\quad(\ka\to\infty).
 \label{eq:SC-B-small}
\end{equation}
Choose $\ka$ large enough for $\|B_\ka\|<1$, including the shift in
Proposition~\ref{prop:global-orbit-control}.  Closed-form factorization gives
\begin{equation}
 (L_{u_0}+\ka)^{-1}
 =(D+\ka)^{-1/2}(I-B_\ka)^{-1}(D+\ka)^{-1/2}.
 \label{eq:SC-resolvent-form-factorization}
\end{equation}
Expand $(I-B)^{-1}-I=B+B(I-B)^{-1}B$.
The linear term is trace class by Proposition~\ref{prop:SC-first-Born};
the remainder is trace class by
$\mathfrak S_2\mathcal B\mathfrak S_2\subset\mathfrak S_1$.
\end{proof}

\begin{theorem}[Completeness of the wave operators and multiplicity one]
	\label{thm:SC-complete-channel}
	Under \eqref{eq:NS-assumptions}, the ordinary wave operators
	\begin{equation}
		W_\pm(L_{u_0},D)
		=\operatorname*{s-lim}_{t\to\pm\infty}
		e^{itL_{u_0}}e^{-itD}
		\label{eq:SC-wave-operators}
	\end{equation}
	exist and are complete.  The restriction of $L_{u_0}$ to
	$E_{u_0}((0,\infty))L_+^2$ is unitarily equivalent to multiplication by
	$\lambda$ on $L^2(0,\infty)$.  In particular, its positive spectrum is
	purely absolutely continuous and has multiplicity one.
\end{theorem}

\begin{proof}[Proof outline]
Apply Kato--Rosenblum \cite{Kato1957,Rosenblum} to the trace-class
difference of the shifted inverse operators.
The invariance principle \cite[Chapter~6]{Yafaev}, followed by removal
of the common shift, gives complete ordinary wave operators.
Completeness and intertwining identify the absolutely continuous part
with $D=M_\xi$ on $L^2(0,\infty)$, which has multiplicity one.
Section~\ref{sec:no-singular} excludes positive singular spectrum;
exhaustion of $(0,\infty)$ completes the identification.
\end{proof}

\subsection{Proof of Theorem~\ref{thm:main-no-singular}: the positive absolutely continuous spectrum}

Fix a real $v:=u_0\in H^s(\R)$, $s>-\frac12$, and choose
$-\tfrac12<\rho<\min\{s,0\},\enspace \sigma=1+\rho\in(\tfrac12,1),\enspace \Pi_-=I-\Pi.$
The closed form in Proposition~\ref{prop:closed-Lax}, based on
\cite{BKV,KLV}, depends only on $v$ and gives the same $L_v$ for every
such $\rho$. Throughout this subsection, $J\Subset(0,\infty)$ and
$\langle f,h\rangle=\int\overline f h$, with extension by local
Sobolev duality where needed.

Choose real smooth unit-scale cutoffs $\chi_j$, $j\in\Z$, with
$\sum_j\chi_j^2=1$, uniformly bounded derivatives, and finite overlap.
Put
$a_j^r(f)=\|\chi_j f\|_{H^r}$, $\|p\|_{\mathscr Z_\rho}=\sum_j a_j^\rho(p)$, and $\|f\|_{\mathscr W_\sigma}=\sup_j a_j^\sigma(f)$.
Cutoffs on the right of local estimates may be enlarged by a fixed
factor. Smooth localization gives
$\|f\|_{H^r}^2\asymp\sum_j a_j^r(f)^2$ for the exponents used below.
With $\tau=-\rho$, we have
$\sigma=1-\tau>\max\{\tau,\frac12\}$. The product estimate
$H^\sigma H^\tau\subset H^\tau$ and its dual
$H^\rho H^\sigma\subset H^\rho$
\cite[Theorems~5.1 and~8.1]{BHnew}, followed by localization, yield
\begin{align}
 \|af\|_{H^\rho}
 &\le C_\rho\|a\|_{H^\rho}\|f\|_{\mathscr W_\sigma},
 \label{UW-product-uniform}\\
 \|af\|_{\mathscr Z_\rho}
 &\le C_\rho\|a\|_{H^\rho}\|f\|_{H^\sigma}.
 \label{UW-product-sum}
\end{align}
The second inequality uses $\ell^2\cdot\ell^2\subset\ell^1$.
The same argument gives bounded multiplication by $\mathscr W_\sigma$ on
$H^\sigma$, $H^\rho$, $H^{-\rho}$, and
$\mathscr Z_\rho$.

Fix a real primitive $\Phi'=v$ and set $\mathfrak g=e^{i\Phi}$.
The local primitive estimate and the fractional difference quotient
formula give
\[
 \mathfrak g,\overline{\mathfrak g}\in\mathscr W_\sigma,
 \enspace
 \|\mathfrak g\|_{\mathscr W_\sigma}
 +\|\overline{\mathfrak g}\|_{\mathscr W_\sigma}
 \le C_\rho(1+\|v\|_{H^\rho}).
\]
Local smooth approximation gives $D\mathfrak g=v\mathfrak g$.
Consequently the self-adjoint operator
\[
\mathcal A_v=M_{\mathfrak g}DM_{\overline{\mathfrak g}}, \enspace \operatorname{Dom}\mathcal A_v=\mathfrak g H^1(\R),
\]
acts as $D-v$ in distributions. For nonreal $z$, its resolvent
$R_{\mathcal A_v}(z) =M_{\mathfrak g}(D-z)^{-1}M_{\overline{\mathfrak g}}$
extends boundedly from $H^\rho$ to $H^\sigma$. The local chain rule
gives
$(D-v-z)(\mathfrak g f)=\mathfrak g(D-z)f$ for $f\in H^\sigma$;
hence this extension inverts $D-v-z:H^\sigma\to H^\rho$.

For $p\in\mathscr Z_\rho$, the outgoing boundary value is
\begin{equation}
 R_{\mathcal A_v}^+(\lambda)p(x)
 =i\mathfrak g(x)e^{i\lambda x}
   \int_{-\infty}^x e^{-i\lambda y}
          \overline{\mathfrak g(y)}p(y)\,dy.
 \label{UW-gauge-resolvent}
\end{equation}
The integrals are distributional block pairings. Their absolute
convergence follows from $-\rho<1/2$ and the summable source norm.
For enlarged cutoffs
$\widetilde\chi_j$, the local primitive estimate
\begin{equation}
 \|\chi_j F\|_{H^{1+\rho}}
 \le C\bigl(\|F\|_{L^2(\operatorname{supp}\widetilde\chi_j)}
       +\|\widetilde\chi_jF'\|_{H^\rho}\bigr)
 \label{UW-primitive-estimate}
\end{equation}
follows by differentiating $\chi_jF$ and comparing Fourier weights.
It proves the uniform bounds
\begin{equation}
 R_{\mathcal A_v}^{\pm}(z):
 \mathscr Z_\rho\longrightarrow\mathscr W_\sigma,
 \enspace \operatorname{Re}z\in J,\enspace
 0\le\pm\operatorname{Im}z\le\varepsilon.
 \label{UW-source-bound}
\end{equation}
Here the lower boundary uses $-i\int_x^\infty$, and the nonreal kernels
use $z$ in place of $\lambda$. For each fixed source, truncation of its
$\mathscr Z_\rho$ tail and continuity of the local exponential kernels
give continuity in $H^\sigma_{\mathrm{loc}}$. The uniform
$\mathscr W_\sigma$ bound upgrades this to continuity in $H^\rho$
after multiplication by $v$, and to continuity of pairings with any
fixed $\mathscr Z_\rho$ source.

Set $R_-(z)=(D-z)^{-1}\Pi_-$. The negative-frequency gap makes
$R_-(z):H^\rho\to H^\sigma$ norm analytic on
$\operatorname{Re}z>0$. Applying the extended gauge inverse to
$(D-v-z)R_-(z)q=\Pi_-q-vR_-(z)q$ gives
\[
K(z):=vR_{\mathcal A_v}(z)\Pi_-
 =vR_-(z)+vR_{\mathcal A_v}(z)vR_-(z)
 \enspace\hbox{on }H^\rho.
\]
The right side defines its boundary values $K_\pm(\lambda)$.
The first term is compact by smooth compactly supported approximation
of $v$ and local Rellich compactness. The second term factors as
\[
 H^\rho\xrightarrow{R_-}H^\sigma
 \xrightarrow{v\,\cdot}\mathscr Z_\rho
 \xrightarrow{R_{\mathcal A_v}^{\pm}}\mathscr W_\sigma
 \xrightarrow{v\,\cdot}H^\rho.
\]
With the full gauge fixed, approximate the two potential factors by
smooth compactly supported functions. The product and source bounds
make the approximation uniform in the closed half-strips. Local
Rellich compactness gives compactness of each approximant, while
continuity of its kernel on fixed supports gives norm continuity.
Thus $K_\pm(\lambda)$ are compact and $K(z)$ extends to them locally
continuously in $\mathcal B(H^\rho)$.

Suppose $(I+K_+(\lambda))q=0$, $q\in H^\rho$. Define
\begin{equation}
 r=R_-(\lambda)q\in H^\sigma,
 \enspace p=vr\in\mathscr Z_\rho,
 \enspace G=-R_{\mathcal A_v}^+(\lambda)p,
 \enspace f=-r+G.
 \label{UW-kernel-fields}
\end{equation}
Then $q=vf$ and $f,G\in\mathscr W_\sigma$. Reality of $v$ gives
\begin{equation}
 \langle p,G\rangle
 =\langle r,q\rangle+\langle r,vr\rangle\in\R,
 \enspace
 \langle r,q\rangle
 =\int_{\xi<0}\frac{|\widehat q(\xi)|^2}{\xi-\lambda}\,d\xi.
 \label{UW-real-flux}
\end{equation}
These pairings converge by $2\rho>-1$, the local product bounds,
and square summation.

The continuous functional
$\mathcal M_\lambda(p) =\int_\R e^{-i\lambda y}\overline{\mathfrak g(y)}p(y)\,dy$
on $\mathscr Z_\rho$ satisfies the flux identity
\begin{equation}
 \operatorname{Im}\langle p,R_{\mathcal A_v}^+(\lambda)p\rangle
 =\tfrac12|\mathcal M_\lambda(p)|^2.
 \label{UW-flux-identity}
\end{equation}
Apply the local chain rule to the primitive in
\eqref{UW-gauge-resolvent}, then integrate on finite intervals and
pass to its two limits to obtain the identity.
Together with \eqref{UW-real-flux}, this gives
\begin{equation}
 \mathcal M_\lambda(vr)=0.
 \label{UW-zero-shell}
\end{equation}

To prove boundary injectivity, divide the unit-block indices into dyadic sets $\mathcal I_n$,
$n\ge0$, according to $\langle j\rangle\asymp2^n$, with
$\#\mathcal I_n\lesssim2^n$, and write
$A_n^r(h)=\Bigl(\sum_{j\in\mathcal I_n}a_j^r(h)^2\Bigr)^{1/2}.$
Define
$\|r\|_{\mathscr R_\sigma} =\sup_n2^{n/2}A_n^\sigma(r),\enspace \|p\|_{\mathscr Y_\rho} =\|p\|_{\mathscr Z_\rho}+\sup_n2^{n/2}A_n^\rho(p).$
Finite enlargements of adjacent blocks give equivalent norms, and
$\mathscr R_\sigma\hookrightarrow H^\sigma$.

\begin{lemma}\label{UW-dyadic-resolvent}
Uniformly for $\lambda\in J$,
$\|R_-(\lambda)p\|_{\mathscr R_\sigma} \le C_{J,\rho}\|p\|_{\mathscr Y_\rho}.$
\end{lemma}
\begin{proof}[Proof outline]
The negative-frequency resolvent kernel satisfies
$|\partial_x^a k_\lambda(x)|\le C_{a,J}|x|^{-a-1}$ for $x\ne0$.
At scale $2^n$, the global $H^\rho\to H^\sigma$ bound controls
nearby source blocks by $C2^{-n/2}\|p\|_{\mathscr Y_\rho}$.
For distant blocks, kernel decay and local Sobolev duality give
$C2^{-n}\|p\|_{\mathscr Z_\rho}$ on each target block.
Square summation over $O(2^n)$ blocks gives the required dyadic bound;
the lowest scales use the global resolvent estimate.
\end{proof}

Choose real smooth $\zeta_-+\zeta_+=1$, with $\zeta_-=1$ on
$x\le-1$ and $\zeta_+=1$ on $x\ge1$. For $b\in H^\rho$ and
$r\in H^\sigma$, define
\begin{align*}
 \mathfrak M_b r(x)=-i\mathfrak g(x)e^{i\lambda x}
 \bigg[&\zeta_-(x)\int_{-\infty}^x
 e^{-i\lambda y}\overline{\mathfrak g(y)}b(y)r(y)\,dy\\
 &-\zeta_+(x)\int_x^\infty
 e^{-i\lambda y}\overline{\mathfrak g(y)}b(y)r(y)\,dy\bigg].
\end{align*}
Every $\mathfrak M_b$ uses the fixed full-potential gauge
$\mathfrak g$. This operator is linear in $b$, and
\eqref{UW-zero-shell} gives $G=\mathfrak M_vr$.
For compactly supported $b$, its output is supported in the interval
spanned by $\operatorname{supp}b$ and $[-1,1]$.

\begin{lemma}\label{UW-tail-estimates}
Uniformly for $\lambda\in J$,
\begin{align}
 \|\mathfrak M_b r\|_{\mathscr W_\sigma}
 &\le C_{\mathfrak g}\|b\|_{H^\rho}\|r\|_{H^\sigma},\notag\\
 \|\mathfrak M_b r\|_{H^\sigma}
 &\le C_{\mathfrak g}\|b\|_{H^\rho}\|r\|_{\mathscr R_\sigma},
 \label{UW-anchor-gain}\\
 \|br\|_{\mathscr Y_\rho}
 &\le C\|b\|_{H^\rho}\|r\|_{\mathscr R_\sigma},
 \label{UW-product-dyadic}\\
 \|b\mathfrak M_b r\|_{\mathscr Y_\rho}
 &\le C_{\mathfrak g}\|b\|_{H^\rho}^2
               \|r\|_{\mathscr R_\sigma}.
 \label{UW-quadratic-dyadic}
\end{align}
\end{lemma}
\begin{proof}[Proof outline]
Equations~\eqref{UW-product-sum} and \eqref{UW-primitive-estimate}
give the first bound. Put $B_n=A_n^\rho(b)$ and
$R=\|r\|_{\mathscr R_\sigma}$. The anchored primitive on scale $n$
is bounded by $G_n=C_{\mathfrak g}R\sum_{m\ge n-C_0}B_m2^{-m/2}$.
Discrete Young's inequality bounds $\sum_n2^nG_n^2$ by
$C_{\mathfrak g}R^2\sum_nB_n^2$. Together with its derivative bound,
this proves \eqref{UW-anchor-gain}. The local product estimate and
dyadic summation give \eqref{UW-product-dyadic}; applying the same
bounds to $b\mathfrak M_br$ gives \eqref{UW-quadratic-dyadic}.
\end{proof}

Split $v=v_c+w$ by a smooth compactly supported cutoff, so that
$v_c\in H^\rho$ has compact support and $\|w\|_{H^\rho}$ is small.
On both $H^\sigma$ and $\mathscr R_\sigma$ consider the same operator
$\mathfrak B_wr=R_-(\lambda)[w\mathfrak M_wr-wr].$
Lemmas~\ref{UW-dyadic-resolvent}--\ref{UW-tail-estimates} and
\eqref{UW-product-uniform} give
\begin{equation}
 \|\mathfrak B_w\|_{\mathcal B(H^\sigma)}
 +\|\mathfrak B_w\|_{\mathcal B(\mathscr R_\sigma)}
 \le C_{\mathfrak g}
 (\|w\|_{H^\rho}+\|w\|_{H^\rho}^2)<1.
 \label{UW-two-space-contraction}
\end{equation}
The constants use the fixed gauge of $v$, so this follows by moving
the cutoff sufficiently far out. The kernel field in
\eqref{UW-kernel-fields} satisfies
\[
 r=\mathfrak B_wr+b_0,\enspace b_0=R_-(\lambda)p_c,
 \enspace
 p_c=v_c\mathfrak M_vr+w\mathfrak M_{v_c}r-v_cr.
\]
The compactly supported source $p_c$ lies in $\mathscr Y_\rho$,
so $b_0\in\mathscr R_\sigma$. The Neumann series with this fixed
$b_0$ converges in $\mathscr R_\sigma$ and, by uniqueness in
$H^\sigma$, equals $r$. Hence
$r\in\mathscr R_\sigma$ and $G,f\in H^\sigma$.
The equation $(D-\lambda)f=\Pi q$ now implies $\Pi_-f=0$.
Testing on the Hardy form domain places $f$ in
$\operatorname{Dom}L_v$ with $L_vf=\lambda f$.
The absence of positive eigenvalues in \cite[Theorem~1.4]{BKV} gives
$f=q=0$. The incoming flux identity and the same anchored primitive
give injectivity for the other boundary sign.

\begin{proof}[Outline of the proof of Theorem~\ref{thm:main-no-singular}]
\leavevmode\par
Injectivity and compact Fredholm theory give uniformly bounded inverses
of $I+K_\pm$ on compact positive-energy intervals. Off the real axis,
self-adjointness excludes a homogeneous solution.
For
$\mathscr D_+=\{h\in\mathcal S(\R):\widehat h\in C_c^\infty((0,\infty))\}$,
set $q_z=(I+K(z))^{-1}vR_{\mathcal A_v}(z)h$ and
$r_z=R_-(z)q_z$. The resolvent is
\begin{equation}
 (L_v-z)^{-1}h
 =R_{\mathcal A_v}(z)h-r_z-R_{\mathcal A_v}(z)(vr_z).
 \label{UW-actual-resolvent}
\end{equation}
The right side belongs to $H_+^\sigma$ and solves the form resolvent
equation. Uniform source bounds imply that
$\langle h,(L_v-z)^{-1}h\rangle$ is bounded up to each compact interval
$J\Subset(0,\infty)$. The Poisson representation implies
$\mu_h|_J\ll d\lambda$; density of $\mathscr D_+$ eliminates the
singular spectral projection. Fixed-source continuity also gives the
boundary limits in \eqref{UW-actual-resolvent}.

To identify multiplicity, fix $\lambda>0$ and put
\[
 R^\pm=R_{\mathcal A_v}^\pm(\lambda),\quad
 T^\pm=R_-+R^\pm vR_-,\quad M^\pm=I+vT^\pm,\quad
 \psi=c_{\rm F}\mathfrak g e^{i\lambda x},\quad c_{\rm F}=(2\pi)^{-1/2}.
\]
The full-line boundary kernel has the rank-one jump
\begin{equation}
 (R^+-R^-)p=2\pi i\psi\langle\psi,p\rangle,
 \enspace p\in\mathscr Z_\rho.
 \label{UW-gauge-jump}
\end{equation}
The distributional equation defines the negative projection of the
boundary fields through $R_-$. Kernel duality then gives
\begin{equation}
 \langle\Pi_-R^+p,k\rangle=\langle p,T^-k\rangle,
 \enspace p\in\mathscr Z_\rho,\quad k\in H^\rho.
 \label{UW-mixed-adjoint}
\end{equation}
Define $q_+=(M^+)^{-1}v\psi$, $r_+=R_-q_+$,
$p_+=vr_+$, and $a_+=\psi-r_+-R^+p_+$.
Then $va_+=q_+$, $(D-\lambda)a_+=\Pi q_+$, and $\Pi_-a_+=0$.
The channel is nonzero and detected by $\mathscr D_+$, since a bounded
channel supported at zero frequency must vanish by its equation.

For $h\in\mathscr D_+$, write the two boundary resolvents as
$F^\pm h=R^\pm h-T^\pm(M^\pm)^{-1}vR^\pm h$.
The inverse difference identity, \eqref{UW-gauge-jump}, and
\eqref{UW-mixed-adjoint} yield
\[
 (F^+-F^-)h=2\pi i a_+\langle a_+,h\rangle.
\]
Stone's formula identifies the spectral density as
$\overline{\langle a_+(\lambda),h\rangle}\langle a_+(\lambda),k\rangle$.
Its matrices have rank at most one. Choose a countable family
$\{h_n\}\subset\mathscr D_+$ dense both in $L_+^2$ and on each compact
Fourier-support test space, and set
$c_n=2^{-n}/(1+\|h_n\|_2^2)$.
The maximal measure $\nu=\sum_nc_n\mu_{h_n}$ has density
$\sum_nc_n|\langle a_+(\lambda),h_n\rangle|^2$, positive at every
$\lambda>0$ and finite almost everywhere. Hence it is equivalent to
Lebesgue measure and the multiplicity is exactly one.
\end{proof}

\subsection{Weighted boundary resolvents and generalized eigenfunctions}

Under \eqref{eq:NS-assumptions}, fix $R>1$ independently of the compact energy interval $J$.
Regard $W_R$ as a map into $L^2(\R)$ with adjoint
$W_R^*=\Pi M_{m_R}$.  Define the dense Hilbert quotient
$\mathcal X=\operatorname{Ran}W_R^*$ with
$\|f\|_{\mathcal X}=\inf\{\|g\|_2:W_R^*g=f\}$.
Density of $\operatorname{Ran}W_R^*$ follows from $m_R>0$; see
Section~\ref{sec:no-singular}.

Let $e_\lambda(x)=(2\pi)^{-1/2}e^{i\lambda x}$ for $\lambda>0$.
If $f=W_R^*g\in\mathcal X$, define the free point trace
$\tau_\lambda f=\widehat{m_Rg}(\lambda)$.
This is independent of the preimage: if $W_R^*g=0$, the continuous Fourier
transform of $m_Rg\in L^1$ vanishes almost everywhere, hence everywhere,
on the positive half-line.
Cauchy--Schwarz gives
\begin{align}
	|\tau_\lambda f|&\le\|m_R\|_2\|g\|_2,\notag\\
	\|\tau_\lambda-\tau_\mu\|_{\mathcal X^*}
	&\le C_{R,\theta}|\lambda-\mu|^\theta,
	\enspace0<\theta<\frac12.
	\label{eq:SC-free-trace-holder}
\end{align}
The H\"older bound uses $|x|^\theta m_R\in L^2$ for $\theta<1/2$.

Fix
\begin{equation}
	\frac12<\sigma<\be,
	\enspace
	0<\vartheta<\min\left\{\sigma-\frac12,\be-\sigma\right\}.
	\label{eq:SC-two-margins}
\end{equation}
On synthesized sources, use the Fourier evaluation $\tau_\lambda$ given by
$(\tau_\lambda\Sigma)\mathbf s=\widehat{\Sigma\mathbf s}(\lambda)$.
The composite functional $\tau_\lambda\Sigma$ is bounded on
$\mathcal Y_\sigma$.  The band-limited estimate
\eqref{eq:NS-bandlimited-weight} and the one-dimensional embedding
$H^\sigma\hookrightarrow C^{0,\vartheta}$ give
\begin{equation}
	\|(\tau_\lambda\Sigma)-(\tau_\mu\Sigma)\|_{\mathcal Y_\sigma^*}
	\le C|\lambda-\mu|^\vartheta.
	\label{eq:SC-source-trace-holder}
\end{equation}

\begin{lemma}[H\"older continuity of the boundary operator families]
	\label{lem:SC-two-margin-boundary}
	Let $z,w$ lie in the same closed upper or lower half-strip, with real parts
	in $J$ and imaginary parts of modulus at most one.  The boundary values of
	the following maps exist in operator norm and satisfy
	\begin{align}
		\|B(z)-B(w)\|_{\mathcal B(\mathcal Y_\sigma)}
		&\le C|z-w|^\vartheta,
		\label{eq:SC-B-holder}\\
		\|A_R(z)-A_R(w)\|_{\mathcal B(\mathcal Y_\sigma,L^2)}
		&\le C_R|z-w|^\vartheta,
		\label{eq:SC-A-holder}\\
		\|C_R(z)-C_R(w)\|_{\mathcal B(L^2,\mathcal Y_\sigma)}
		&\le C_R|z-w|^\vartheta,
		\label{eq:SC-C-holder}\\
		\|F_R(z)-F_R(w)\|_{\mathcal B(L^2)}
		&\le C_R|z-w|^\vartheta.
		\label{eq:SC-F-holder}
	\end{align}
	The estimates for $A_R$, $C_R$, and $F_R$ also use
	$\vartheta<1/2$, which follows from \eqref{eq:SC-two-margins}.
\end{lemma}

\begin{proof}[Proof outline]
Put $h=|z-w|\le1$. For the energy-localized source
$u=\chi_{\rm E}(D)\Sigma\mathbf s$, the band-limited estimate and
$\vartheta<\sigma-1/2$ give
$\|u\|_1+\||x|^\vartheta u\|_1\lesssim\|\mathbf s\|_{\mathcal Y_\sigma}$.
The outgoing or incoming Volterra kernel difference is bounded by
$Ch^\vartheta|t|^\vartheta$. Local differentiation therefore gives
\begin{equation}
 \|\chi_k(G(z)-G(w))\mathbf s\|_{H^{1/2}}
 \le Ch^\vartheta\langle k\rangle^\vartheta
 \|\mathbf s\|_{\mathcal Y_\sigma}.
 \label{eq:SC-G-block-holder}
\end{equation}
Away from the pole, the multiplier difference has an $O(h)$ symbol
of order $-2$ and satisfies the same bound by the block calculus.
Multiplication by the potential and by $m_R$ permits summation because
$\sigma+\vartheta<\be$ and $\vartheta<1/2$, proving the estimates for
$B$ and $A_R$. The identities $C_R=BJ_R$, $F_R=A_RJ_R$ give the
remaining two. The bounds persist on each boundary edge and yield the
operator-norm limits.
\end{proof}

For $\lambda\in J$, use the boundary convention
$R_{u_0}^\pm(\lambda)=(L_{u_0}-\lambda\mp i0)^{-1}$ and
$R_0^\pm(\lambda)=(D-\lambda\mp i0)^{-1}$.

\begin{prop}[Fredholm inverse on the reduced source space]
	\label{prop:SC-lowered-pencil}
	For nonreal $z$ in the upper or lower boundary strip, set
	$\mathcal M_{u_0}(z):=(I-B(z))^{-1}\in\mathcal B(\mathcal Y_\sigma).$
	This inverse exists, and on $\mathcal Y_\sigma$ every $I-B^\pm(\lambda)$
	is invertible.  Moreover, the boundary values are
	\begin{equation}
		\mathcal M_{u_0}^\pm(\lambda):=(I-B^\pm(\lambda))^{-1}
		\in C^{0,\vartheta}
		\bigl(J;\mathcal B(\mathcal Y_\sigma)\bigr),
		\label{eq:SC-N-holder}
	\end{equation}
	and the inverse is uniformly bounded in a sufficiently narrow closed upper
	or lower half-strip.
\end{prop}

\begin{proof}[Proof outline]
The factorization
$B^\pm=\mathcal VG^\pm:\mathcal Y_\sigma\to X_0\to
\mathcal Y_\be\hookrightarrow\mathcal Y_\sigma$
is compact by Lemmas~\ref{lem:NS-radiation-tail} and
\ref{lem:NS-moment-gain}, also off the real axis. A kernel vector gains
membership in $\mathcal Y_\be$, where
Lemma~\ref{lem:NS-Fredholm-kernel} gives injectivity. Fredholm theory
therefore gives the inverse. The identity
\begin{equation}
 \mathcal M_{u_0}(z)-\mathcal M_{u_0}(w)
 =\mathcal M_{u_0}(z)[B(z)-B(w)]\mathcal M_{u_0}(w)
 \label{eq:SC-inverse-holder-identity}
\end{equation}
transfers H\"older continuity to the inverse; compactness of $J$ and
Neumann stability give the uniform half-strip bound.
\end{proof}

\begin{theorem}[Perturbed H\"older limiting absorption principle]
	\label{thm:SC-holder-LAP}
	For every $0<\vartheta<\frac{\be-1/2}{2}$, one may choose
	\begin{equation}
		\frac12+\vartheta<\sigma<\be-\vartheta.
		\label{eq:SC-sigma-choice}
	\end{equation}
	Then $S_{u_0,R}(z)$ extends to the two closed half-strips as a
	$C^{0,\vartheta}$ family in $\mathcal B(L^2)$.  Equivalently,
	\begin{equation}
		R_{u_0}^\pm\in
		C^{0,\vartheta}\bigl(J;\mathcal B(\mathcal X,\mathcal X^*)\bigr).
		\label{eq:SC-rigged-holder-LAP}
	\end{equation}
\end{theorem}

\begin{proof}[Proof outline]
The two margins give uniformly bounded H\"older factors in
\eqref{eq:NS-Woodbury}; the product difference estimate proves
regularity of $S_{u_0,R}$. Its boundary form descends through $W_R^*$
to $\mathcal X$: annihilation of $\ker W_R^*$ is preserved by
operator-norm convergence. This gives \eqref{eq:SC-rigged-holder-LAP}.
\end{proof}

Define the generalized eigenfunctions by
\begin{equation}
 b_\lambda=\mathcal Ve_\lambda,\enspace \mathbf s_\lambda^\pm=\mathcal M_{u_0}^\pm(\lambda)b_\lambda,\enspace c_{\rm F}m_{\mp}(\cdot,\lambda) =e_\lambda+G^\pm(\lambda)\mathbf s_\lambda^\pm. \label{eq:SC-LS-columns}
\end{equation}
The inequality
$|e^{i\lambda x}-e^{i\mu x}|
\le C|\lambda-\mu|^\vartheta\langle x\rangle^\vartheta$ and
$\sigma+\vartheta<\be$ imply
$\|b_\lambda-b_\mu\|_{\mathcal Y_\sigma}
\le C|\lambda-\mu|^\vartheta$.
Consequently, if
\begin{equation}
	\Phi_{u_0,R}^\pm(\lambda)=c_{\rm F}W_R m_{\mp}(\cdot,\lambda)
	=W_Re_\lambda+A_R^\pm(\lambda)\mathbf s_\lambda^\pm,
	\label{eq:SC-rigged-columns}
\end{equation}
then
\begin{equation}
	\Phi_{u_0,R}^\pm\in C^{0,\vartheta}(J;L^2(\R)).
	\label{eq:SC-columns-holder}
\end{equation}
The source equation is
\begin{equation}
	\mathbf s_\lambda^\pm=b_\lambda+B^\pm \mathbf s_\lambda^\pm
	=c_{\rm F}\mathcal V m_{\mp}(\cdot,\lambda),
	\label{eq:SC-LS-source-equation}
\end{equation}
so $(L_{u_0}-\lambda)m_{\mp}(\cdot,\lambda)=0$ in $H^{-1/2}_{\mathrm{loc}}$.  The
Volterra construction fixes the incident wave $e_\lambda$: the boundaries
$\lambda+i0$ and $\lambda-i0$ give $m_-$ and $m_+$, respectively.
For smooth decaying potentials these are the Jost functions of
\cite[(2.4)--(2.7)]{GGM}, with
$c_{\rm F}=(2\pi)^{-1/2}$; the source equation and approximation retain
this normalization for rough potentials.

For $f=W_R^*g\in\mathcal X$, define
\begin{align}
	\mathbf s_f^\pm(\lambda)
	&=\mathcal M_{u_0}^\pm(\lambda)C_R^\pm(\lambda)g
	=\mathcal V R_{u_0}^\pm(\lambda)f,
	\label{eq:SC-boundary-source-f}\\
	\gamma_{u_0}^\pm(\lambda)[f]
	&=\tau_\lambda f+(\tau_\lambda\Sigma) \mathbf s_f^\pm(\lambda).\notag
\end{align}
The first expression in \eqref{eq:SC-boundary-source-f} annihilates
$\ker W_R^*$, hence is independent of $g$.  Equations
\eqref{eq:SC-free-trace-holder}, \eqref{eq:SC-source-trace-holder},
\eqref{eq:SC-C-holder}, and \eqref{eq:SC-N-holder} give
$\gamma_{u_0}^\pm\in C^{0,\vartheta}(J;\mathcal X^*)$.

\begin{prop}[Reciprocity for the generalized eigenfunctions]
	\label{prop:SC-reciprocity}
	For every $\lambda\in J$ and $g\in L^2$,
	\begin{equation}
		\gamma_{u_0}^\pm(\lambda)[W_R^*g]
		=\langle\Phi_{u_0,R}^{\mp}(\lambda),g\rangle
		=\langle c_{\rm F}m_{\pm}(\cdot,\lambda),W_R^*g\rangle.
		\label{eq:SC-reciprocity}
	\end{equation}
	Thus the distribution representative of $\gamma_{u_0}^+$ is $c_{\rm F}m_+(\cdot,\lambda)$,
	and conversely.
\end{prop}

\begin{proof}[Proof outline]
For $u_0\in C_c^\infty$ and $T=T_{u_0}$, the source
$Te_\lambda$ belongs to $\mathcal X$. Symmetry and the weighted
boundary adjoint relation identify
\[
 \gamma_{u_0}^-(\lambda)[f]
 =\langle(I+R_{u_0}^+(\lambda)T)e_\lambda,f\rangle
 =\langle c_{\rm F}m_-(\cdot,\lambda),f\rangle,
\]
where the last equality is the Lippmann--Schwinger identity.
Reversing signs proves the other identity.
For rough $u_0$, take the weighted smooth approximation of
Proposition~\ref{prop:global-orbit-control}. Weighted localization and
the boundary estimates give uniform convergence on $J$ of the sources
$b_{u_{0,n},\lambda}$, the pencils $B_{u_{0,n}}^\pm$, and the maps
$C_{u_{0,n},R}^\pm$ in their respective source/operator norms.
Equation~\eqref{eq:SC-inverse-holder-identity} gives convergence of the
inverses. Passing to the limit in the smooth identity proves reciprocity.
\end{proof}

For a Banach-space vector $u$ and functional $\tau$, write
$(u\boxtimes\tau)\mathbf s=u\,\tau(\mathbf s)$.
The free Plemelj formula and \eqref{eq:SC-inverse-holder-identity} give
\begin{align}
	G^+-G^-&=2\pi i\,e_\lambda\boxtimes(\tau_\lambda\Sigma),\enspace
	B^+-B^-=2\pi i\,b_\lambda\boxtimes(\tau_\lambda\Sigma),\notag\\
	\mathcal M_{u_0}^+-\mathcal M_{u_0}^-&=
	2\pi i\,(\mathcal M_{u_0}^+b_\lambda)\boxtimes((\tau_\lambda\Sigma) \mathcal M_{u_0}^-).
	\label{eq:SC-N-jump}
\end{align}
\begin{theorem}[Rank-one density formula for the spectral measure]
	\label{thm:SC-pointwise-jump}
	For every $\lambda\in J$,
	\begin{equation}
 \frac{S_{u_0,R}^+(\lambda)-S_{u_0,R}^-(\lambda)}{2\pi i} =|\Phi_{u_0,R}^+(\lambda)\rangle \langle\Phi_{u_0,R}^+(\lambda)| =|\Phi_{u_0,R}^-(\lambda)\rangle \langle\Phi_{u_0,R}^-(\lambda)|. \label{eq:SC-pointwise-sandwiched-jump}
\end{equation}
	This family is $C^{0,\vartheta}$ in trace norm.  Between the fixed weighted spaces,
	\begin{equation}
		\frac{R_{u_0}^+(\lambda)-R_{u_0}^-(\lambda)}{2\pi i}
		=\gamma_{u_0}^+(\lambda)^*\gamma_{u_0}^+(\lambda)
		=\gamma_{u_0}^-(\lambda)^*\gamma_{u_0}^-(\lambda).
		\label{eq:SC-pointwise-rigged-jump}
	\end{equation}
\end{theorem}

\begin{proof}[Proof outline]
The free Plemelj identities are
\begin{align}
 F_R^+-F_R^-&=2\pi i|W_Re_\lambda\rangle\langle W_Re_\lambda|,
 \label{eq:SC-F-jump}\\
 A_R^+-A_R^-&=2\pi i(W_Re_\lambda)\boxtimes(\tau_\lambda\Sigma),
 \label{eq:SC-A-jump}\\
 C_R^+-C_R^-&=2\pi i b_\lambda\boxtimes
 \langle W_Re_\lambda,\,\cdot\,\rangle.
 \label{eq:SC-C-jump}
\end{align}
Insert these and \eqref{eq:SC-N-jump} into the boundary difference of
\eqref{eq:NS-Woodbury}. The vector factor is
$W_Re_\lambda+A_R^+\mathcal M_{u_0}^+b_\lambda=\Phi_{u_0,R}^+$;
reciprocity identifies the functional factor as its adjoint.
Reversing signs gives the second factorization.
The rank-one difference estimate and \eqref{eq:SC-columns-holder}
give trace-norm H\"older continuity. Evaluation on preimages under
$W_R^*$ gives the rigged identity.
\end{proof}

\begin{prop}[Jost connection coefficient and scattering phase]
	\label{prop:SC-phase}
	For every $\lambda\in J$,
	\begin{equation}
		\Phi_{u_0,R}^\pm(\lambda)\neq0,\enspace
		\gamma_{u_0}^\pm(\lambda)\neq0.
		\label{eq:SC-nondegenerate}
	\end{equation}
	There is a unique $\ell_{u_0}(\lambda)\in\mathbb T$ such that
	\begin{equation}
		m_+(\cdot,\lambda)=
		\ell_{u_0}(\lambda)m_-(\cdot,\lambda),
		\enspace
		\gamma_{u_0}^-(\lambda)=
		\ell_{u_0}(\lambda)\gamma_{u_0}^+(\lambda).
		\label{eq:SC-relative-phase}
	\end{equation}
	Moreover, $\ell_{u_0}\in C^{0,\vartheta}(J;\mathbb T)$, locally on
	compact positive-energy intervals.
\end{prop}

\begin{proof}[Proof outline]
Positivity of $m_R$, the source equation, and
\eqref{eq:SC-LS-columns} exclude a zero weighted column.
Reciprocity then excludes zero functionals.
The equal nonzero rank-one operators in
\eqref{eq:SC-pointwise-sandwiched-jump} determine a unique unimodular
factor. Injectivity of multiplication by $m_R$ and reciprocity give
\eqref{eq:SC-relative-phase}. The quotient
$\ell_{u_0}=\langle\Phi_{u_0,R}^+,\Phi_{u_0,R}^-\rangle/
\|\Phi_{u_0,R}^+\|_2^2$ has denominator bounded away from zero on
compact intervals, so the H\"older column bounds give its regularity.
\end{proof}

\subsection{Distorted Fourier analysis and scattering data}

The local constructions agree on overlaps: they use the same free boundary
map $G^\pm(\lambda)$ and hence the same $B^\pm(\lambda)$, while
Proposition~\ref{prop:SC-lowered-pencil} gives uniqueness of the source
solution in $\mathcal Y_\sigma$.  Thus the distorted Fourier functionals
and generalized eigenfunctions patch to locally H\"older functions on
$(0,\infty)$.  Write
$P_{\mathrm c}=E_{u_0}((0,\infty))$ and
$P_{\mathrm d}=E_{u_0}((-\infty,0))$; since $E_{u_0}(\{0\})=0$, their sum is the
identity.

\begin{prop}[Spectral density and completeness of the positive-energy transform]
	For every Borel $\Delta\subset J$,
	\begin{align}
		W_RE_{u_0}(\Delta)W_R^*
		&=\int_\Delta|\Phi_{u_0,R}^\pm(\lambda)\rangle
		\langle\Phi_{u_0,R}^\pm(\lambda)|\,\rmd\lambda,
		\label{eq:SC-sandwiched-measure}\\
		\langle f,E_{u_0}(\Delta)h\rangle
		&=\int_\Delta\overline{\gamma_{u_0}^\pm(\lambda)[f]}
		\gamma_{u_0}^\pm(\lambda)[h]\,\rmd\lambda,
		\enspace f,h\in\mathcal X.
		\label{eq:SC-rigged-measure}
	\end{align}
	The first integral is locally a trace-norm Bochner integral.  Globally,
	\begin{equation}
		\int_0^\infty\overline{\gamma_{u_0}^\pm(\lambda)[f]}
		\gamma_{u_0}^\pm(\lambda)[h]\,\rmd\lambda
		=\langle f,P_{\mathrm c}h\rangle.
		\label{eq:SC-positive-exhaustion}
	\end{equation}
	Consequently,
	$\mathcal F_{u_0}^\pm f=\gamma_{u_0}^\pm(\,\cdot\,)[f]$ extends uniquely from
	$\mathcal X$ to $L_+^2$ and satisfies
	\begin{equation}
		\mathcal F_{u_0}^{\pm *}\mathcal F_{u_0}^\pm=P_{\mathrm c}.
		\label{eq:SC-partial-isometry}
	\end{equation}
\end{prop}

\begin{proof}[Proof outline]
In Stone's formula on $[a,b]\subset J$, uniform boundary convergence
and Theorem~\ref{thm:SC-pointwise-jump} identify the density.
The endpoints have no atoms; uniqueness of matrix-element measures
extends the identity to Borel sets. H\"older columns give the
trace-norm integral, and reciprocity gives the formula on $\mathcal X$.
Then exhaust $(0,\infty)$ by $(1/n,n)$. Monotone convergence,
polarization, and density of $\mathcal X$ give
\eqref{eq:SC-positive-exhaustion} and \eqref{eq:SC-partial-isometry}.
\end{proof}

\begin{lemma}[One-sided Abel identity]
	\label{lem:SC-Abel}
	For $h,f\in L_+^2$,
	\begin{equation}
		\langle W_\pm h,f\rangle
		=\lim_{\epsilon\downarrow0}\frac{\epsilon}{\pi}
		\int_\R
		\left\langle
		R_0(\lambda\pm i\epsilon)h,
		(L_{u_0}-\lambda\mp i\epsilon)^{-1}f
		\right\rangle\,\rmd\lambda.
		\label{eq:SC-Abel-identity}
	\end{equation}
	If $\widehat h\in C_c^\infty(J)$ and $f\in\mathcal X$, then
	\begin{equation}
		\langle W_\pm h,f\rangle
		=\int_J\overline{\widehat h(\lambda)}
		\gamma_{u_0}^\pm(\lambda)[f]\,\rmd\lambda.
		\label{eq:SC-Abel-stationary}
	\end{equation}
\end{lemma}

\begin{proof}[Proof outline]
Write $R_{u_0}(z)=(L_{u_0}-z)^{-1}$ and set
\[
\mathcal A_{\epsilon,+}(h,f)=2\epsilon\int_0^\infty e^{-2\epsilon t}
\langle e^{itL_{u_0}}e^{-itD}h,f\rangle\,\rmd t.
\]
Half-line Plancherel gives
\begin{equation}
 \mathcal A_{\epsilon,+}(h,f)
 =\frac{\epsilon}{\pi}\int_\R
 \langle R_0(\lambda+i\epsilon)h,
 R_{u_0}(\lambda+i\epsilon)f\rangle\,\rmd\lambda.
 \label{eq:SC-Abel-spectral-Fubini}
\end{equation}
The spectral square-integral identity
$\int_\R\|(B-\lambda-i\epsilon)^{-1}v\|_2^2\,\rmd\lambda
=(\pi/\epsilon)\|v\|_2^2$
justifies the identity for all $h,f$. Abel convergence to the wave
operator proves \eqref{eq:SC-Abel-identity}; reverse time for the other sign.

For the stationary formula choose
$\operatorname{supp}\widehat h\subset J\Subset\operatorname{int}J'
\Subset(0,\infty)$ and $f=W_R^*g$. Put
$P_\epsilon(r)=\epsilon/[\pi(r^2+\epsilon^2)]$ and
$\Gamma_\epsilon^\pm(\mu,\lambda)[f]=\tau_\mu f+
(\tau_\mu\Sigma)\mathcal V R_{u_0}(\lambda\pm i\epsilon)f$.
The resolvent source factorization gives
\begin{align}
 &\frac{\epsilon}{\pi}\int_{J'}
 \langle R_0(\lambda\pm i\epsilon)h,
 R_{u_0}(\lambda\pm i\epsilon)f\rangle\,\rmd\lambda\notag\\
 &\quad=\int_J\overline{\widehat h(\mu)}
 \int_{J'}P_\epsilon(\mu-\lambda)
 \Gamma_\epsilon^\pm(\mu,\lambda)[f]\,\rmd\lambda\,\rmd\mu.
 \label{eq:SC-Abel-on-window-exact}
\end{align}
The H\"older estimates bound
$|\Gamma_\epsilon^\pm(\mu,\lambda)[f]-\gamma_{u_0}^\pm(\mu)[f]|$
by $C\|g\|_2(|\lambda-\mu|+\epsilon)^\vartheta$.
The Poisson approximate identity gives the limit on $J$. Outside $J'$,
spectral separation and the square-integral identity give an
$O(\epsilon^{1/2})$ contribution. Thus $\epsilon\downarrow0$ proves
\eqref{eq:SC-Abel-stationary}.
\end{proof}

\begin{theorem}[Stationary representation of the wave operators]
	\label{thm:SC-onto-transform}
	Let $\mathcal F_0:L_+^2\to L^2(0,\infty)$ be the restricted unitary Fourier
	transform.  Then
	\begin{equation}
		\mathcal F_{u_0}^\pm=\mathcal F_0W_\pm^*P_{\mathrm c}
		=\mathcal F_0W_\pm^*.
		\label{eq:SC-stationary-gauge}
	\end{equation}
	The restriction $\mathcal F_{u_0}^\pm:P_{\mathrm c}L_+^2
	\to L^2(0,\infty)$ is unitary and onto.
\end{theorem}

\begin{proof}[Proof outline]
Equation~\eqref{eq:SC-Abel-stationary}, tested on
$f\in\mathcal X$ and $h\in\mathcal F_0^*C_c^\infty(0,\infty)$,
identifies $\mathcal F_{u_0}^\pm$ with $\mathcal F_0W_\pm^*$ by
density. Completeness in Theorem~\ref{thm:SC-complete-channel} gives
$W_\pm^*W_\pm=I$ and $W_\pm W_\pm^*=P_{\mathrm c}$.
Together with \eqref{eq:SC-partial-isometry}, these identities prove
unitarity onto the whole positive-energy data space.
\end{proof}

Use the fixed shift $A_0=L_{u_0}+\ka$ and Sobolev scales from
Proposition~\ref{prop:global-orbit-control}.

For $|r|\le1/2$, define
\begin{align}
	\mathcal F_{u_0,r}^\pm f
	&=(\lambda+\ka)^{-r}\mathcal F_{u_0}^\pm P_{\mathrm c}A_0^rf,
	\label{eq:SC-Hr-transform}\\
	d_j^{(r)}(f)
	&=(\ka-a_j)^{-r}\langle f_j,A_0^rf\rangle.
	\label{eq:SC-Hr-discrete-coordinate}
\end{align}
When $r<0$, these expressions are interpreted through the Hilbert-scale
completion and its dual pairing.
On the continuous spectral core,
$\mathcal F_{u_0}^\pm A_0^r=(\lambda+\ka)^r\mathcal F_{u_0}^\pm$ and
$\langle f_j,A_0^rf\rangle=(\ka-a_j)^r\langle f_j,f\rangle$.  Hence
$\mathcal F_{u_0,r}^\pm f=\mathcal F_{u_0}^\pm f$ and
$d_j^{(r)}(f)=\langle f_j,f\rangle$
whenever the right sides are classical.  Thus
\eqref{eq:SC-Hr-transform}--\eqref{eq:SC-Hr-discrete-coordinate} are the
closed scale extensions of the ordinary spectral coordinates.

\begin{theorem}[Sobolev-scale Plancherel theorem and inverse transform]
	\label{thm:SC-Hr-Plancherel}
	For every $|r|\le1/2$,
	\begin{equation}
		\mathcal H_{A_0}^r=H_+^r,\enspace
		\|f\|_{H^r}\lesssim\|A_0^rf\|_2\lesssim\|f\|_{H^r}.
		\label{eq:SC-scale-identification}
	\end{equation}
	Moreover,
	\begin{equation}
 \|A_0^rf\|_2^2 =\sum_j(\ka-a_j)^{2r}|d_j^{(r)}(f)|^2 \enspace+\int_0^\infty(\lambda+\ka)^{2r} |\mathcal F_{u_0,r}^\pm f(\lambda)|^2\,\rmd\lambda. \label{eq:SC-Hr-Plancherel}
\end{equation}
	The complete transform is unitary from $\mathcal H_{A_0}^r$ onto
	\begin{equation}
		\ell^2\bigl((\ka-a_j)^{2r}\bigr)
		\oplus L^2\bigl((0,\infty),(\lambda+\ka)^{2r}\rmd\lambda\bigr).
		\label{eq:SC-Hr-data-space}
	\end{equation}
	For data $(d,g)$ in \eqref{eq:SC-Hr-data-space}, the inverse is
	\begin{equation}
		A_0^{-r}\left(
		\sum_j(\ka-a_j)^rd_jf_j
		+\mathcal F_{u_0}^{\pm *}[(\lambda+\ka)^rg(\lambda)]
		\right).
		\label{eq:SC-Hr-inverse}
	\end{equation}
	The sum is order-independent, and both terms converge strongly in $H^r$.
\end{theorem}

\begin{proof}[Proof outline]
Proposition~\ref{prop:global-orbit-control} identifies the scales.
At $r=0$, the negative eigenlines and continuous unitary transform
give Parseval. Apply it to $A_0^rf$ to obtain
\eqref{eq:SC-Hr-Plancherel}. Applying the $L^2$ inverse to the weighted
data, then $A_0^{-r}$, gives \eqref{eq:SC-Hr-inverse} and surjectivity.
Spectral truncations converge in the weighted data norm, hence in
$H^r$; orthogonality makes the discrete sum order-independent.
\end{proof}

The original potential satisfies $(u_0)_+=\Pi u_0\in H_+^{s}$ as a
distribution.  Define its transform by the closed scale extension.  If
$v_n\in \Pi C_c^\infty$ and $v_n\to v$ in $H^{s}$, set
$\mathcal F_{u_0,s}^\pm v
=\lim_{n\to\infty}\mathcal F_{u_0}^\pm v_n$ in
$L^2\bigl((0,\infty),(\lambda+\ka)^{2s}\rmd\lambda\bigr)$.
Theorem~\ref{thm:SC-Hr-Plancherel} makes this limit Cauchy and independent
of the approximation.  Equivalently, for $\zeta\in C_c^\infty(0,\infty)$,
$\int_0^\infty\overline{\zeta(\lambda)}
(\mathcal F_{u_0,s}^\pm v)(\lambda)\,\rmd\lambda
=\langle(\mathcal F_{u_0}^\pm)^*\zeta,v\rangle_{H^{-s},H^{s}}$.

Set $g_{u_0}^\pm=\mathcal F_{u_0,s}^\pm((u_0)_+)$.
The Jost connection coefficient in Proposition~\ref{prop:SC-phase} extends by density to
\begin{equation}
	g_{u_0}^-=\ell_{u_0}\,g_{u_0}^+
	\enspace\text{almost everywhere on }(0,\infty).
	\label{eq:SC-gq-phase}
\end{equation}
Spectral calculus gives
$\langle f_j,A_0^s(u_0)_+\rangle
=(\ka-a_j)^s\langle f_j,(u_0)_+\rangle_{H^{-s},H^s}$, so
\begin{equation}
	\mathcal F_{u_0,s}^\pm (u_0)_+=g_{u_0}^\pm,\enspace
	d_j^{(s)}((u_0)_+)
	=\langle f_j,(u_0)_+\rangle_{H^{-s},H^{s}}=:d_j.
	\label{eq:SC-input-coordinate-identification}
\end{equation}

\begin{definition}[Radiation profiles]
	\label{def:SC-radiation}
	For $\lambda>0$, define
	\begin{equation}
		\widehat u_\infty^+(\lambda)=g_{u_0}^-(\lambda),
		\enspace
		\widehat u_\infty^-(\lambda)=g_{u_0}^+(\lambda),
		\label{eq:SC-radiation-positive-frequency}
	\end{equation}
	and impose real symmetry
	$\widehat u_\infty^\pm(-\lambda)
	=\overline{\widehat u_\infty^\pm(\lambda)}$.
\end{definition}

Since $\|u_\infty^\pm\|_{H^{s}}^2
=2\int_0^\infty\langle\lambda\rangle^{2s}
|g_{u_0}^\mp(\lambda)|^2\,\rmd\lambda$,
weighted Plancherel gives $u_\infty^\pm\in H^{s}(\R;\R)$.
The scattering phase relations are
\[
	\widehat u_\infty^+(\lambda)
	=\ell_{u_0}(\lambda)\,\widehat u_\infty^-(\lambda),
	\enspace\lambda>0,\enspace 
	\mathcal F_0(W_+^*W_-)\mathcal F_0^*
	=M_{\overline{\ell_{u_0}}}.
\]
These follow from \eqref{eq:SC-gq-phase} and
$\mathcal F_{u_0}^\pm=\mathcal F_0W_\pm^*$, respectively.  Thus the physical
Lax scattering matrix is the conjugate of the Jost connection coefficient.
With the normalization of \cite[(2.5), (2.7)]{GGM}, we have $\mathcal F_{u_0}^\pm f=c_{\rm F}\widetilde f^{\,\pm}$
for smooth data, with the rough transform defined by the extensions above.
The opposite sign in \eqref{eq:SC-radiation-positive-frequency} agrees with
\eqref{eq:prep-classical-radiation-data}.  In particular, $\ell_{u_0}$ is
the radiation multiplier of \eqref{eq:prep-Jost-phase}; the
transmission coefficient for $m_-$ is its inverse.

\section{Uniform radiation limits for the unitary flow}
\label{sec:uniform-sewing}

All cutoff functions in this section are real-valued.
Throughout, \eqref{eq:NS-assumptions} holds with
$1/2<\be<1$.  We retain the Fourier and pairing conventions of
Section~\ref{subsec:notation-guide}, put $\varrho_t=\sqrt{2t}$, and abbreviate
\( U_{u_0}^{\mathrm{cyc}}(t)= U_{u_0}^{\mathrm{cyc}}(t,0)\).  We use the fixed real
Schwartz approximation from Proposition~\ref{prop:global-orbit-control};
its weighted convergence is \eqref{eq:weighted-smooth-approximation}.
Let \(U_n(t,0)\) be the auxiliary unitary flow introduced in Section~\ref{sec:rough-Gerard}.  For
\(\chi_0\in C_c^\infty(0,\infty)\), put
$h_{\chi_0,n}=\chi_0(L_{u_{0,n}})(u_{0,n})_+$ and
$h_{\chi_0}=\chi_0(L_{u_0})(u_0)_+$.
In particular, the \(\be\)-weighted block norms in
Section~\ref{sec:no-singular} are bounded uniformly in \(n\).

\subsection{Stationary inputs and spectral averaging}

\begin{lemma}[Convergence of localized scattering data]
\label{lem:CR-GGM-stationary-input}
Let $I\Subset(0,\infty)$ be compact, and let
$m_{\pm,n}(\cdot,a)$ be the Jost functions in
\eqref{eq:SC-LS-columns} for $u_{0,n}$.
Then
\begin{equation}
 g_{u_{0,n}}^{\pm}(a)
 =\langle c_{\rm F}m_{\pm,n}(\cdot,a),u_{0,n}\rangle,
 \enspace
 \sup_{a\in I}|g_{u_{0,n}}^{\pm}(a)-g_{u_0}^{\pm}(a)|\longrightarrow0.
 \label{eq:CR-GGM-stationary-input-convergence}
\end{equation}
Here the pairing is the absolutely convergent sum of local
$H^{-s}$--$H^s$ pairings, conjugate-linear in the first entry.
\end{lemma}

\begin{proof}[Proof outline]
Cauchy--Schwarz on unit blocks gives
$|\langle\phi,v\rangle|\le C_\be\|\phi\|_{X_0}
\|\langle x\rangle^\be v\|_{H^s}$ because $\be>1/2$.
Weighted approximation, \eqref{eq:SC-inverse-holder-identity}, and
Lemma~\ref{lem:NS-radiation-tail}(i) give uniform Jost convergence in
$X_0$.  Reciprocity identifies the pairing with the transform on smooth
data; weighted density and the closed scale extension pass this
identification to $g_{u_0}^{\pm}$ and give the claimed convergence.
\end{proof}

\begin{lemma}[Spectral averaging for the moving resolvent]
	\label{lem:CR-GGM-ray-averaging}
	Let \(J\Subset(0,\infty)\) be an interval, \(K\Subset\C_+\) be compact,
	\(z=x_0+iy\in K\), and
	\(0\le r\le1+s\).  If
	\(\mathcal R_{n,t,\lambda}(z)
	=(X^*-2t(L_{u_{0,n}}-\lambda)-z)^{-1}\) for \(\lambda\in J\),
	then
	\begin{equation}
		t\int_J\|\mathcal R_{n,t,\lambda}(z)h\|_{H^r}^2\,\rmd\lambda
		\le \frac{C_r\pi}{2y}\|h\|_{H^r}^2 .
		\label{eq:CR-GGM-ray-bound}
	\end{equation}
	The constant is independent of \(n\) and \(t>0\).  Moreover, if
	\(t_k\to+\infty\), \(n_k\) is arbitrary, \(z_k\in K\), and
	\(h_k\) is bounded in \(H_+^{1+s}\), then every weak cluster point of
	$F_k(\lambda)=\varrho_{t_k}
	\mathcal R_{n_k,t_k,\lambda}(z_k)h_k$ in \(L^2(J;H_+^{1+s})\) is zero.  Consequently, for every relatively
	compact \(\mathcal K\subset L^2(J;H_+^{-s-1})\),
	\[
\sup_{G\in\mathcal K}
		\left|\int_J\langle F_k(\lambda),G(\lambda)\rangle\,\rmd\lambda\right|
		\longrightarrow0.
\]
	The same conclusion holds with a multiplier
	\(\omega_k\in L^\infty(J)\), \(\|\omega_k\|_\infty\le1\), inserted in
	the integral and with \(G(\lambda)\equiv v_k\), where
	\(\{v_k\}\) is relatively compact in \(H_+^{-s-1}\).
\end{lemma}

\begin{proof}[Proof outline]
Intertwining by $U_n(t,0)$ and \eqref{eq:uniform-fractional-scales}
reduce the bound to
\[
 \int_\R\|(X^*-(a+iy))^{-1}f\|_{H^r}^2\,\rmd a
 \le \frac\pi y\|f\|_{H^r}^2.
\]
Fourier representation and Plancherel prove this estimate; the change
$a=x_0-2t\lambda$ gives \eqref{eq:CR-GGM-ray-bound} uniformly in $n$.

After selecting a convergent subsequence of parameters, graph-core
testing gives $(L_{u_{0,*}}-\lambda)F(\lambda)=0$ for a weak limit.
Distorted Plancherel eliminates the discrete coordinates and confines
the continuous coordinate to the null set $\eta=\lambda$, so $F=0$.
Finite nets handle compact testers; the same graph-core argument for
$\overline{\omega_k}F_k$ gives the multiplier assertion.
\end{proof}

In the stationary representation of Gassot, G\'erard, and Miller
\cite[(5.10), (8.4)--(8.6)]{GGM}, we subtract the nondecaying
outgoing tail using a correction supported away from zero frequency.  Fix
\[
J\Subset I\Subset(0,\infty),\enspace a_0=\inf I,\enspace
	\frac12<\sigma_0<\be,\enspace
	0<\vartheta<\min\{\sigma_0-\tfrac12,\be-\sigma_0\}.
\]
Choose \(\chi\in C_c^\infty(I)\) equal to one on a neighborhood of
\(J\).  In every cyclic application below \(\chi_0\in C_c^\infty(J)\),
so in particular $\chi\equiv1$ on $\operatorname {supp}\chi_0$.
Use \((\tau_a\Sigma)\mathbf s=\widehat{\Sigma\mathbf s}(a)\) on
\(\mathcal Y_{\sigma_0}\), with \(\Sigma\), \(G^+\), and \(\mathcal V\) as in
\eqref{eq:NS-source-space}--\eqref{eq:NS-factorization-maps}; a subscript
\(n\) means that \(u_0\) is replaced by \(u_{0,n}\).  Set
\[
G^+(a)=G(a+i0)=(D-a-i0)^{-1}\Pi\Sigma,\enspace
	\mathcal M_{u_{0,n}}^+(a)=(I-\mathcal V_nG^+(a))^{-1}.
\]
Choose a real even \(\zeta\in\mathcal S\) with integral one and
\(\operatorname {supp}\widehat\zeta\subset(-a_0/4,a_0/4)\), and define
$b_+(x)=\int_{-\infty}^x\zeta(y)\,\rmd y$.
Then \(e^{iax}b_+\) and
\(e^{iax}b_+'\) have distributional Fourier support in
\([a-a_0/4,a+a_0/4]\subset(0,\infty)\), and
$\int_\R b_+'(x)\,\rmd x=1$.

\begin{lemma}[Cancellation at the spectral parameter and compactness of the divided difference]
	\label{lem:CR-GGM-centering}
	Let \(e_a=c_{\rm F}e^{iax}\) and define the incident and constant-source
	curves by
	\begin{equation}
		\mathbf s_{e,n,a}=c_{\rm F}^{-1}\mathcal M_{u_{0,n}}^+(a)\mathcal V_ne_a,
		\enspace \mathbf s_{1,n,a}=\mathcal M_{u_{0,n}}^+(a)\mathcal V_n\mathbf1.
		\label{eq:CR-GGM-source-curves}
	\end{equation}
	The factor \(c_{\rm F}^{-1}\) makes the incident coefficient equal to one.
	Let \(m_{-,n}=e^{iax}+G^+(a)\mathbf s_{e,n,a}\) and
	\(m_{1,n}=1+G^+(a)\mathbf s_{1,n,a}\).  The first is the
	Jost function $m_-$ of \cite[(2.4)--(2.7)]{GGM}
	for $u_{0,n}$.  Use its transmission coefficient
	$\Gamma_n=\ell_{u_{0,n}}^{-1}$ and Wu's connection coefficient
	$\varkappa_n$ \cite[Lemmas~4.2 and~4.4]{WuDirect}, and put
	\begin{align*}
		S_{n,a}
		&=\chi'(a)\mathbf s_{e,n,a}
		+\chi(a)\varkappa_n(a)\mathbf s_{1,n,a},\\
		H_n(a)&=\chi(a)(\Gamma_n(a)-1),\enspace
		C_n(a)=\partial_aH_n(a).
\end{align*}
	Then, after decreasing \(\vartheta\) if necessary, the following families
	are relatively compact:
	\begin{equation}
		\{S_{n,\cdot}\}_n
		\Subset C^{0,\vartheta}(I;\mathcal Y_{\sigma_0}),
		\enspace
		\{H_n\}_n\Subset C_c^{1,\vartheta}(I),
		\enspace
		\{C_n\}_n\Subset C_c^{0,\vartheta}(I).
		\label{eq:CR-GGM-coefficient-compactness}
	\end{equation}
	Let \(\mathcal Df=(\theta_kf)_k\) be the unit-block analysis map and define
	\[
d_{n,a}=\mathcal D\!\left(iC_n(a)e^{iax}b_+'(x)\right),
		\enspace
		\Sigma d_{n,a}=iC_n(a)e^{iax}b_+'(x),
		\enspace
		\widetilde S_{n,a}=S_{n,a}+d_{n,a},
\]
	then
	\begin{equation}
		(\tau_a\Sigma)(\widetilde S_{n,a})=0
		\label{eq:CR-GGM-zero-normal-trace}
	\end{equation}
	and the regularized outgoing field
	\begin{equation}
		\Phi_{n,a}
		=\Pi\{G^+(a)S_{n,a}
		-C_n(a)e^{iax}b_+(x)\}
		=G^+(a)\widetilde S_{n,a}
		\label{eq:CR-GGM-centered-field}
	\end{equation}
	has Fourier transform
	\begin{equation}
		\widehat\Phi_{n,a}(\xi)
		=\mathbf1_{\{\xi>0\}}
		\frac{\gamma_n(\xi,a)}{\xi-a},\enspace
		\gamma_n(\xi,a)
		=\widehat{\Sigma\widetilde S_{n,a}}(\xi),\enspace
		\gamma_n(a,a)=0.
		\label{eq:CR-GGM-normal-quotient}
	\end{equation}
	For every
	\(0\le\delta<\sigma_0-1/2\),
	\begin{equation}
		\{\Phi_{n,\cdot}:n\ge1\}\Subset C(I;X_\delta).
		\label{eq:CR-GGM-centered-Xdelta-compactness}
	\end{equation}
	The kernels
	$R_n(\xi,a)=\mathbf1_{\{\xi>0\}}\frac{\gamma_n(\xi,a)}{\xi-a}$
	are relatively compact in
	\[
	\mathfrak Z_s=L^2_\xi\!\left(\langle\xi\rangle^{-2(s+1)}\rmd\xi;
	L_a^1(I)\right).
	\]
	Consequently, for compact \(K\subset\C_+\),
	\begin{equation}
		\lim_{T\to\infty}
		\sup_{\substack{n,\ t\ge T,\ z\in K,\ \lambda\in J}}
		\left\|\langle\xi\rangle^{-s-1}
		\int_Ie^{-it(a-\lambda)^2-i\bar z(a-\lambda)}
		R_n(\xi,a)\,\rmd a\right\|_{L^2_\xi}=0,
		\label{eq:CR-GGM-mixed-Fresnel}
	\end{equation}
\end{lemma}

\begin{proof}[Proof outline]
Lemma~\ref{lem:SC-two-margin-boundary}, the inverse identity, and weighted
approximation make the two source curves compact.  In the present
normalization Wu's coefficients satisfy
\[
 \beta_n=ic_{\rm F}^{-1}(\tau_a\Sigma)\mathbf s_{1,n,a},\qquad
 \Gamma_n=1+ic_{\rm F}^{-1}(\tau_a\Sigma)\mathbf s_{e,n,a},\qquad
 \varkappa_n=-\frac{c_{\rm F}}a\overline{g_{u_{0,n}}^-}.
\]
The differential identities in \cite[Lemmas~4.1, 4.2 and~4.4]{WuDirect},
including $\partial_a\Gamma_n=\varkappa_n\beta_n$, express the differentiated
outgoing field as a constant source plus $G^+(a)S_{n,a}$.
Together with Arzel\`a--Ascoli they give
\eqref{eq:CR-GGM-coefficient-compactness}, uniformly in $n$.

The cancellation is encoded by
\begin{equation}
 (\tau_a\Sigma)S_{n,a}
 =-ic_{\rm F}\{\chi'(a)(\Gamma_n-1)+\chi(a)\partial_a\Gamma_n\}
 =-ic_{\rm F}C_n(a).
 \label{eq:CR-GGM-S-trace}
\end{equation}
Since $(D-a)(C_ne^{iax}b_+)=-iC_ne^{iax}b_+'$ and
$\int b_+'=1$, the correction has trace $ic_{\rm F}C_n$ and
$G^+(a)d_{n,a}=-\Pi(C_ne^{iax}b_+)$.  This proves the zero trace,
centered-field identity, and divided-difference representation.

The correction preserves compact source families.  The zero-trace part
of Lemma~\ref{lem:NS-radiation-tail}, with the strict loss
$\delta<\sigma_0-1/2$, gives compactness in $C(I;X_\delta)$.
To obtain compactness in $\mathfrak Z_s$, fix
$0<\delta<\sigma_0-1/2$ and use
$|\gamma_n(\xi,a)/(\xi-a)|\le C|\xi-a|^{\delta-1}$ near the diagonal,
separation of the denominator at the Hardy edge, and the tail bound
\[
 \int_{|\xi|>R}\langle\xi\rangle^{-2(s+1)}
 \left(\int_I\frac{|\gamma_n(\xi,a)|}{|\xi-a|}\,\rmd a\right)^2\rmd\xi
 \le CR^{-2s-3}.
\]
Finite-sum approximation in this mixed norm and van der Corput yield
\eqref{eq:CR-GGM-mixed-Fresnel}, uniformly in all stated parameters.
\end{proof}

\subsection{Radiation estimates and asymptotics}

For the smooth potential $v=u_{0,n}$, use the kernels $r_{t,z,\lambda}$
and $p_{t,z,\lambda}$ of Gassot, G\'erard, and Miller
\cite[(5.4) and the display immediately before (8.1)]{GGM}, with the
cutoff $\chi$ fixed above.  The product $\chi(\eta)r$ is zero for
$\eta=\lambda+\alpha/\varrho_t\le0$.  Their $b$ and $b_1=b-b_0$
are given by \cite[(8.5)--(8.6)]{GGM}; in our normalization,
$b_0(x,\eta)=-c_{\rm F}e^{-i\eta x}\frac{\chi(\eta)}{\eta}
\overline{g_{u_{0,n}}^-(\eta)}$.
For $t>0$, $z\in\C_+$, and $\lambda\in J$, define
\begin{equation}
 \phi_{t,z,\lambda}(x)=\frac{x^2}{4t}
       +\left(\lambda-\frac{\bar z}{2t}\right)x,
 \enspace
 \mathcal W H(x)=e^{i\phi_{t,z,\lambda}(x)}
       \int_{-\infty}^xe^{-i\phi_{t,z,\lambda}(y)}H(y)\,\rmd y.
 \label{eq:CR-GGM-Volterra-phase}
\end{equation}
With $\eta=\lambda+\alpha/\varrho_t$, their integration by parts
\cite[(8.4)]{GGM} becomes
\begin{equation}
 p=-\frac1{\varrho_t}\int_I
 e^{-it(\eta-\lambda)^2-i\bar z(\eta-\lambda)}
 \mathcal W[e^{i\eta\cdot}b(\cdot,\eta)]\,\rmd\eta.
 \label{eq:CR-GGM-p-by-parts}
\end{equation}
Its $b_0$ contribution after Hardy projection is
\begin{equation}
 \mathfrak t^0=\varrho_t c_{\rm F}\int_I
 e^{-it(\eta-\lambda)^2-i\bar z(\eta-\lambda)}
 \frac{\chi(\eta)}{\eta}\overline{g_{u_{0,n}}^-(\eta)}\,\rmd\eta,
 \enspace
 p^{\rm th}=\frac{\mathfrak t^0}{2t}\Pi\mathcal W\mathbf1.
 \label{eq:CR-GGM-threshold-scalar}
\end{equation}
Thus $\mathfrak t^0$ is their $J^0$ in \cite[(8.7)]{GGM}, and
$p^{\rm th}=\Pi(e^{it\lambda^2}p^0)$ with $p^0$ from
\cite[(8.12)]{GGM}.  These kernels depend on the potential and cutoff;
the source $h_n^\circ$ enters only through $f$ and $q$ below.

\begin{lemma}[Decomposition and estimates for the GGM remainder]
	\label{lem:CR-GGM-remainder-ledger}
	Let \(t>0\) and \(\varrho_t=(2t)^{1/2}\).  Suppose
	\(\{h_n^\circ\}\) is relatively compact in \(H_+^{1+s}\), and set
	\begin{align*}
		f_{n,t,z,\lambda}
		&=(X^*-2t(L_{u_{0,n}}-\lambda)-z)^{-1}h_n^\circ,\\
		q_{n,t,z,\lambda}
		&=h_n^\circ-2tT_{u_{0,n}}f_{n,t,z,\lambda}
		=(X^*-2t(D-\lambda)-z)f_{n,t,z,\lambda},
	\end{align*}
	with $F_{n,t,z,\lambda}=\varrho_t f_{n,t,z,\lambda}$.
	Use the kernels and threshold contribution just defined by
	Gassot, G\'erard, and Miller \cite[(5.4), (8.4)--(8.7)]{GGM},
	in the form \eqref{eq:CR-GGM-p-by-parts}--
	\eqref{eq:CR-GGM-threshold-scalar}.  Abbreviate
	\(f=f_{n,t,z,\lambda}\), \(F=F_{n,t,z,\lambda}\),
	\(q=q_{n,t,z,\lambda}\), and \(p=p_{n,t,z,\lambda}\).  Use the decomposition
	$\Pi p=p^{\rm th}+p^{\rm out}+p^{\rm c}$ induced respectively by the \(b_0\) threshold term, the field
	\(C_n(a)e^{iax}b_+(x)\), and the regularized outgoing field
	\eqref{eq:CR-GGM-centered-field}.  Set
	\(K^{\rm c}=\varrho_t^{-1}\Pi
	(x+2it\partial_x+2t\lambda-\bar z)p^{\rm c}\).  Then
	\begin{equation}
		\langle q,\Pi p\rangle
		=\frac{i}{\varrho_t}\overline{I_+(F)}\mathfrak t^0
		+\langle h_n^\circ,p^{\rm out}\rangle
		-\langle F,\varrho_t u_{0,n} p^{\rm out}\rangle
		+\langle F,K^{\rm c}\rangle .
		\label{eq:CR-GGM-master-identity}
	\end{equation}
	Here \(\mathfrak t^0\) is the threshold scalar in the \(b_0\) contribution.  For every
	compact \(K\subset\C_+\), every
	\(\psi\in L^2(J)\cap L^\infty(J)\), and each
	\(1/2<\be_0<\be\), the four terms satisfy, uniformly for
	\(n\), \(t\ge T\), and \(z\in K\),
	\begin{align}
		&\|\varrho_t^{-1}\overline{I_+(F)}\mathfrak t^0\|_{L^1(J)}=O(\varrho_t^{-1}),
		\label{eq:CR-GGM-threshold-ledger}\\
		&\|\langle h_n^\circ,p^{\rm out}\rangle\|_{L^2(J)}=o(1),
		\label{eq:CR-GGM-first-slot-ledger}\\
		&\|\varrho_t u_{0,n} p^{\rm out}\|_{L^2(J;H^{-s-1})}
		\le C\varrho_t^{1-2\be_0}=o(1),
		\label{eq:CR-GGM-product-ledger}\\
		&\sup_{\lambda\in J}\|K^{\rm c}(\lambda)\|_{H^{-s-1}}=o(1).
		\label{eq:CR-GGM-centered-ledger}
	\end{align}
	Consequently
	\begin{equation}
		\lim_{T\to\infty}
		\sup_{\substack{n,\ t\ge T,\ z\in K}}
		\left|\int_J\psi(\lambda)
		\langle q_{n,t,z,\lambda},\Pi p_{n,t,z,\lambda}\rangle\,
		\rmd\lambda\right|=0.
		\label{eq:CR-GGM-tested-remainder}
	\end{equation}
\end{lemma}

\begin{proof}[Proof outline]
The Volterra identity
$(x+2it\partial_x+2t\lambda-\bar z)\mathcal WH=2itH$ gives
\[
 K^{\rm c}=-i\int_I e^{-it(a-\lambda)^2-i\bar z(a-\lambda)}
 \Phi_{n,a}\,\rmd a.
\]
Thus \eqref{eq:CR-GGM-mixed-Fresnel} proves the centered estimate.
The triangular Fourier representation of $\mathcal W$ shows that the
outgoing and centered contributions have zero Green trace.  The
threshold contribution solves the homogeneous free moving-resolvent
equation and has trace $2\pi\mathfrak t^0$.  Green's formula therefore
gives \eqref{eq:CR-GGM-master-identity}.

Spectral averaging bounds $F$ in $L^2(J;H^{1+s})$.  The coordinate
formula and Hardy Plancherel bound $I_+(F)$ in $L^2(J)$, while Fresnel
Plancherel and the localized convergence of $g_{u_{0,n}}^-$ bound the
threshold scalar.  Consequently
\begin{equation}
 \int_J\|F\|_{H^{1+s}}^2\,\rmd\lambda
 +\|I_+(F)\|_{L^2(J)}^2+\|\mathfrak t^0\|_{L^2(J)}^2\le C.
 \label{eq:CR-GGM-trace-ray-bound}
\end{equation}
This proves \eqref{eq:CR-GGM-threshold-ledger}.

For the source pairing, integration by parts using $C_n=H_n'$ gives
\begin{align}
 \langle h_n^\circ,p^{\rm out}\rangle
 ={}&\varrho_t\int_\R b_+'(y)\mathcal A_{n,t,z}(\lambda;y)
                \mathcal L_{n,t,z}(\lambda;y)\,\rmd y\notag\\
 &-\varrho_t\int_IH_n(a)e^{-it(a-\lambda)^2-i\bar z(a-\lambda)}
             \mathcal B_{n,b}(a)\,\rmd a,
 \label{eq:CR-GGM-smooth-first-slot-ledger}
\end{align}
where
\begin{align*}
 \mathcal A_{n,t,z}(\lambda;y)
 &=\int_y^\infty\overline{h_n^\circ(x)}e^{i\lambda x+ix^2/(4t)-i\bar zx/(2t)}\,\rmd x,\\
 \mathcal L_{n,t,z}(\lambda;y)
 &=\int_IH_n(a)e^{-it(a-\lambda)^2-i\bar z(a-\lambda)
       +iy(a-\lambda)-iy^2/(4t)+i\bar zy/(2t)}\,\rmd a,\\
 \mathcal B_{n,b}(a)&=\int_\R b_+(x)\overline{h_n^\circ(x)}e^{iax}\,\rmd x.
\end{align*}
The two terms have the same leading limit.  To justify this cancellation,
choose $\eta_R\in C_c^\infty(\R)$ with $0\le\eta_R\le1$ equal to one on $[-R,R]$,
truncate $b_+'$ to $\zeta_R=\eta_Rb_+'$, and set
\[
b_{+,R}(x)=\int_{-\infty}^x\zeta_R(y)\,\rmd y.
\]
Then
\begin{equation}
 \|b_{+,R}-b_+\|_\infty\le\|\zeta_R-b_+'\|_1\longrightarrow0.
 \label{eq:CR-GGM-step-truncation-error}
\end{equation}
Writing $\mathcal B_{n,R}$ for $\mathcal B_{n,b}$ with $b_+$ replaced by
$b_{+,R}$, Fubini yields
\begin{equation}
 \mathcal B_{n,R,t,z}:=\int_\R\zeta_R(y)\mathcal A_{n,t,z}(\lambda;y)\,\rmd y
 =\int_\R b_{+,R}(x)\overline{h_n^\circ(x)}e^{i\lambda x+ix^2/(4t)-i\bar zx/(2t)}\,\rmd x.
 \label{eq:CR-GGM-same-word-FTOC}
\end{equation}
For fixed $R$, stationary phase and Fresnel Plancherel give, with
$c_{\rm Fr}=\sqrt{2\pi}e^{-i\pi/4}$,
\begin{align}
 \varrho_t\mathcal L_{n,t,z}(\lambda;y)&\longrightarrow c_{\rm Fr}H_n(\lambda)
 \quad\hbox{uniformly for }y\in\operatorname{supp}\zeta_R,
 \label{eq:CR-GGM-L-uniform-limit}\\
 \mathcal B_{n,R,t,z}&\longrightarrow\mathcal B_{n,R}\quad\hbox{in }L^2(J),
 \label{eq:CR-GGM-BR-uniform-limit}\\
 \varrho_t\int_I H_n(a)e^{-it(a-\lambda)^2-i\bar z(a-\lambda)}
 \mathcal B_{n,R}(a)\,\rmd a
 &\longrightarrow c_{\rm Fr}H_n\mathcal B_{n,R}\quad\hbox{in }L^2(J).
 \label{eq:CR-GGM-second-line-limit}
\end{align}
All three limits are uniform in $n,z$ by compactness of the source and
coefficient families.  The opposite signs in
\eqref{eq:CR-GGM-smooth-first-slot-ledger} cancel these limits.
The truncation errors are bounded uniformly by
$C\|\zeta_R-b_+'\|_1+C\|b_{+,R}-b_+\|_\infty\|h_n^\circ\|_2$:
the exponential weights in the two Volterra factors cancel, and
Plancherel controls the second term.  Taking $t\to\infty$ first and
then $R\to\infty$ proves \eqref{eq:CR-GGM-first-slot-ledger}.

For the product term, put $\varepsilon=\varrho_t^{-1}$ and
$m_{\varepsilon,k}=\min\{\varepsilon^2\langle k\rangle,1\}$.
Complete the square in the Volterra representation and use localized
Fresnel Plancherel in $\lambda$.  Volterra damping on the positive cells
and the Schwartz left tail of $b_+$ give uniform $L^2_x$ and $H_x^1$
block bounds.  Interpolation yields
\[
 \|\theta_kp^{\rm out}\|_{L^2_\lambda H_x^{1/2}}
 \le C\langle\varepsilon^2k\rangle^{1/2}m_{\varepsilon,k}.
\]
The unit-block product estimate and the uniform $\be_0$-weighted
potential bound then sum these estimates to give
\[
 \|\varepsilon^{-1}u_{0,n}p^{\rm out}\|_{L^2(J;H^{-s-1})}
 \le C\varepsilon^{2\be_0-1}=o(1),
\]
which is \eqref{eq:CR-GGM-product-ledger}.
Sobolev duality and Cauchy--Schwarz, using
\eqref{eq:CR-GGM-trace-ray-bound}, combine all four bounds into
\eqref{eq:CR-GGM-tested-remainder}.
\end{proof}

\begin{lemma}[Localization away from the spectral diagonal at rough regularity]
	\label{lem:CR-GGM-rough-localization}
	Let \(\chi\in C_c^\infty(0,\infty)\) equal one on a neighborhood of
	\(J\), \(t>0\), \(\varrho_t=\sqrt{2t}\), and let \(F,q\) be as in
	Lemma~\ref{lem:CR-GGM-remainder-ledger}.  Then for every
	\(\psi\in C_c^\infty(J)\), uniformly in \(n\) and \(z\) in compact
	subsets of \(\C_+\),
	\begin{equation}
		\int_J\psi(\lambda)\varrho_t c_{\rm F}\int_0^\infty
		e^{it(\xi-\lambda)^2+iz\xi}(1-\chi(\xi))
		\widehat q(\xi)\,\rmd\xi\rmd\lambda\longrightarrow0.
		\label{eq:CR-GGM-off-window}
	\end{equation}
	Frequency integration by parts differentiates only the smooth kernel.
\end{lemma}

\begin{proof}[Proof outline]
The normalized Fresnel operators converge strongly to multiplication by
$e^{iz\lambda}$.  Since $(1-\chi)|_J=0$, compactness of the sources
removes their off-window contribution uniformly.

For the forcing, Fourier inversion gives a pairing of $F$ with
$u_{0,n}\mathscr K$, where
\[
 \mathscr K_{t,z,\lambda}(x)=2tc_{\rm F}\int_0^\infty
 e^{it(\xi-\lambda)^2+iz\xi-ix\xi}(1-\chi(\xi))\,\rmd\xi.
\]
Put $d=\operatorname{dist}(J,\operatorname{supp}(1-\chi))>0$.
Integration by parts on the smooth kernel gives the endpoint expansion
\begin{equation}
 \mathscr K_{t,z,\lambda}(x)
 =c_{\rm F}\frac{e^{it\lambda^2}}{i\{\lambda+(x-z)/(2t)\}}
 +\mathscr E_{t,z,\lambda}(x),\qquad
 \|\mathscr E\|_{C^N(|x|\le dt/2)}\le C_Nt^{-1}.
 \label{eq:CR-GGM-kernel-endpoint}
\end{equation}
Combining its interior estimate with van der Corput on the outer blocks
and the weighted local product bound yields, for $1/2<\be_0<\be$,
\[
 \left\|u_{0,n}\left(\mathscr K_{t,z,\lambda}
 -c_{\rm F}\frac{e^{it\lambda^2}}{i\lambda}\right)\right\|_{L^2(J;H^{-s-1})}
 \le C(t^{-1}+t^{-\be_0}+t^{1/2-\be_0})=o(1).
\]
Spectral averaging removes this error.  The remaining endpoint is a
bounded scalar multiplier paired with $(u_{0,n})_+$ and $F$.
Weighted approximation and $H^s\hookrightarrow H^{-s-1}$ make these
potential vectors relatively compact in $H^{-s-1}$, so the multiplier
assertion of Lemma~\ref{lem:CR-GGM-ray-averaging} removes the endpoint.
\end{proof}

\begin{theorem}[Compact-window scattering asymptotics]
	\label{thm:CR-GGM-continuous-sewing}
	Let \(u_0\in H^{s}(\R;\R)\) satisfy
	\(\langle x\rangle^\be u_0\in H^{s}\), where
	\(-1/2<s<0\) and \(\be>1/2\).  For every
	\(0<\delta<R<\infty\) and \(\sigma\in\{+,-\}\),
	\begin{equation}
		\lim_{t\to\sigma\infty}
		\left\| U_{u_0}^{\mathrm{cyc}}(t)E_{u_0}([\delta,R])(u_0)_+
		-\mathcal F_0^*\!\left(
		e^{it\lambda^2}\mathbf1_{[\delta,R]}(\lambda)
		g_{u_0}^{-\sigma}(\lambda)\right)\right\|_{H_+^s}=0.
		\label{eq:CR-GGM-continuous-sewing}
	\end{equation}
\end{theorem}

\begin{proof}[Proof outline]
First take a smooth cutoff $\chi_0\in C_c^\infty(J)$ and set
$h_n=h_{\chi_0,n}$, $b_n=\chi_0g_{u_{0,n}}^-$.
The compact spectral-kernel lemma and localized scattering-data
convergence give
\[
 \{h_n\}\Subset H_+^{1+s},\qquad
 b_n\longrightarrow\chi_0g_{u_0}^-\quad\hbox{in }L^2.
\]
Apply Lemmas~\ref{lem:CR-GGM-remainder-ledger} and
\ref{lem:CR-GGM-rough-localization} with $h_n^\circ=h_n$ and use their
abbreviations $f,F,q,p$.
The GGM identities extend by linearity in the source and graph-core
density to these $h_n$:
\begin{align}
 &\varrho_t e^{it\lambda^2}\mathcal E_{z-2t\lambda}U_n(t,0)h_n\notag\\
 &\quad=\varrho_t c_{\rm F}\int_0^\infty
 e^{it(a-\lambda)^2+iza}(1-\chi(a))\widehat q(a)\,\rmd a
 \label{eq:CR-GGM-arbitrary-source-row}\\
 &\qquad+\varrho_t c_{\rm F}\int_0^\infty
 e^{it(a-\lambda)^2+iza}\chi(a)b_n(a)\,\rmd a
 +\frac{e^{iz\lambda}}{2\pi}\overline{\langle q,\Pi p\rangle}.\notag
\end{align}
The free ray cancels the middle term because $\chi b_n=b_n$.
Lemmas~\ref{lem:CR-GGM-rough-localization} and
\ref{lem:CR-GGM-remainder-ledger} remove the other two terms after
testing in $\lambda$.  For
$g_{n,t}=U_n(t,0)h_n-\mathcal F_0^*(e^{it\lambda^2}b_n)$ this yields
\begin{equation}
 \lim_{T\to\infty}\sup_{n,t\ge T,z\in K}
 \left|\int_J\psi(\lambda)\varrho_t e^{it\lambda^2}
 \mathcal E_{z-2t\lambda}g_{n,t}\,\rmd\lambda\right|=0
 \quad(\psi\in C_c^\infty(J)).
 \label{eq:CR-GGM-ray-WOT}
\end{equation}

For $d_{n,t}=e^{-itD^2}g_{n,t}$, the ray observable is a normalized
Fresnel transform of $e^{iz\xi}\widehat d_{n,t}$.
Move this unitary transform to fixed testers and use its strong limit
$I$; then $\langle\widehat d_{n,t},\phi\rangle\to0$ uniformly in $n$
for testers supported in $J$.  A finite net permits the compact family
$\phi=b_n$.  Distorted Plancherel supplies matching norms, so
\[
 \|e^{-itD^2}U_n(t,0)h_n-\mathcal F_0^*b_n\|_2^2
 =-2\Re\langle\widehat d_{n,t},b_n\rangle\longrightarrow0
\]
uniformly in $n$.
At fixed $t$, spectral-kernel and smooth-flow convergence pass both
vectors to their rough-data limits in $L^2$.  Take $n\to\infty$ in the
uniform estimate first and then $t\to+\infty$ to obtain
\begin{equation}
 \lim_{t\to+\infty}
 \|U_{u_0}^{\mathrm{cyc}}(t)\chi_0(L_{u_0})(u_0)_+
 -\mathcal F_0^*(e^{it\lambda^2}\chi_0g_{u_0}^-)\|_{H_+^s}=0.
 \label{eq:CR-GGM-smooth-window-sewing}
\end{equation}
Approximate $\mathbf1_{[\delta,R]}$ by smooth cutoffs
$0\le\chi_\epsilon\le1$, equal to one on $[\delta,R]$ and supported in
$(\delta-\epsilon,R+\epsilon)$.  Distorted and free Plancherel bound both errors,
uniformly in time, by the weighted spectral mass of the excess intervals.
Take the late-time upper limit first and then shrink the enlargement;
absolute continuity of that mass proves the positive-time assertion.

For negative time, reflect the data by $u_{0,n}^\sharp(x)=u_{0,n}(-x)$
and use $\mathscr Rf(x)=\overline{f(-x)}$.
Covariance of the flow and spectral cutoffs reverses time, while
covariance of the distorted transform interchanges its two signs.
Apply the positive-time argument to this reflected family, taking the
rough-data limit at fixed time, then the time limit, then the cutoff
limit.  The resulting profile is $g_{u_0}^+$, with the same weighted norm
by \eqref{eq:SC-gq-phase}.
\end{proof}

\section{Spectral decomposition and ordered resolution}
\label{sec:four-block-resolution}

We assume the hypotheses of Theorem~\ref{thm:main-resolution}.
Write $ U_{u_0}^{\mathrm{cyc}}(t)= U_{u_0}^{\mathrm{cyc}}(t,0)$,
$W(t)=e^{t\partial_x|D|}$, $A_0=L_{u_0}+\ka\ge1$, and
$u_+(t)=\Pi u(t)= U_{u_0}^{\mathrm{cyc}}(t)(u_0)_+$.
For a Borel set $I\subset(0,\infty)$ and $\sigma\in\{+,-\}$ define the
positive-frequency free solution
$r_{I,+}^{\sigma}(t)=\mathcal F_0^*
(e^{it\lambda^2}\mathbf1_I(\lambda)g_{u_0}^{-\sigma}(\lambda))$.
Definition~\ref{def:SC-radiation} gives
$2\Re r_{(0,\infty),+}^{\sigma}(t) =W(t)u_\infty^\sigma.$

\subsection{Discrete, threshold, and high-energy estimates}

\begin{lemma}[Strong asymptotics of an evolving eigenfunction]
	\label{lem:CR-one-atom}
	Let $L_{u_0}f_j=-a_jf_j$, $\|f_j\|_2=1$, and put
	$P_j=E_{u_0}(\{-a_j\})$, $h_j=P_j(u_0)_+=d_jf_j$, and
	$d_j=\langle f_j,(u_0)_+\rangle_{H^{-s},H^{s}}$.
	Here $P_j$ acts as the bounded atomic spectral projection on the cyclic
	$s$-scale.  Its range is the one-dimensional $L^2$ eigenspace, and
	$P_j(u_0)_+=d_jf_j\in L^2$.  Then, in
	both time directions,
	\begin{equation}
		\lim_{t\to\pm\infty}
		\left\| U_{u_0}^{\mathrm{cyc}}(t)h_j
		-\Pi R_{p_j}(\,\cdot-2a_jt\,)\right\|_{L^2}=0.
		\label{eq:CR-atomic-strong-L2}
	\end{equation}
	Consequently the same convergence holds in $H_+^s$.
\end{lemma}

\begin{proof}[Proof outline]
The graph-convergent Riesz eigenvectors give $h_{j,n}\to h_j$ in $L^2$.
Fixed-contour calculus and intertwining yield convergence of their
flows at each fixed $t$.  Hardy evaluation supplies the corresponding
hypothesis of Lemma~\ref{lem:ordered-Schur-extraction}.

That lemma, \eqref{eq:Wu-phase-product}, and \eqref{eq:pj-graph-data}
give, locally uniformly in $z\in\C_+$,
\begin{equation}
 [U_{u_0}^{\mathrm{cyc}}(t)h_j](z+2a_jt)
 \longrightarrow\frac1{i(b_j-z)},\qquad
 b_j=\langle f_j,X^*f_j\rangle=c_j-\frac{i}{2a_j},
 \quad t\to\pm\infty.
 \label{eq:CR-atomic-Hardy-limit}
\end{equation}
Here $n\to\infty$ is taken at fixed $t$ before $|t|\to\infty$, removing
the approximating velocity error.  Since $p_j=-b_j$, this is the Hardy
profile $\Pi R_{p_j}$.

Totality of the Laplace kernels gives weak $L^2$ convergence.
Cyclic isometry, Wu's residue, and the soliton Fourier formula give
matching squared norms $2\pi a_j$, upgrading it to strong $L^2$ and
hence $H^s$ convergence.
\end{proof}

For $\delta>0$ put $P_\delta^-=E_{u_0}((-\infty,-\delta])$ and
$S_{\delta,+}(t)=\sum_{a_j\ge\delta}\Pi R_{p_j}(\,\cdot-2a_jt\,)$.

\begin{coro}[Reconstruction of the truncated discrete spectrum]
	\label{coro:CR-hard-discrete}
	For each fixed $\delta>0$ and $\sigma\in\{+,-\}$,
	\[
\lim_{t\to\sigma\infty}
		\| U_{u_0}^{\mathrm{cyc}}(t)P_\delta^-(u_0)_+-S_{\delta,+}(t)\|_{H_+^s}=0.
\]
\end{coro}

\begin{proof}[Proof outline]
Since $\sum_ja_j<\infty$, only finitely many $a_j$ satisfy $a_j\ge\delta$.
Expand $P_\delta^-(u_0)_+$ as their atomic sum and apply
Lemma~\ref{lem:CR-one-atom} termwise.  The same graph data $p_j$ apply
in both time directions.
\end{proof}

Put $P_\delta^0=E_{u_0}((-\delta,\delta))$ and
$H_R=E_{u_0}((R,\infty))$.  Set
$\Theta_s(\delta)=\|P_\delta^0 A_0^s(u_0)_+\|_2$.

\begin{lemma}[Threshold and high-energy bounds for the fixed spectral vector]
	\label{lem:CR-neutral-high}
	There is a constant $C$, independent of $t,\delta,R$, such that
	\begin{align}
		\sup_{t\in\R}\| U_{u_0}^{\mathrm{cyc}}(t)P_\delta^0 (u_0)_+\|_{H_+^s}
		&\le C\Theta_s(\delta),                                 \label{eq:CR-neutral-Peter}\\
		\sup_{t\in\R}\| U_{u_0}^{\mathrm{cyc}}(t)H_R(u_0)_+\|_{H_+^s}
		&\le C\|H_RA_0^s(u_0)_+\|_2.                                 \label{eq:CR-high-Peter}
	\end{align}
	Moreover,
	\begin{align}
		\Theta_s(\delta)^2
		&=2\pi\sum_{a_j<\delta}a_j(\ka-a_j)^{2s}
		+\int_0^\delta(\lambda+\ka)^{2s}
		|g_{u_0}^\pm(\lambda)|^2\,\rmd\lambda,                   \label{eq:CR-neutral-ledger}\\
		\|H_RA_0^s(u_0)_+\|_2^2
		&=\int_R^\infty(\lambda+\ka)^{2s}
		|g_{u_0}^\pm(\lambda)|^2\,\rmd\lambda.                  \label{eq:CR-high-ledger}
	\end{align}
	In particular,
	\begin{equation}
		\Theta_s(\delta)\to0\enspace(\delta\downarrow0),
		\enspace
		\|H_RA_0^s(u_0)_+\|_2\to0\enspace(R\to\infty).
		\label{eq:CR-tail-limits}
	\end{equation}
\end{lemma}

\begin{proof}[Proof outline]
Cyclic isometry, intertwining, and uniform scale equivalence give
\[
 \|U_{u_0}^{\mathrm{cyc}}(t)E_{u_0}(B)(u_0)_+\|_{H_+^s}
 \le C\|E_{u_0}(B)A_0^s(u_0)_+\|_2.
\]
Use the threshold and high-energy sets for $B$.
The input-coordinate formula supplies atomic weights $2\pi a_j$ and
density $|g_{u_0}^{\pm}|^2$, proving both spectral identities.
Finally the finite spectral measure of $A_0^s(u_0)_+$, together with
$E_{u_0}(\{0\})=0$, makes the two tails vanish.
\end{proof}

By \eqref{eq:SC-radiation-positive-frequency} and equivalence of
$\langle\lambda\rangle^s$ and $(\lambda+\ka)^s$ on $(0,\infty)$,
\begin{align}
	\sup_t\|r_{(0,\delta),+}^{\sigma}(t)\|_{H_+^s}
	&\le C\Theta_s(\delta),                                 \label{eq:CR-free-low}\\
	\sup_t\|r_{(R,\infty),+}^{\sigma}(t)\|_{H_+^s}^2
	&\le C\int_R^\infty(\lambda+\ka)^{2s}
	|g_{u_0}^{-\sigma}(\lambda)|^2\,\rmd\lambda
	\longrightarrow0.                                      \label{eq:CR-free-high}
\end{align}
Since $|e^{it\lambda^2}|=1$, both spectral norms are independent of time.

\subsection{Spectral decomposition and resolution in Sobolev spaces}
\label{subsec:CR-nonnegative}

Fix $0<\delta<R$.  Apply $ U_{u_0}^{\mathrm{cyc}}(t)$ to the disjoint
spectral partition $(-\infty,-\delta]$, $(-\delta,\delta)$,
$[\delta,R]$, $(R,\infty)$, and split the free solution over the
corresponding positive intervals.  In the cyclic $s$-scale this gives
\begin{equation}
\begin{aligned}
 u_+(t)-S_{\delta,+}(t)-r_{(0,\infty),+}^\sigma(t)
 &=\bigl[ U_{u_0}^{\mathrm{cyc}}(t)P_\delta^-(u_0)_+-S_{\delta,+}(t)\bigr]\\
 &\enspace+\bigl[ U_{u_0}^{\mathrm{cyc}}(t)P_\delta^0 (u_0)_+
 -r_{(0,\delta),+}^\sigma(t)\bigr]\\
 &\enspace+\bigl[ U_{u_0}^{\mathrm{cyc}}(t)E_{u_0}([\delta,R])(u_0)_+
 -r_{[\delta,R],+}^\sigma(t)\bigr]\\
 &\enspace+\bigl[ U_{u_0}^{\mathrm{cyc}}(t)H_R(u_0)_+
 -r_{(R,\infty),+}^\sigma(t)\bigr].
\end{aligned}
\label{eq:CR-four-block-identity}
\end{equation}

For every real $F\in H^s$, the two Fourier half-lines are orthogonal, so
\begin{equation}
	F=2\Re \Pi F,
	\enspace \|F\|_{H^s}^2=2\|\Pi F\|_{H_+^s}^2.
	\label{eq:CR-real-reconstruction}
\end{equation}

\begin{prop}[Effective remainder modulus]
	\label{prop:CR-effective-four-block}
	For $T\ge1$ and $\sigma\in\{+,-\}$, define
	\[
 v_j(t,x)=[ U_{u_0}^{\mathrm{cyc}}(t)h_j](x+2a_jt),\enspace
 r_j=\Pi R_{p_j},\enspace
 \mathfrak a_j^\sigma(T):=\sup_{\sigma t\ge T}
 \left|\left\langle v_j(t)-r_j,r_j\right\rangle_{L^2}\right|,
\]
	where $h_j=P_j(u_0)_+$.  For $0<\delta<R$, also put
	\begin{align*}
		x_{\delta,R}(t)
		&:=e^{-itD^2} U_{u_0}^{\mathrm{cyc}}(t)E_{u_0}([\delta,R])(u_0)_+,\\
		y_{\delta,R}^\sigma
		&:=\mathcal F_0^*\!\left(
		\mathbf1_{[\delta,R]}g_{u_0}^{-\sigma}\right),\\
		\mathfrak c_{\delta,R}^\sigma(T)
		&:=\sup_{\sigma t\ge T}
		\left|\left\langle x_{\delta,R}(t)-y_{\delta,R}^\sigma,
		y_{\delta,R}^\sigma\right\rangle_{L^2}\right|,\\
		\mathfrak h_s(R)
		&:=\left\{\int_R^\infty(\lambda+\ka)^{2s}
		|g_{u_0}^+(\lambda)|^2\,\rmd\lambda\right\}^{1/2}
		=\left\{\int_R^\infty(\lambda+\ka)^{2s}
		|g_{u_0}^-(\lambda)|^2\,\rmd\lambda\right\}^{1/2}.
\end{align*}
	The equality in the last definition follows from \eqref{eq:SC-gq-phase}.
	For every fixed $j$ and
	every fixed $0<\delta<R$,
	\begin{equation}
		\mathfrak a_j^\sigma(T)\downarrow0,
		\enspace
		\mathfrak c_{\delta,R}^\sigma(T)\downarrow0
		\enspace(T\to\infty).
		\label{eq:CR-correlation-moduli-vanish}
	\end{equation}
	There is a constant $C_*$, independent of $T,\delta,R$, such that
	\begin{align}
		&\sup_{\sigma t\ge T}
		\left\|u(t)-W(t)u_\infty^\sigma
		-\sum_{a_j\ge\delta}R_{p_j}(\,\cdot-2a_jt\,)
		\right\|_{H^s}
		\notag\\
		&\enspace\le
		2\sum_{a_j\ge\delta}\sqrt{\mathfrak a_j^\sigma(T)}
		+2\sqrt{\mathfrak c_{\delta,R}^\sigma(T)}
		+C_*\bigl\{\Theta_s(\delta)+\mathfrak h_s(R)\bigr\}.
		\label{eq:CR-effective-four-block}
	\end{align}
	The sum on the right is finite.
\end{prop}

\begin{proof}[Proof outline]
Translation, cyclic isometry, and Wu's residue give
$\|v_j(t)\|_2=\|r_j\|_2$.  The equal-norm identity
$\|v-r\|_2^2=-2\Re\langle v-r,r\rangle$ therefore yields, uniformly
for $\sigma t\ge T$,
\begin{equation}
 \|U_{u_0}^{\mathrm{cyc}}(t)h_j
 -\Pi R_{p_j}(\,\cdot-2a_jt\,)\|_{H_+^s}
 \le\sqrt{2\mathfrak a_j^\sigma(T)}.
 \label{eq:CR-atomic-modulus-bound}
\end{equation}
Lemma~\ref{lem:CR-one-atom} makes the atomic modulus vanish.

Distorted Plancherel gives the analogous equality of norms for
$x_{\delta,R}(t)$ and $y_{\delta,R}^\sigma$, so
\begin{equation}
 \|U_{u_0}^{\mathrm{cyc}}(t)E_{u_0}([\delta,R])(u_0)_+
 -r_{[\delta,R],+}^\sigma(t)\|_{H_+^s}
 \le\sqrt{2\mathfrak c_{\delta,R}^\sigma(T)}
 \quad(\sigma t\ge T).
 \label{eq:CR-continuous-modulus-bound}
\end{equation}
The compact-window proof gives $L^2$ convergence, including its
sharp-cutoff approximation, and hence vanishing of this modulus.

Insert these estimates and the threshold and high-energy bounds into
\eqref{eq:CR-four-block-identity}.  The triangle inequality and real
reconstruction give \eqref{eq:CR-effective-four-block}; the factor
$\sqrt2$ in reconstruction produces the displayed coefficients $2$.
\end{proof}

\begin{prop}[Uniform finite-cutoff remainder estimate]
	\label{prop:CR-finite-cutoff}
	There is $C_0$, independent of $\delta$, such that
	\[
\varepsilon_\delta^\sigma:=
		\limsup_{t\to\sigma\infty}
		\left\|u(t)-W(t)u_\infty^\sigma
		-\sum_{a_j\ge\delta}R_{p_j}(\,\cdot-2a_jt\,)
		\right\|_{H^s}
		\le C_0\Theta_s(\delta).
\]
	Consequently $\varepsilon_\delta^\sigma\to0$ as $\delta\downarrow0$.
\end{prop}

\begin{proof}[Proof outline]
In \eqref{eq:CR-effective-four-block}, first let $T\to\infty$ at fixed
$\delta,R$, removing the finite atomic sum and compact-window modulus.
Letting $R\to\infty$ gives the fixed-$\delta$ bound with $C_0=C_*$.
Finally $\delta\downarrow0$ makes that bound vanish by
\eqref{eq:CR-tail-limits}.
\end{proof}

\begin{coro}[A two-sided remainder modulus for finite negative spectrum]
	\label{coro:CR-finite-spectrum-error-modulus}
	Assume that the negative spectrum consists of
	$-a_1,\ldots,-a_N$, where $0\le N<\infty$.  Put
	\[
a_{\min}:=
		\begin{cases}
			1,&N=0,\\
			\min_{1\le j\le N}a_j,&N\ge1,
		\end{cases}
\]
	and interpret atomic sums as empty when $N=0$.  Fix an admissible $C_*$
	from Proposition~\ref{prop:CR-effective-four-block}; dependence on
	$s$ and this choice is suppressed.  For $T\ge1$ set
	\begin{equation}
		\mathfrak a_j^{\leftrightarrow}(T)
		:=\max_{\sigma\in\{+,-\}}\mathfrak a_j^\sigma(T),
		\enspace
		\mathfrak c_{\delta,R}^{\leftrightarrow}(T)
		:=\max_{\sigma\in\{+,-\}}
		\mathfrak c_{\delta,R}^\sigma(T),
		\label{eq:CR-two-sided-correlation-moduli}
	\end{equation}
	and
	\begin{align}
		\mathcal R_{u_0,N}(T)
		:={}2\sum_{j=1}^N
		\sqrt{\mathfrak a_j^{\leftrightarrow}(T)}
		+\inf_{\substack{0<\delta<a_{\min}\\
				R>\max\{1,\delta\}}}
		\left[
		2\sqrt{\mathfrak c_{\delta,R}^{\leftrightarrow}(T)}
		+C_*\{\Theta_s(\delta)+\mathfrak h_s(R)\}
		\right].
		\label{eq:CR-effective-finite-spectrum-modulus}
	\end{align}
	Then $\mathcal R_{u_0,N}$ is nonincreasing,
	$\mathcal R_{u_0,N}(T)\to0$, and
	\begin{equation}
		\max_{\sigma\in\{+,-\}}\ \sup_{\sigma t\ge T}
		\left\|u(t)-W(t)u_\infty^\sigma
		-\sum_{j=1}^NR_{p_j}(\,\cdot-2a_jt\,)\right\|_{H^s}
		\le\mathcal R_{u_0,N}(T).
		\label{eq:CR-finite-spectrum-effective-error}
	\end{equation}
	The time-dependent terms in \eqref{eq:CR-effective-finite-spectrum-modulus}
	are the following weak observables.  If
	$b_j=-p_j=c_j-i/(2a_j)$, then
	\begin{equation}
		\mathfrak a_j^{\leftrightarrow}(T)
		=2\pi\sup_{|t|\ge T}
		\left|\mathcal E_{\overline b_j}v_j(t)-a_j\right|,
		\label{eq:CR-atomic-Laplace-rate}
	\end{equation}
	while Plancherel gives
	\begin{align}
		\mathfrak c_{\delta,R}^{\leftrightarrow}(T)
		={}&\max_{\sigma\in\{+,-\}}\ \sup_{\sigma t\ge T}
		\left|\int_\delta^R
		\overline{
			(\mathcal F_0x_{\delta,R}(t))(\lambda)
			-g_{u_0}^{-\sigma}(\lambda)}
		g_{u_0}^{-\sigma}(\lambda)\,\rmd\lambda\right|.
		\label{eq:CR-continuous-window-observable-rate}
	\end{align}
\end{coro}

\begin{proof}[Proof outline]
For $\delta<a_{\min}$ all atoms occur in the discrete block.  Take the
time supremum, the maximum over the two signs, and the infimum over
admissible $(\delta,R)$ in Proposition~\ref{prop:CR-effective-four-block}.
This gives \eqref{eq:CR-finite-spectrum-effective-error}.

The soliton Fourier formula identifies
$\langle v_j-r_j,r_j\rangle=2\pi\overline{\mathcal E_{\overline b_j}(v_j-r_j)}$
and $\mathcal E_{\overline b_j}r_j=a_j$, proving the atomic observable
formula.  Free Plancherel gives the continuous observable formula.
Both correlation moduli decrease to zero.  Choose $\delta$ small and
$R$ large to control the static tails, then let $T\to\infty$ at those
fixed cutoffs.  This proves the stated limit and monotonicity.
\end{proof}

To treat nonnegative regularity, fix ${r_*}\ge0$ and $s_0\in(-\frac12,0)$.  If
$u_0,\langle x\rangle^\beta u_0\in H^{r_*}$, they also belong to
$H^{s_0}$, so all the preceding constructions apply.  The closed-form Lax
operator and its spectral measure are independent of this choice of
$s_0$.  Use $u_+(t)=\Pi u(t)$ and
$A_t=L_{u(t)}+\ka\ge1$ from Section~\ref{subsec:notation-guide},
with one fixed shift.  Since $(u_0)_+\in L^2$,
the cyclic isometry and spectral intertwining give
\begin{equation}
 u_+(t)= U_{u_0}^{\mathrm{cyc}}(t)(u_0)_+,\enspace
 A_t^r U_{u_0}^{\mathrm{cyc}}(t)h
 = U_{u_0}^{\mathrm{cyc}}(t)A_0^rh,
 \label{eq:CR-higher-intertwining}
\end{equation}
for every $r\ge0$ and $h\in\operatorname{Dom}(A_0^r)$ in the cyclic
subspace.  Bounded spectral truncation and the spectral theorem give
the identity for unbounded powers.  The resulting $u_+(t)$ agrees in $H^{s_0}$ with
the original cyclic construction.

If $m\ge1$ is
an integer, $v\in H^{m-1}(\mathbb R;\mathbb R)$, and
$A_v=L_v+\ka\ge1$, then
\begin{equation}
 \operatorname{Dom}(A_v^m)=H_+^m,\enspace
 \|f\|_{H^r}\lesssim\|A_v^rf\|_2\lesssim\|f\|_{H^r}\enspace(0\le r\le m).
 \label{eq:CR-higher-domain}
\end{equation}
The constants are uniform when $\|v\|_{H^{m-1}}$ is bounded and $\ka$
is fixed.  For $m=1$, the infinitesimal $D$-bound of multiplication by
$v\in L^2$ identifies the form realization with domain $H_+^1$.
For higher integers, the equation
$(D+\ka)f=A_vf+\Pi(vf)$ and the one-dimensional product estimate
bootstrap the domain in both directions.  The same bounds control the
graph norms uniformly, and interpolation gives all $r\in[0,m]$.

Start with $L^2$ conservation.  For $u_0\in H^k$ and an integer $k\ge1$,
the already established uniform $H^{k-1}$ bound,
\eqref{eq:CR-higher-domain} with $m=k$, and
\eqref{eq:CR-higher-intertwining} give
$\sup_t\|u(t)\|_{H^k}\le C\|A_0^k(u_0)_+\|_2<\infty$.
Induction provides all integer bounds up to $\lfloor {r_*}\rfloor$.
Taking $m=\lfloor {r_*}\rfloor+1>{r_*}$ in
\eqref{eq:CR-higher-domain} now yields, uniformly in $t$,
\[
\|f\|_{H^r}\lesssim\|A_t^rf\|_2\lesssim\|f\|_{H^r}\enspace(0\le r\le m),\enspace
 \sup_t\| U_{u_0}^{\mathrm{cyc}}(t)h\|_{H^{r_*}}
 \le C\|A_0^{r_*} h\|_2.
\]
Applying the spectral theorem to $A_0^{r_*} (u_0)_+$ and using the
$L^2$ distorted Plancherel formula gives
\begin{equation}
 \|A_0^{r_*} (u_0)_+\|_2^2
 =2\pi\sum_j a_j(\ka-a_j)^{2{r_*}}
  +\int_0^\infty(\lambda+\ka)^{2{r_*}}
       |g_{u_0}^{\pm}(\lambda)|^2\,\rmd\lambda<\infty.
 \label{eq:CR-higher-spectral-moment}
\end{equation}
Consequently the same radiation profiles constructed at exponent $s_0$
belong to $H^{r_*}$.  Moreover, for any Borel set $B$,
\[
 \sup_t\| U_{u_0}^{\mathrm{cyc}}(t)E_{u_0}(B)(u_0)_+\|_{H^{r_*}}
 \le C\|E_{u_0}(B)A_0^{r_*} (u_0)_+\|_2.
\]
Thus the threshold and high-energy estimates, and the analogous free
estimates, hold with $s$ replaced by ${r_*}$; both tails vanish by
\eqref{eq:CR-higher-spectral-moment} and $E_{u_0}(\{0\})=0$.

For a fixed isolated atom or a fixed compact positive spectral interval,
\eqref{eq:CR-higher-intertwining}--\eqref{eq:CR-higher-domain} give a
uniform $H^m$ bound on the evolving spectral vector.  The soliton and
free compact spectral limits also have uniform $H^m$ bounds, by smoothness
and compact Fourier support, respectively.  The $L^2$ convergence in
Lemma~\ref{lem:CR-one-atom} and in the compact-window scattering argument
(including the sharp-cutoff approximation) therefore implies convergence
in $H^{r_*}$, by
$\|F\|_{H^{r_*}}\le\|F\|_2^{1-{r_*}/m}\|F\|_{H^m}^{{r_*}/m}$.
Applying the exact decomposition \eqref{eq:CR-four-block-identity}, first
letting $t\to\sigma\infty$, then $R\to\infty$, and finally
$\delta\downarrow0$, proves \eqref{eq:resolution-infinite} in
$H^{r_*}$.  If the negative spectrum is finite and nonempty, keep $\delta$
below its smallest absolute eigenvalue and still let $\delta\downarrow0$
to remove the positive threshold tail; this gives \eqref{eq:resolution-finite}.
The same argument applies to the empty negative spectrum.  Both time
directions use the same spectral parameters.

\section{Asymptotic profiles, orthogonality, and norm splitting}
\label{sec:germ}

Throughout this section, \eqref{eq:NS-assumptions} holds with $-\frac12<s<0$.

\subsection{Orthogonality and the space of asymptotic profiles}

\begin{lemma}[Sobolev norms of solitons and asymptotic orthogonality]
	\label{lem:germ-profile-orthogonality}
	If $a>0$, $c\in\mathbb R$, and $p=-c+i/(2a)$, then
	\begin{align}
		\widehat{R_p}(\xi)
		&=\sqrt{2\pi}e^{-|\xi|/(2a)}e^{-ic\xi},                 \label{eq:germ-soliton-Fourier}\\
		\|R_p\|_{H^s}^2
		&=4\pi a\int_0^\infty(1+a^2y^2)^se^{-y}\,\rmd y
		=\mathsf M_s(a)\le4\pi a.                              \label{eq:germ-soliton-energy}
	\end{align}
	For $j\neq k$ and $f\in H^s$,
	\begin{align}
		\langle R_{p_j}(\cdot-2a_jt),
		R_{p_k}(\cdot-2a_kt)\rangle_{H^s}&\longrightarrow0,
		\label{eq:germ-soliton-soliton}\\
		\langle R_{p_j}(\cdot-2a_jt),W(t)f\rangle_{H^s}
		&\longrightarrow0                                           \label{eq:germ-soliton-radiation}
	\end{align}
	as $t\to\pm\infty$.
\end{lemma}

\begin{proof}[Proof outline]
The Fourier transform of the Cauchy profile gives
\eqref{eq:germ-soliton-Fourier}; integration with the Sobolev weight gives
\eqref{eq:germ-soliton-energy}.  The soliton cross terms have integrable
amplitudes and nonzero linear phases because the eigenvalues are simple.
For the radiation cross term, Cauchy--Schwarz gives an integrable amplitude,
and the phase $\phi_a(\xi)=\xi|\xi|+2a\xi$ has
$\phi_a'(\xi)=2(|\xi|+a)>0$.  A change of variables in the latter case
and the Riemann--Lebesgue lemma prove both orthogonality limits.
\end{proof}

By \eqref{eq:germ-soliton-energy} and \eqref{eq:negative-spectrum-sum},
\begin{equation}
	\sum_j\mathsf M_s(a_j)\le4\pi\sum_ja_j<\infty.
	\label{eq:germ-energy-summability}
\end{equation}

Put $\mathfrak I_+=[1,\infty)$ and $\mathfrak I_-=(-\infty,-1]$.
Let $C_b(\mathfrak I_\sigma;H^s)$ be the bounded continuous $H^s$-valued
paths and let $C_{0,\sigma}$ be those tending to zero as
$t\to\sigma\infty$.  Define the Banach quotient
$\mathcal G_\sigma^s=C_b(\mathfrak I_\sigma;H^s)/C_{0,\sigma}$.

\begin{lemma}[Exact quotient norm]
	\label{lem:germ-quotient-norm}
	For $F\in C_b(\mathfrak I_\sigma;H^s)$,
	\[
\|[F]\|_{\mathcal G_\sigma^s}
		=\limsup_{t\to\sigma\infty}\|F(t)\|_{H^s}.
\]
\end{lemma}

\begin{proof}[Proof outline]
Adding an element of $C_{0,\sigma}$ preserves the limsup and gives the
lower bound.  For the reverse bound, cancel $F$ on an initial segment
using a continuous cutoff, and bound the remaining tail by the limsup
plus an arbitrary $\varepsilon>0$.
\end{proof}

In the remainder of this subsection, $H^s$ and $\mathcal G_\sigma^s$ denote
their complexifications.  The physical BO profiles form the real subspace.

For finitely supported $z=(z_j)$ define
\begin{equation}
	\mathcal J_{\sigma,0}(f,z)
	=\left[W(t)f+\sum_jz_jR_{p_j}(\,\cdot-2a_jt\,)\right].
	\label{eq:germ-profile-map-finite}
\end{equation}

\begin{prop}[The Hilbert space of asymptotic profiles]
	\label{prop:germ-profile-isometry}
	The map \eqref{eq:germ-profile-map-finite} extends uniquely to an isometric
	embedding with closed range
	\[
\mathcal J_\sigma\in
		\mathcal B\bigl(H^s(\R)\oplus\ell^2(\mathsf M_s),
		\mathcal G_\sigma^s\bigr),
		\enspace
		\ell^2(\mathsf M_s)
		=\left\{z:\sum_j\mathsf M_s(a_j)|z_j|^2<\infty\right\},
\]
	and
	\begin{equation}
		\|\mathcal J_\sigma(f,z)\|_{\mathcal G_\sigma^s}^2
		=\|f\|_{H^s}^2+\sum_j\mathsf M_s(a_j)|z_j|^2.
		\label{eq:germ-profile-isometry}
	\end{equation}
\end{prop}

\begin{proof}[Proof outline]
For finite $z$, asymptotic orthogonality removes all cross terms.
Lemma~\ref{lem:germ-quotient-norm} then gives
\eqref{eq:germ-profile-isometry}.  Completion extends this isometry
uniquely and makes its range closed.
\end{proof}

By \eqref{eq:germ-energy-summability}, ${\bf1}\in\ell^2(\mathsf M_s)$,
so $\mathbf S^\sigma=\mathcal J_\sigma(0,{\bf1})$ is defined in
$\mathcal G_\sigma^s$.  For
$S_\delta(t)=\sum_{a_j\ge\delta}R_{p_j}(\cdot-2a_jt)$, the isometry gives
\begin{equation}
	\|\mathbf S^\sigma-[S_\delta]\|_{\mathcal G_\sigma^s}^2
	=\sum_{a_j<\delta}\mathsf M_s(a_j)\longrightarrow0.
	\label{eq:germ-threshold-tail}
\end{equation}

\subsection{Proof of Theorem~\ref{thm:main-resolution}: global weighted soliton resolution}

\begin{proof}[Outline of the proof of Theorem~\ref{thm:main-resolution}]
Subsection~\ref{subsec:CR-nonnegative} treats $s\ge0$.  For
$-\frac12<s<0$, Proposition~\ref{prop:negative-spectrum} provides the
spectral data and \eqref{eq:germ-energy-summability} controls the soliton
tail.  Apply Proposition~\ref{prop:CR-finite-cutoff} at fixed
$\delta>0$, take $t\to\sigma\infty$, and then let $\delta\downarrow0$.
This gives \eqref{eq:resolution-finite} and
\eqref{eq:resolution-infinite}.  In the finite-spectrum case,
Corollary~\ref{coro:CR-finite-spectrum-error-modulus} gives the asserted
error bound with
$f_{u_0,N}(T)=\mathcal R_{u_0,N}(T)+e^{-T}$.

For norm splitting, first keep $\delta$ fixed.  Orthogonality gives
\[
\|W(t)u_\infty^\sigma+S_\delta(t)\|_{H^s}^2
\longrightarrow
\|u_\infty^\sigma\|_{H^s}^2+
\sum_{a_j\ge\delta}\mathsf M_s(a_j).
\]
The finite-cutoff approximation and the reverse triangle inequality
transfer this limit to $u(t)$ with an error tending to zero as
$\delta\downarrow0$.  Summability then yields the full norm identity.
The relation $g_{u_0}^-=\ell_{u_0}g_{u_0}^+$, with
$|\ell_{u_0}|=1$, gives equality of the two radiation norms.

Finally, the quotient norm and \eqref{eq:germ-threshold-tail} allow the
same ordered limits to be expressed as
\begin{equation}
[u-W(\cdot)u_\infty^\sigma]=\mathbf S^\sigma
\enspace\hbox{in }\mathcal G_\sigma^s.
\label{eq:main-germ-resolution}
\end{equation}
\end{proof}

\begin{prop}[Exact obstruction to convergence toward a fixed finite soliton sum]
	\label{prop:germ-no-fixed-finite-sum}
	For a finite set $F\subset\mathcal J$, put
	$S_F(t)=\sum_{j\in F}R_{p_j}(\,\cdot-2a_jt\,)$ and
	$\mathcal E_F^\sigma(t)=u(t)-W(t)u_\infty^\sigma-S_F(t)$.
	Then
	\begin{equation}
		\limsup_{t\to\sigma\infty}
		\|\mathcal E_F^\sigma(t)\|_{H^s}^2
		=\sum_{j\in\mathcal J\setminus F}\mathsf M_s(a_j).
		\label{eq:germ-fixed-finite-exact-tail}
	\end{equation}
	Consequently, if the negative spectrum is infinite, no fixed finite sum
	drawn from the spectral soliton family has an $H^s$ remainder tending to
	zero.
\end{prop}

\begin{proof}[Proof outline]
Equation~\eqref{eq:main-germ-resolution} identifies the remainder class
with $\mathcal J_\sigma(0,\mathbf1_{\mathcal J\setminus F})$.
The profile isometry and the exact quotient norm give the stated tail.
\end{proof}

More generally, for finite sets $F_N\subset\mathcal J$, the corresponding
finite sums converge to $\mathbf S^\sigma$ in $\mathcal G_\sigma^s$
precisely when
$\sum_{j\notin F_N}\mathsf M_s(a_j)\to0$.  Every increasing finite
exhaustion $F_N\uparrow\mathcal J$ satisfies this criterion by
\eqref{eq:germ-energy-summability}.

\section{Sharpness of the spatial weight}
\label{sec:weight-sharpness}

Write $u(t;v)$ for the BO solution with initial datum $v$ and
$W(t)=e^{t\partial_x|D|}$ for the free group.  For real
$v\in C_c^\infty(\mathbb R)$ with negative eigenvalues
$-a_1,\ldots,-a_N$, let $F(v)$ be its future radiation profile.
The radiation and Plancherel identities \cite[(2.6)--(2.7)]{GGM},
together with the Wu relation \cite[(1.15)]{BKV}, give
\begin{align}
 \widehat{F(v)}(\lambda)
 &=\mathcal F_v^-(\Pi v)(\lambda),\enspace
 \widehat{F(v)}(-\lambda)=\overline{\widehat{F(v)}(\lambda)}
 \enspace(\lambda>0),\label{EP-radiation}\\
 \|F(v)\|_2^2&=\|v\|_2^2-4\pi\sum_{j=1}^N a_j.
 \label{EP-mass}
\end{align}
The factor $4\pi$ follows from
$\|\Pi v\|_2^2=\frac12\|v\|_2^2$ and
$|\langle f_j,\Pi v\rangle|^2=2\pi a_j$ for normalized eigenvectors.
Theorems~1.1 and~2.1 of~\cite{GGM}, recalled in
Proposition~\ref{GGM26}, identify \eqref{EP-radiation} as the radiation
in the global $H^1$ resolution of $u(t;v)$; their decay assumptions
hold for compactly supported smooth data.

\subsection{Small perturbations with a fixed radiation phase}

Fix $\psi\in C_c^\infty((-1,1))$, $\psi\ge0$, $\int\psi=1$,
and $\theta<0$.  Set
\[
L_{n,k}=4^{n+k},\enspace
 b_n(x)=\frac{\theta}{n}\sum_{k=1}^n L_{n,k}^{-1}
 \psi\!\left(\frac{x+3L_{n,k}}{L_{n,k}}\right).
\]
The packet supports lie in the disjoint intervals $[-4L_{n,k},-2L_{n,k}]$, and
\begin{gather}
 b_n\le0,\enspace \int b_n=\theta,\enspace \|b_n\|_1=|\theta|,
 \enspace \operatorname{dist}(\operatorname{supp}b_n,0)\longrightarrow\infty,\notag\\
 \|\langle x\rangle^{1/2}b_n\|_{H^m}
 \le C_{m,\psi}|\theta|n^{-1/2},\enspace
 \|b_n\|_2\le\frac{C_\psi|\theta|}{n2^n},\enspace
 \|b_n\|_\infty\le\frac{C_\psi|\theta|}{n4^n}.
 \label{EP-packet-smallness}
\end{gather}
The weighted estimate follows by summing the squared block bounds over
the disjoint supports; the unweighted estimates use the geometric
separation of the scales.

Let $\mathcal T_\theta$ be the real $L^2$ isometry whose Fourier
multiplier is $e^{-i\theta}$ on positive frequencies and $e^{i\theta}$
on negative frequencies.

\begin{lemma}[Radiation phase]
\label{EP-phase-lemma}
For real $w\in C_c^\infty(\mathbb R)$ with finitely many negative
eigenvalues, $w+b_n$ has the same number of negative eigenvalues for
all sufficiently large $n$, and
\begin{equation}
 F(w+b_n)\longrightarrow\mathcal T_\theta F(w)
 \enspace\hbox{in }L^2(\mathbb R).
 \label{EP-phase-limit}
\end{equation}
\end{lemma}
\begin{proof}[Proof outline]
Fix $J\Subset(0,\infty)$.  The left-normalized Jost equation is
$m_{-,v}=e^{i\lambda x}+(D-\lambda-i0)^{-1}\Pi(vm_{-,v})$.
Split its kernel into the one-sided oscillatory kernel and a remainder
$K_\lambda$, for which
\[
\|K_\lambda*g\|_\infty\le C_J\|g\|_2,
\qquad |K_\lambda(z)|\le C_J|z|^{-1}\quad(|z|\ge1).
\]
On the packet support, the Volterra part has a uniformly bounded inverse
and accumulates the phase $\exp(i\int b_n)=e^{i\theta}$.
The remainder is small by \eqref{EP-packet-smallness}.  If $I_0$ contains
$\operatorname{supp}w$ and $d_n$ is its distance from the packets, this gives
an $O(\|b_n\|_2+d_n^{-1})$ error when the wave reaches $I_0$.

To include the core, restrict the Jost equation to the core and packet
supports and eliminate the packet component.  The return coupling from
the core has size $O(d_n^{-1})$, so the resulting core equation is a
small perturbation of that for $w$.  Its inverse is uniform on $J$:
Fredholm theory reduces this to excluding a homogeneous solution, whose
resolvent pole cancellation would produce a positive eigenfunction,
contrary to \cite[Theorem~1.4]{BKV}.  Consequently
\[
\sup_{\lambda\in J}
\|m_{-,w+b_n}-e^{i\theta}m_{-,w}\|_{L^\infty(I_0)}
\longrightarrow0.
\]
The source pairing
$\mathcal F_v^-(\Pi w)(\lambda)
=(2\pi)^{-1/2}\int w\overline{m_{-,v}}$
follows by Hardy frequency cutoff approximation.  It converts the core
limit into convergence of the radiation transforms on $J$; the packet
source has spectral $L^2$ norm at most $\|\Pi b_n\|_2\to0$.
Exhausting the positive axis gives weak convergence in
\eqref{EP-phase-limit}.

Finally,
$0\le L_{w+b_n}-L_w\le\|b_n\|_\infty I$.
The min--max principle preserves exactly the original finitely many
negative eigenvalues for large $n$, and their values converge.
Equation~\eqref{EP-mass} gives convergence of radiation norms, upgrading
the weak limit to strong $L^2$ convergence.
\end{proof}

\subsection{A complete family of initial data}

Let $\mathscr E$ be the completion of real $C_c^\infty(\mathbb R)$ in
the increasing norms
$p_m(v)=\|\langle x\rangle^{1/2}v\|_{H^m},\enspace m=0,1,2,\ldots,$
with complete metric
\[
d(v,w)=\sum_{m\ge0}2^{-m-1}\min\{1,p_m(v-w)\}.
\]
The $p_0$ limit identifies each element uniquely with a real function;
compatibility of the $H^m$ limits gives
$\langle x\rangle^{1/2}v\in\bigcap_{m\ge0}H^m$.
Thus $\mathscr E\hookrightarrow H^m$ continuously for every $m$.
Let $\mathscr C$ be the closure in $\mathscr E$ of the nonnegative
functions in real $C_c^\infty$.  It is a closed convex cone of
nonnegative functions, since convergence in $p_1$ implies uniform
convergence.  By \eqref{EP-packet-smallness}, $b_n\to0$ in $\mathscr E$.

\begin{lemma}
\label{EP-base-family}
There are $w_*\in C_c^\infty(\mathbb R)$, $w_*\ge0$, and
$\varepsilon_*>0$ such that
\begin{align*}
 \mathscr X&=\{w_*-v:v\in\mathscr C,\ \|v\|_2\le\varepsilon_*\},\\
 \mathscr D&=\{w_*-v:v\in C_c^\infty,\ v\ge0,
                       \ \|v\|_2<\varepsilon_*\}
\end{align*}
satisfy the following properties.  The space $\mathscr X$ is nonempty
and complete, and $\mathscr D$ is dense in it.  Every $u\in\mathscr X$
has exactly one simple negative eigenvalue $-a(u)$, with
\begin{equation}
 a_*/2\le a(u)\le a_*,\enspace
 \|u\|_2^2-4\pi a(u)\ge\eta>0,
 \label{EP-family-bounds}
\end{equation}
where $-a_*$ is the negative eigenvalue of $L_{w_*}$.
For each $w\in\mathscr D$, $w+b_n\in\mathscr D$ for all large $n$.
\end{lemma}
\begin{proof}[Proof outline]
Start with $Q_b(x)=2b/(x^2+b^2)$.  Its Hardy space decomposition using
$B(x)=(x-ib)/(x+ib)$ gives one negative eigenvalue, $-(2b)^{-1}$,
with the complementary restriction unitarily equivalent to $D$.
A sufficiently large smooth compact cutoff $w_*$ of $Q_b$ still has
exactly one negative eigenvalue: operator monotonicity bounds their
number, and the original eigenfunction remains a negative trial vector.

For a normalized Hardy function the estimate
$\|f\|_4^4\le\langle f,Df\rangle/\pi$ gives
$4\pi a_*\le\|w_*\|_2^2$.  Equality would force the eigenfunction to
have compact support, contradicting its Hardy support condition.
Thus $r_*:=\|w_*\|_2^2-4\pi a_*>0$.
For $u=w_*-v$, positivity of $T_v$ prevents additional negative
eigenvalues, while the original trial vector keeps one below
$-a_*/2$ when $\|v\|_2$ is small.  Reducing $\varepsilon_*$ also
preserves a defect of at least $r_*/2$, giving
\eqref{EP-family-bounds} and simplicity by \cite[Theorem~1.4]{BKV}.

Closedness gives completeness of $\mathscr X$.  Shrinking $v$ slightly
before approximation by nonnegative compactly supported smooth functions
gives density of $\mathscr D$.  Since $-b_n\ge0$ and $\|b_n\|_2\to0$,
the perturbations preserve $\mathscr D$ for all large $n$.
\end{proof}

Theorem~\ref{thm:main-no-singular} also gives purely absolutely
continuous positive spectrum of multiplicity one for every member
of $\mathscr X$.

\subsection{Proof of Theorem~\ref{thm:main-weight-sharpness}: sharpness of the spatial weight}

\begin{proof}[Outline of the proof of Theorem~\ref{thm:main-weight-sharpness}]
Let $\chi_-$ be a smooth cutoff, equal to one for $x\le-1$ and
zero for $x\ge0$.  Set
$B_t(v)=W(-t)[\chi_-u(t;v)]$.  For fixed $t$ this is continuous
$\mathscr X\to L^2$ by \cite[Theorem~1.1]{KLV}.
Nonstationary phase, followed by density, gives
$W(-t)\chi_-W(t)f\to f$ in $H^s$; the translated solitons tend to
zero after multiplication by $\chi_-$.
Thus any resolution \eqref{EP-failure} forces convergence of $B_t$
to its radiation profile, and the classical resolution gives
$B_t(w)\to F(w)$ in $L^2$ for $w\in\mathscr D$.

Each orbit from $\mathscr X$ is uniformly bounded in $H^1$.
Indeed, the conserved $L^2$ norm controls the potential relatively to
$D$, so the graph norm of $L_{u(t;v)}+\kappa$ controls $H^1$ uniformly
in time for sufficiently large fixed $\kappa$.
The Lax transport identity in \cite[Proposition~3.9(ii)--(iv)]{BKVasymp}
then bounds this graph norm on $\Pi u(t;v)$.
The same uniform $H^1$ bound holds for $B_t(v)$.

Take $\theta=-\pi$.  Lemma~\ref{EP-phase-lemma} gives
$F(w+b_n)\to-F(w)$, whereas
\eqref{EP-family-bounds} gives $\|F(w)\|_2\ge\sqrt\eta$.
For $0<\delta<\sqrt\eta$ consider the closed sets
\[
E_N=\bigcap_{j,k\ge N}
\{v\in\mathscr X:\|B_j(v)-B_k(v)\|_2\le\delta\}.
\]
If $E_N$ had interior, density would provide $w\in\mathscr D$ there,
and $w+b_n$ would remain in that interior for large $n$.
First sending one time index to infinity and then $n\to\infty$,
continuity of $B_N$ would imply $2\|F(w)\|_2\le2\delta$, a contradiction.
Baire's theorem therefore gives $u_0\in\mathscr X$ for which
$B_j(u_0)$ is not Cauchy in $L^2$.

The definition of $\mathscr X$, Lemma~\ref{EP-base-family}, and
Theorem~\ref{thm:main-no-singular} supply the required regularity and
spectral properties.  A resolution in $H^s$ with $s\ge0$ directly
contradicts this failure of the Cauchy property.  For $-\frac12<s<0$,
the uniform $H^1$ bound puts the distributional limit in $H^1$, and
interpolation upgrades $H^s$ convergence to $L^2$ convergence, giving
the same contradiction.  Multiplication by
$\langle x\rangle^{\beta-1/2}$ finally gives the asserted weighted
regularity for every $\beta\le\frac12$.
\end{proof}

\section{Proof of Theorem \ref{thm3}: trace-saturation characterization of infinite-order multisolitons}
\label{sec:inverse-spectral-reconstruction}

Throughout this section,
$v$ and $u_0$ are real-valued $L^2$ functions and
$L_v=D-\Pi(v\,\cdot)$ is understood through its closed quadratic form.

\begin{lemma}[Sharp trace and its local equality case]
	\label{lem:ISP-sharp-trace}
	Let $\lambda_j<0$ enumerate the complete negative spectrum of $L_v$, with
	normalized eigenfunctions $\psi_j$.  Then
	\begin{equation}
		2\pi\operatorname{Tr}(-L_v)_+
		=2\pi\sum_j|\lambda_j|
		=\sum_j|\langle\psi_j,\Pi v\rangle|^2
		\le \|\Pi v\|_2^2.
		\label{eq:ISP-sharp-trace}
	\end{equation}
	In particular, $(-L_v)_+$ is trace class, and
\begin{equation}
\|\Pi v\|_2^2-2\pi\sum_j|\lambda_j|
=\|E_v([0,\infty))\Pi v\|_2^2.
\label{eq:ISP-continuous-defect}
\end{equation}
If $u_0$ satisfies the
	hypotheses of Theorem~\ref{thm3}, put $a_k=-\lambda_k$ and choose the
	phases of the normalized eigenfunctions $f_k$ so that
	$\langle f_k,\Pi u_0\rangle=\sqrt{2\pi a_k}$.  Then
	\begin{equation}
		\Pi u_0=\sum_{k\ge1}\sqrt{2\pi a_k}\,f_k
		\enspace\hbox{in }L_+^2(\mathbb R),
		\label{eq:ISP-eigen-expansion}
	\end{equation}
	and
	\begin{equation}
		u_0(x)=2\pi\sum_{k\ge1}|f_k(x)|^2
		\enspace\hbox{in }L^1_{\rm loc}(\mathbb R)\hbox{ and for a.e. }x.
		\label{eq:ISP-local-trace}
	\end{equation}
\end{lemma}

\begin{proof}[Proof outline]
The Wu relation in \cite[Theorem~1.4 and Proposition~7.1]{BKV} and
Bessel's inequality give \eqref{eq:ISP-sharp-trace}; subtracting the
negative atomic masses gives \eqref{eq:ISP-continuous-defect}.
Trace saturation therefore puts $\Pi u_0$ entirely in the negative
spectral subspace and yields \eqref{eq:ISP-eigen-expansion}.

For $0\le w\in L_c^\infty$, apply the Ky Fan principle to
$u_0\pm\varepsilon w$ with trial vectors $f_1,\ldots,f_N$, and bound
the resulting trace by \eqref{eq:ISP-sharp-trace}.
The plus sign first bounds the increasing local density sums.
Pass $N\to\infty$ at fixed $\varepsilon>0$, subtract the saturated
identity, and then let $\varepsilon\downarrow0$ for both signs.
Duality gives \eqref{eq:ISP-local-trace}, including $u_0\ge0$ almost
everywhere.
\end{proof}

\begin{proof}[Outline of the proof of Theorem~\ref{thm3}]
Let $P_{\rm d}=E_{u_0}(( -\infty,0))$, $P_{\rm c}=I-P_{\rm d}$,
and $\mathcal H_{\rm d,c}=P_{\rm d,c}L_+^2$.
Lemma~\ref{lem:ISP-sharp-trace} puts $(u_0)_+$ in
$\mathcal H_{\rm d}=\overline{\spn\{f_k\}}$.
Proposition~\ref{prop:negative-spectrum} gives simplicity and excludes
a zero eigenvalue.  The proof proceeds by identifying this spectral
subspace as a Hardy model space and recovering the potential from its
inner function.

For $S_th=e^{itx}h$, $t\ge0$, put
$\mathcal C_t=P_{\rm c}S_t^*P_{\rm d}$ and
$L_{\rm c}=L_{u_0}|_{\mathcal H_{\rm c}}\ge0$.
The domain result \eqref{eq:operator-domain}, followed by one bootstrap,
gives $\operatorname{Dom}(L_{u_0})=H_+^1$ with equivalent graph norm.
Computing the shift commutator on a Fourier core and extending by graph
norm yields
\[
(L_{\rm c}+t+a_k)\mathcal C_tf_k
=c_{\rm F}\int_0^t\widehat f_k(r)\mathcal C_{t-r}(u_0)_+\,\rmd r,
\qquad c_{\rm F}=(2\pi)^{-1/2}.
\]
On a spectral window $\Delta=[\lambda_0,R]\Subset(0,\infty)$,
diagonalize $L_{\rm c}$ by a direct-integral unitary $U$ over a control
measure $\mu$ on $\Delta$, with arbitrary fiber multiplicity, and put
$g(t,\lambda)=[UE_{u_0}(\Delta)\mathcal C_t(u_0)_+](\lambda)$.
First sum the commutator identity over finitely many eigenfunctions,
using \eqref{eq:ISP-eigen-expansion}, and then pass to the limit in
$L^2([0,T]\times\Delta,\rmd t\,\rmd\mu)$ for fixed $T>0$.
This gives a homogeneous Volterra equation
\[
g(t,\lambda)=\int_0^tK(t,\lambda;r)g(t-r,\lambda)\,\rmd r,
\qquad
\|K(t,\lambda;\cdot)\|_2
\le c_{\rm F}\lambda_0^{-1}\|(u_0)_+\|_2.
\]
Here $K=c_{\rm F}\widehat{(A_{\rm d}+\lambda+t)^{-1}(u_0)_+}$,
where $A_{\rm d}=-L_{u_0}|_{\mathcal H_{\rm d}}$.
The kernel bound both justifies the finite-sum limit and makes the
Volterra map contractive in $L^2(\rmd t\,\rmd\mu)$ on sufficiently short
time intervals.  Iteration gives $g=0$ on every finite interval.
Exhausting the positive spectrum and using the absence of a zero
eigenvalue gives $\mathcal C_t(u_0)_+=0$.
Substitution into the eigenfunction identity, strong continuity in $t$,
and density of $\spn\{f_k\}$ in $\mathcal H_{\rm d}$ then give
$\mathcal C_t=0$.  Thus $\mathcal H_{\rm d}$ is invariant under $S_t^*$.

Consequently $\mathcal H_{\rm c}$ is invariant under $S_t$ and under
the Cayley multiplier $(x-i)/(x+i)$.  It is nonzero and proper by the
essential spectrum and the nonempty negative spectrum.
Beurling's theorem \cite{GarciaModel} therefore gives an inner function
$\Theta$ with
$\mathcal H_{\rm c}=\Theta L_+^2$ and
$\mathcal H_{\rm d}=K_\Theta:=L_+^2\ominus\Theta L_+^2$.
Since $(f_k)$ is an orthonormal basis of $K_\Theta$, the model-space
kernel formula and \eqref{eq:ISP-local-trace} give
\[
0\le\frac{1-|\Theta(x+iy)|^2}{2y}
=2\pi\sum_k|f_k(x+iy)|^2\le P_yu_0(x),
\]
where $P_y$ is the Poisson operator.
At almost every boundary point, first let $y\downarrow0$ in finite
partial sums and then let their size tend to infinity.
Poisson convergence and the local trace identity identify the limit
as $u_0(x)$.  Julia--Carath\'eodory theory
\cite{GarciaModel,Garnett} gives
$-i\Theta'(x)/\Theta(x)=u_0(x)$ almost everywhere.

The canonical factorization of $\Theta$ now gives
\[
u_0(x)=\alpha_\infty+
\sum_m\frac{2\operatorname{Im}\zeta_m}{|x-\zeta_m|^2}
+\int_{\mathbb R}\frac{1+t^2}{(x-t)^2}\,\rmd\sigma(t),
\]
where $\alpha_\infty\ge0$, $(\zeta_m)$ are its zeros, and $\sigma$ is a
positive singular measure.  Positivity and $u_0\in L^2$ force
$\alpha_\infty=0$.  A nonzero singular measure would, by Tonelli's
theorem on a bounded interval carrying positive singular mass, force
$u_0$ to fail local integrability.  Hence $\sigma=0$ and $\Theta$ is a
pure Blaschke product with infinitely many zeros, since $K_\Theta$
is infinite-dimensional.

Set $p_m^{\rm in}=-\overline{\zeta_m}$.  The resulting positive pole
sums increase to $u_0$ and converge strongly in $L^2$ by domination.
The finite-sum Cauchy integral identity passes to the limit to give
\[
4\pi\sum_{m,n\ge1}
\frac{\operatorname{Im}p_m^{\rm in}}
{|p_m^{\rm in}-\overline{p_n^{\rm in}}|^2}
=\|\Pi u_0\|_2^2<\infty.
\]
This proves \eqref{eq:ISP-pole-condition}.  Its diagonal part controls
$\sum_m(\operatorname{Im}p_m^{\rm in})^{-1}$, so the Weierstrass test
upgrades the pole representation to absolute uniform convergence.
The continuous representative satisfies \eqref{eq:ISP-pole-representation}.
\end{proof}


\begin{thebibliography}{99}
\bibitem{ABFS} {\small \textsc{Ablowitz, M. J., Fokas, A. S., Musslimani, Z. H.} On a new non-local formulation of water waves, {\it Journal of Fluid Mechanics,} {\bf 562} (2006), 313--343.}

\bibitem{AT91} {\small \textsc{Amick, C. J., Toland, J. F.} Uniqueness and related analytic properties for the Benjamin--Ono equation---a nonlinear Neumann problem in the plane, {\it Acta Mathematica,} {\bf 167} (1991), 107--126.}

\bibitem{BKV} {\small \textsc{Badreddine, R., Killip, R., Vi\c{s}an, M.} Orbital stability of Benjamin--Ono multisolitons, {\it arXiv preprint arXiv:2509.14153v2,} (2025, revised 2026).}

\bibitem{BKVasymp} {\small \textsc{Badreddine, R., Killip, R., Vi\c{s}an, M.} Asymptotic stability of Benjamin--Ono multisolitons in $L^2(\mathbb{R})$, {\it arXiv preprint arXiv:2606.12951,} (2026).}

\bibitem{BHnew} {\small \textsc{Behzadan, A., Holst, M.} Multiplication in Sobolev spaces, revisited, {\it arXiv preprint arXiv:1512.07379v3,} (2021).}

\bibitem{Benjamin} {\small \textsc{Benjamin, T. B.} Internal waves of permanent form in fluids of great depth, {\it Journal of Fluid Mechanics,} {\bf 29} (1967), 559--592.}

\bibitem{BGGM} {\small \textsc{Blackstone, E., Gassot, L., G\'erard, P., Miller, P. D.} The Benjamin--Ono initial-value problem for rational data with application to long-time asymptotics and scattering, {\it Annales de l'Institut Henri Poincar\'e C, Analyse non lin\'eaire,} (2025).}

\bibitem{BGGM24} {\small \textsc{Blackstone, E., Gassot, L., G\'erard, P., Miller, P. D.} The Benjamin--Ono equation in the zero-dispersion limit for rational initial data: generation of dispersive shock waves, {\it Communications on Pure and Applied Mathematics,} (2024), e70044.}

\bibitem{BJM18} {\small \textsc{Borghese, M., Jenkins, R., McLaughlin, K. D. T.-R.} Long time asymptotic behavior of the focusing nonlinear Schr\"odinger equation, {\it Annales de l'Institut Henri Poincar\'e C, Analyse non lin\'eaire,} {\bf 35} (2018), 887--920.}

\bibitem{BP08} {\small \textsc{Burq, N., Planchon, F.} On well-posedness for the Benjamin--Ono equation, {\it Mathematische Annalen,} {\bf 340} (2008), 497--542.}

\bibitem{CLW23} {\small \textsc{Charlier, C., Lenells, J., Wang, D. S.} The ``good'' Boussinesq equation: long-time asymptotics, {\it Analysis \& PDE,} {\bf 16} (2023), 1351--1388.}

\bibitem{CL21} {\small \textsc{Chen, G., Liu, J.} Soliton resolution for the focusing modified KdV equation, {\it Annales de l'Institut Henri Poincar\'e C, Analyse non lin\'eaire,} {\bf 38} (2021), 2005--2071.}

\bibitem{Chen} {\small \textsc{Chen, X.} Explicit formula for the Benjamin--Ono equation with square integrable and real valued initial data and applications to the zero dispersion limit, {\it Pure and Applied Analysis,} {\bf 7} (2025), 101--126.}

\bibitem{Che25} {\small \textsc{Chen, X.} Scattering of the defocusing Calogero--Moser derivative nonlinear Schr\"odinger equation, {\it arXiv preprint arXiv:2511.06432,} (2025).}

\bibitem{CJ16} {\small \textsc{Cuccagna, S., Jenkins, R.} On the asymptotic stability of $N$-soliton solutions of the defocusing nonlinear Schr\"odinger equation, {\it Communications in Mathematical Physics,} {\bf 343} (2016), 921--969.}

\bibitem{DavisAcrivos} {\small \textsc{Davis, R. E., Acrivos, A.} Solitary internal waves in deep water, {\it Journal of Fluid Mechanics,} {\bf 29} (1967), 593--607.}

\bibitem{DP11} {\small \textsc{Deift, P. A., Park, J.} Long-time asymptotics for solutions of the NLS equation with a delta potential and even initial data, {\it International Mathematics Research Notices,} {\bf 2011} (2011), 5505--5624.}

\bibitem{DVZ94} {\small \textsc{Deift, P. A., Venakides, S., Zhou, X.} The collisionless shock region for the long-time behavior of solutions of the KdV equation, {\it Communications on Pure and Applied Mathematics,} {\bf 47} (1994), 199--206.}

\bibitem{DZ93} {\small \textsc{Deift, P. A., Zhou, X.} A steepest descent method for oscillatory Riemann--Hilbert problems. Asymptotics for the MKdV equation, {\it Annals of Mathematics,} {\bf 137} (1993), 295--368.}

\bibitem{DKMM22} {\small \textsc{Duyckaerts, T., Kenig, C., Martel, Y., Merle, F.} Soliton resolution for critical co-rotational wave maps and radial cubic wave equation, {\it Communications in Mathematical Physics,} {\bf 391} (2022), 779--871.}

\bibitem{DKM23} {\small \textsc{Duyckaerts, T., Kenig, C., Merle, F.} Soliton resolution for the radial critical wave equation in all odd space dimensions, {\it Acta Mathematica,} {\bf 230} (2023), 1--92.}

\bibitem{FrankRead25} {\small \textsc{Frank, R. L., Read, L.} Jost solutions and direct scattering for the continuum Calogero-Moser equation, {\it arXiv preprint arXiv:2510.11403,} (2025).}


\bibitem{FLMP} {\small \textsc{Freire, R., Linares, F., Mu\~{n}oz, C., Ponce, G.} On local energy decay for solutions of the Benjamin--Ono equation, {\it Annales de l'Institut Henri Poincar\'e C, Analyse non lin\'eaire,} {\bf 41} (2024), no.~3, 725--747.}

\bibitem{FP11} {\small \textsc{Fonseca, G., Ponce, G.} The IVP for the Benjamin--Ono equation in weighted Sobolev spaces, {\it Journal of Functional Analysis,} {\bf 260} (2011), 436--459.}

\bibitem{FLP12} {\small \textsc{Fonseca, G., Linares, F., Ponce, G.} The IVP for the Benjamin--Ono equation in weighted Sobolev spaces II, {\it Journal of Functional Analysis,} {\bf 262} (2012), 2031--2049.}

\bibitem{GG} {\small \textsc{Gassot, L., G\'erard, P.} Infinite-order multisoliton solutions to the Benjamin--Ono equation and soliton resolution, {\it arXiv preprint arXiv:2603.15419v2,} (2026).}

\bibitem{GGM} {\small \textsc{Gassot, L., G\'erard, P., Miller, P. D.} A proof of the soliton resolution conjecture for the Benjamin--Ono equation, {\it arXiv preprint arXiv:2601.10488v2,} (2026).}

\bibitem{GarciaModel} {\small \textsc{Garcia, S. R., Mashreghi, J., Ross, W. T.} {\it Introduction to Model Spaces and Their Operators}, Cambridge University Press, Cambridge, (2016).}

\bibitem{Garnett} {\small \textsc{Garnett, J. B.} {\it Bounded Analytic Functions}, revised first edition, Springer, New York, (2007).}

\bibitem{GGM2} {\small \textsc{Gassot, L., G\'erard, P., Miller, P. D.} The Benjamin-Ono Equation in the Long-Time Limit: Linearized Self-Similar Universality, {\it arXiv preprint arXiv:2606.02989,} (2026).}

\bibitem{GassotLaurens25} {\small \textsc{Gassot, L., Laurens, T.} Global well-posedness for the ILW equation in $ H^ s (\mathbb {T}) $ for $ s>-\frac12$, {\it arXiv preprint arXiv:2506.05149,} (2025).}

\bibitem{Gerard} {\small \textsc{G\'erard, P.} An explicit formula for the Benjamin--Ono equation, {\it Tunisian Journal of Mathematics,} {\bf 5} (2023), 593--603.}

\bibitem{Ger26} {\small \textsc{G\'erard, P.} Lectures on integrable equations of Benjamin--Ono type, {\it EMS Surveys in Mathematical Sciences,} (2026).}

\bibitem{GK21} {\small \textsc{G\'erard, P., Kappeler, T.} On the integrability of the Benjamin--Ono equation on the torus, {\it Communications on Pure and Applied Mathematics,} {\bf 74} (2021), 1685--1747.}

\bibitem{GKT} {\small \textsc{G\'erard, P., Kappeler, T., Topalov, P.} Sharp well-posedness results of the Benjamin--Ono equation in $H^s(\mathbb{T}, \mathbb{R})$ and qualitative properties of its solutions, {\it Acta Mathematica,} {\bf 231} (2023), 31--88.}

\bibitem{GKT2} {\small \textsc{G\'erard, P., Kappeler, T., Topalov, P.} On the spectrum of the Lax operator of the Benjamin-Ono equation on the torus, {\it Journal of Functional Analysis,} {\bf 279} (2020), 108762.}

\bibitem{GL24} {\small \textsc{G\'erard, P., Lenzmann, E.} The Calogero--Moser derivative nonlinear Schr\"odinger equation, {\it Communications on Pure and Applied Mathematics,} {\bf 77} (2024), 4008--4062.}

\bibitem{GL26} {\small \textsc{G\'erard, P., Lenzmann, E.} Global well-posedness and soliton resolution for the half-wave maps equation with rational data, {\it Forum Math. Sigma,} {\bf 13} (2025), 190.}

\bibitem{GP23} {\small \textsc{G\'erard, P., Pushnitski, A.} Unbounded Hankel operators and the flow of the cubic Szeg\H{o} equation, {\it Inventiones mathematicae,} {\bf 232} (2023), 995--1026.}

\bibitem{GP24} {\small \textsc{G\'erard, P., Pushnitski, A.} The cubic Szeg\H{o} equation on the real line: explicit formula and well-posedness on the Hardy class, {\it Communications in Mathematical Physics,} {\bf 405} (2024), 167.}



\bibitem{GT23} {\small \textsc{G\'erard, P., Topalov, P.} On the low regularity phase space of the Benjamin--Ono equation, {\it arXiv preprint arXiv:2308.07829,} (2023).}

\bibitem{HKO96} {\small \textsc{Hayashi, N., Kato, K., Ozawa, T.} Dilation method and smoothing effects of solutions to the Benjamin--Ono equation, {\it Proceedings of the Royal Society of Edinburgh Section A: Mathematics,} {\bf 126} (1996), 273--285.}

\bibitem{HM06} {\small \textsc{Helfrich, K. R., Melville, W. K.} Long nonlinear internal waves, {\it Annual Review of Fluid Mechanics,} {\bf 38} (2006), 395--425.}

\bibitem{HG27} {\small \textsc{Hildebrandt, T. H., Graves, L. M.} Implicit functions and their differentials in general analysis, {\it Transactions of the American Mathematical Society,} {\bf 29} (1927), 127--153.}

\bibitem{IfrimTataru} {\small \textsc{Ifrim, M., Tataru, D.} Well-posedness and dispersive decay of small data solutions for the Benjamin--Ono equation, {\it Annales Scientifiques de l'\'Ecole Normale Sup\'erieure,} {\bf 52} (2019), 297--335.}

\bibitem{IonescuKenig} {\small \textsc{Ionescu, A. D., Kenig, C. E.} Global well-posedness of the Benjamin--Ono equation in low-regularity spaces, {\it Journal of the American Mathematical Society,} {\bf 20} (2007), 753--798.}

\bibitem{ILP} {\small \textsc{Isaza, P., Linares, F., Ponce G.} On the propagation of regularities in solutions of the Benjamin--Ono equation, {\it Journal of Functional Analysis,} {\bf 270} (2016), 976--1000.}

\bibitem{JL23} {\small \textsc{Jendrej, J., Lawrie, A.} Soliton resolution for the energy-critical nonlinear wave equation in the radial case, {\it Annals of PDE,} {\bf 9} (2023).}

\bibitem{JL25} {\small \textsc{Jendrej, J., Lawrie, A.} Soliton resolution for energy-critical wave maps in the equivariant case, {\it Journal of the American Mathematical Society,} {\bf 38} (2025), 783--875.}

\bibitem{JLPS18} {\small \textsc{Jenkins, R., Liu, J., Perry, P., Sulem, C.} Soliton resolution for the derivative nonlinear Schr\"odinger equation, {\it Communications in Mathematical Physics,} {\bf 363} (2018), 1003--1049.}

\bibitem{KatoBook} {\small \textsc{Kato, T.} Perturbation Theory for Linear Operators, Springer, Berlin, (1995).}

\bibitem{Kato1957} {\small \textsc{Kato, T.} Perturbation of continuous spectra by trace class operators, {\it Proceedings of the Japan Academy,} {\bf 33} (1957), 260--264.}

\bibitem{KLM98} {\small \textsc{Kaup, D. J., Lakoba, T. I., Matsuno, Y.} Complete integrability of the Benjamin--Ono equation by means of action-angle variables, {\it Physics Letters A,} {\bf 238} (1998), 123--133.}

\bibitem{KenigMartel} {\small \textsc{Kenig, C. E., Martel, Y.} Asymptotic stability of solitons for the Benjamin--Ono equation, {\it Revista Matem\'atica Iberoamericana,} {\bf 25} (2009), 909--970.}

\bibitem{KLV} {\small \textsc{Killip, R., Laurens, T., Vi\c{s}an, M.} Sharp well-posedness for the Benjamin--Ono equation, {\it Inventiones Mathematicae,} {\bf 236} (2024), 999--1054.}

\bibitem{KLV25} {\small \textsc{Killip, R., Laurens, T., Vi\c{s}an, M.} Scaling-critical well-posedness for continuum Calogero--Moser models on the line, {\it Communications of the American Mathematical Society,} {\bf 5} (2025), 284--320.}

\bibitem{KimKimKwon24} {\small \textsc{Kim, K., Kim, T., Kwon, S.} Construction of smooth chiral finite-time blow-up solutions to Calogero--Moser derivative nonlinear Schr\"odinger equation, {\it arXiv preprint arXiv:2404.09603,} (2024).}

\bibitem{KK24} {\small \textsc{Kim, T., Kwon, S.} Soliton resolution for Calogero--Moser derivative nonlinear Schr\"odinger equation, {\it arXiv preprint arXiv:2408.12843,} (2024).}

\bibitem{KS21} {\small \textsc{Klein, C., Saut, J.-C.} Nonlinear dispersive equations, Springer, Cham, (2021).}

\bibitem{KochTzvetkov} {\small \textsc{Koch, H., Tzvetkov, N.} On the local well-posedness of the Benjamin--Ono equation in $H^s(\mathbb{R})$, {\it International Mathematics Research Notices,} {\bf 2003} (2003), 1449--1464.}

\bibitem{KT05} {\small \textsc{Koch, H., Tzvetkov, N.} Nonlinear wave interactions for the Benjamin--Ono equation, {\it International Mathematics Research Notices,} {\bf 2005} (2005), 1833--1847.}

\bibitem{LTY22a} {\small \textsc{Li, Z. Q., Tian, S. F., Yang, J. J.} On the soliton resolution and the asymptotic stability of $N$-soliton solution for the Wadati--Konno--Ichikawa equation with finite density initial data in space-time solitonic regions, {\it Advances in Mathematics,} {\bf 409} (2022), 108639.}

\bibitem{LTY22b} {\small \textsc{Li, Z. Q., Tian, S. F., Yang, J. J.} Soliton resolution for the Wadati--Konno--Ichikawa equation with weighted Sobolev initial data, {\it Annales Henri Poincar\'e,} {\bf 23} (2022), 2611--2655.}

\bibitem{LP} {\small \textsc{Linares, F., Ponce, G.}On decay and regularity properties of solutions to the Benjamin-Ono equation, {\it Mathematical Research Letters,} {\bf 33} (2026), 313--340.}

\bibitem{MX07} {\small \textsc{Miller, P. D., Xu, Z.} On the zero-dispersion limit of the Benjamin--Ono Cauchy problem for positive initial data, {\it Communications on Pure and Applied Mathematics,} {\bf 64} (2011), 205--270.}

\bibitem{Molinet} {\small \textsc{Molinet, L.} Global well-posedness in $L^2$ for the periodic Benjamin--Ono equation, {\it American Journal of Mathematics,} {\bf 130} (2008), 635--683.}

\bibitem{MP12} {\small \textsc{Molinet, L., Pilod, D.} The Cauchy problem for the Benjamin--Ono equation in $L^2$ revisited, {\it Analysis \& PDE,} {\bf 5} (2012), 365--395.}

\bibitem{MST01} {\small \textsc{Molinet, L., Saut, J.-C., Tzvetkov, N.} Ill-posedness issues for the Benjamin--Ono and related equations, {\it SIAM Journal on Mathematical Analysis,} {\bf 33} (2001), 982--988.}

\bibitem{MP} {\small \textsc{Mo\c{s}incat, R., Pilod, D.} Unconditional uniqueness for the Benjamin--Ono equation, {\it Pure and Applied Analysis,} {\bf 5} (2023), 285--322.}

\bibitem{N79} {\small \textsc{Nakamura, A.} B\"acklund transform and conservation laws of the Benjamin--Ono equation, {\it Journal of the Physical Society of Japan,} {\bf 47} (1979), 1335--1340.}

\bibitem{Ono} {\small \textsc{Ono, H.} Algebraic solitary waves in stratified fluids, {\it Journal of the Physical Society of Japan,} {\bf 39} (1975), 1082--1091.}

\bibitem{Osb10} {\small \textsc{Osborne, A. R.} Nonlinear ocean wave and the inverse scattering transform, {\it International Geophysics Series, Volume 97,} Academic Press, Burlington, (2010).}

\bibitem{Pau24} {\small \textsc{Paulsen, M. O.} Justification of the Benjamin--Ono equation as an internal water waves model, {\it Annals of PDE,} {\bf 10} (2024), Paper No.~25.}

\bibitem{Ponce} {\small \textsc{Ponce, G.} On the global well-posedness of the Benjamin--Ono equation, {\it Differential and Integral Equations,} {\bf 4} (1991), 527--542.}

\bibitem{ReedSimon} {\small \textsc{Reed, M., Simon, B.} Methods of Modern Mathematical Physics III: Scattering Theory, Academic Press, New York, (1979).}

\bibitem{Rosenblum} {\small \textsc{Rosenblum, M.} Perturbation of the continuous spectrum and unitary equivalence, {\it Pacific Journal of Mathematics,} {\bf 7} (1957), 997--1010.}

\bibitem{Saut} {\small \textsc{Saut, J.-C.} Sur quelques g\'en\'eralisations de l'\'equation de Korteweg--de Vries, {\it Journal de Math\'ematiques Pures et Appliqu\'ees,} {\bf 58} (1979), 21--61.}

\bibitem{Saut2} {\small \textsc{Saut, J.-C.} Benjamin--Ono and intermediate long wave equations: modeling, IST and PDE, {\it arXiv preprint arXiv:1811.08652,} (2018).}

\bibitem{Schilling} {\small \textsc{Schilling, R. L., Song, R., Vondra\v{c}ek, Z.} Bernstein Functions: Theory and Applications, second edition, de Gruyter, Berlin, (2012).}

\bibitem{SimonTrace} {\small \textsc{Simon, B.} Trace Ideals and Their Applications, second edition, American Mathematical Society, Providence, RI, (2005).}

\bibitem{Ste93} {\small \textsc{Stein, E. M., Murphy, T. S.} Harmonic analysis: real-variable methods, orthogonality, and oscillatory integrals, {\it Princeton Mathematical Series, Volume 43,} Princeton University Press, Princeton, (1993).}

\bibitem{Sun21} {\small \textsc{Sun, R.} Complete integrability of the Benjamin--Ono equation on the multi-soliton manifolds, {\it Communications in Mathematical Physics,} {\bf 383} (2021), 1051--1092.}

\bibitem{TaoBO} {\small \textsc{Tao, T.} Global well-posedness of the Benjamin--Ono equation in $H^1(\mathbb{R})$, {\it Journal of Hyperbolic Differential Equations,} {\bf 1} (2004), 27--49.}

\bibitem{T06} {\small \textsc{Tao, T.} Nonlinear Dispersive Equations: Local and Global Analysis, {\it American Mathematical Society,} (2006).}

\bibitem{TaoSolitons} {\small \textsc{Tao, T.} Why are solitons stable?, {\it Bulletin of the American Mathematical Society,} {\bf 46} (2009), 1--33.}

\bibitem{TT26} {\small \textsc{Tian, S. F., Tong, J. F.} On Cauchy problem for the spin-1 Gross--Pitaevskii equation: soliton resolution conjecture and asymptotic analysis, {\it Communications in Mathematical Physics,} {\bf 407} (2026), 52.}

\bibitem{Tz10} {\small \textsc{Tzvetkov, N.} Construction of a Gibbs measure associated to the periodic Benjamin--Ono equation, {\it Probability Theory and Related Fields,} {\bf 146} (2010), 481--514.}

\bibitem{Wu} {\small \textsc{Wu, Y.} Simplicity and finiteness of discrete spectrum of the Benjamin--Ono scattering operator, {\it SIAM Journal on Mathematical Analysis,} {\bf 48} (2016), 1348--1367.}

\bibitem{WuDirect} {\small \textsc{Wu, Y.} Jost solutions and the direct scattering problem of the Benjamin--Ono equation, {\it SIAM Journal on Mathematical Analysis,} {\bf 49} (2017), 5158--5206.}

\bibitem{Yafaev} {\small \textsc{Yafaev, D. R.} Mathematical Scattering Theory: General Theory, American Mathematical Society, Providence, RI, (1992).}

\bibitem{YTL26} {\small \textsc{Yang, J. J., Li, Z. Q., Tian, S. F.} The modified Camassa--Holm equation with nonzero background: Soliton resolution conjecture and asymptotic stability of $N$-soliton solutions, {\it Advances in Mathematics,} {\bf 484} (2026), 110552.}

\bibitem{ZabuskyKruskal} {\small \textsc{Zabusky, N. J., Kruskal, M. D.} Interaction of solitons in a collisionless plasma and the recurrence of initial states, {\it Physical Review Letters,} {\bf 15} (1965), 240--243.}

\end{thebibliography}
\end{document}